\documentclass[10pt, a4paper,
oneside, headinclude,footinclude]{scrartcl}
\usepackage[utf8]{inputenc}

\usepackage[
nochapters, 
pdfspacing, 
dottedtoc 
]{classicthesis} 
\usepackage{amsopn}
\usepackage[T1]{fontenc} 
\usepackage{multirow}
\usepackage[utf8]{inputenc} 
\usepackage{hhline}
\usepackage{graphicx} 
\graphicspath{{Figures/}} 
\usepackage{indentfirst}
\usepackage{enumitem} 
\usepackage{savesym}
\usepackage{mathrsfs}
\usepackage{geometry}
\usepackage{esint}
\usepackage{longtable}

\usepackage[
    backend=biber,
    style=numeric,
    natbib=false,
    url=false, 
    doi=false,
    eprint=false,
    maxnames=50
]{biblatex}

\usepackage{subfig} 

\usepackage{amsmath, amssymb,amsthm,amsfonts} 

\usepackage[english,capitalize]{cleveref}

\usepackage{thmtools}
\usepackage{thm-restate}

\usepackage{hyperref}
\newcommand{\vertiii}[1]{{\left\vert\kern-0.25ex\left\vert\kern-0.25ex\left\vert #1 
    \right\vert\kern-0.25ex\right\vert\kern-0.25ex\right\vert}}

\theoremstyle{plain}
\newtheorem{teorema}{Theorem}[section]
\newtheorem{proposizione}[teorema]{Proposition}
\newtheorem{lemma}[teorema]{Lemma}
\newtheorem{corollario}[teorema]{Corollary}
\newtheorem*{theorem*}{Theorem}

\theoremstyle{definition}
\newtheorem{definizione}{Definition}[section]

\theoremstyle{remark}
\newtheorem{osservazione}{Remark}[section]

\newcommand{\Tan}{\mathrm{Tan}}
\newcommand{\N}{\mathbb{N}}

\newcommand{\Q}{\mathbb{Q}}
\newcommand{\R}{\mathbb{R}}

\newcommand{\G}{Gr}
\newcommand{\res}

\newcounter{const}
\newcommand{\newC}{\refstepcounter{const}\ensuremath{C_{\theconst}}}
\newcommand{\oldC}[1]{\ensuremath{C_{\ref{#1}}}}

\newcounter{eps}
\newcommand{\newep}{\refstepcounter{eps}\ensuremath{\varepsilon_{\theeps}}}
\newcommand{\oldep}[1]{\ensuremath{\varepsilon_{\ref{#1}}}}

\DeclareMathOperator*{\lip}{Lip_1^+}

\DeclareMathOperator*{\supp}{supp}

\DeclareMathOperator*{\diam}{diam}

\DeclareMathOperator*{\dist}{dist}

\hypersetup{
colorlinks=true, breaklinks=true,bookmarksnumbered,
urlcolor=webbrown, linkcolor=RoyalBlue, citecolor=webgreen, 
pdftitle={}, 
pdfauthor={\textcopyright}, 
pdfsubject={}, 
pdfkeywords={}, 
pdfcreator={pdfLaTeX}, 
pdfproducer={LaTeX with hyperref and ClassicThesis} 
} 

\title{\normalfont\spacedallcaps{On rectifiable measures in Carnot groups: Marstrand--Mattila rectifiability criterion}} 
\author{Gioacchino Antonelli\textsuperscript{*} and Andrea Merlo\textsuperscript{**}}

\date{}

\begin{document}

\renewcommand{\sectionmark}[1]{\markright{\spacedlowsmallcaps{#1}}} 
\lehead{\mbox{\llap{\small\thepage\kern1em\color{halfgray} \vline}\color{halfgray}\hspace{0.5em}\rightmark\hfil}} 
\pagestyle{scrheadings}
\maketitle 
\setcounter{tocdepth}{2}
\paragraph*{Abstract}

In this paper we continue the study of the notion of $\mathscr{P}$-rectifiability in Carnot groups. We say that a Radon measure is $\mathscr{P}_h$-rectifiable, for $h\in\mathbb N$, if it has positive $h$-lower density and finite $h$-upper density almost everywhere, and, at almost every point, it admits a unique tangent measure up to multiples.

In this paper we prove a Marstrand--Mattila rectifiability criterion in arbitrary Carnot groups for $\mathscr{P}$-rectifiable measures with tangent planes that admit a normal complementary subgroup. Namely, in this co-normal case, even if a priori the tangent planes at a point might not be the same at different scales, a posteriori the measure has a unique tangent almost everywhere.

Since every horizontal subgroup of a Carnot group has a normal complement, our criterion applies in the particular case in which the tangents are one-dimensional horizontal subgroups. Hence, as an immediate consequence of our Marstrand--Mattila rectifiability criterion and a result of Chousionis--Magnani--Tyson, we obtain the one-dimensional Preiss's theorem in the first Heisenberg group $\mathbb H^1$. More precisely, we show that a Radon measure $\phi$ on $\mathbb H^1$ with positive and finite one-density with respect to the Koranyi distance is absolutely continuous with respect to the one-dimensional Hausdorff measure $\mathcal{H}^1$, and it is supported on a one-rectifiable set in the sense of Federer, i.e., it is supported on the countable union of the images of Lipschitz maps from $A\subseteq \mathbb R$ to $\mathbb H^1$.

{\let\thefootnote\relax\footnotetext{* \textit{Scuola Normale Superiore, Piazza dei Cavalieri, 7, 56126 Pisa, Italy,}}}
{\let\thefootnote\relax\footnotetext{** \textit{Univerité Paris-Saclay, 307 Rue Michel Magat Bâtiment, 91400 Orsay, France.}}}


\section{Introduction}

In Euclidean spaces a Radon measure $\phi$ is said to be $k$-\emph{rectifiable} if it is absolutely continuous with respect to the $k$-dimensional Hausdorff measure $\mathcal{H}^k$ and it is supported on a countable union of $k$-dimensional Lipschitz submanifolds, see \cite[\S 3.2.14]{Federer1996GeometricTheory}. In Euclidean spaces, the previous global definition arises as a consequence of the local structure of the measure, as it is clear from the following Proposition which is classically attributed to Marstrand and Mattila, see e.g., \cite[Theorem 16.7]{Mattila1995GeometrySpaces}.

\begin{proposizione}\label{prop:loctoglob}
Assume $\phi$ is a Radon measure on $\R^n$ and $k$ is a natural number such that $1\leq k\leq n$. Then, $\phi$ is a $k$-rectifiable measure if and only if  for $\phi$-almost every $x\in\mathbb{R}^n$ we have
\begin{itemize}
     \item[(i)]$0<\Theta^k_*(\phi,x)\leq\Theta^{k,*}(\phi,x)<+\infty$,
    \item[(ii)]$\mathrm{Tan}_k(\phi,x) \subseteq \{\lambda\mathcal{H}^k\llcorner V:\lambda>0,\,\,\text{and $V$ is a $k$-dimensional vector subspace}\}$,
\end{itemize}
where $\Theta^k_*(\phi,x)$ and $\Theta^{k,*}(\phi,x)$ are, respectively, the lower and the upper $k$-density of $\phi$ at $x$, see \cref{def:densities}, and $\mathrm{Tan}_k(\phi,x)$ is the set of $k$-tangent measures to $\phi$ at $x$, see \cref{def:TangentMeasure}, while $\mathcal{H}^k$ is the Hausdorff measure.
\end{proposizione}

The previous infinitesimal characterization of rectifiable measures in the Euclidean spaces is at the core of the definition of $\mathscr{P}$-rectifiable measures, which have been introduced by the second named author in \cite[Definition 3.1 \& Definition 3.2]{MarstrandMattila20}, in the setting of Carnot groups, and which have been studied by the two authors in \cite{antonelli2020rectifiable, antonelli2020rectifiable2}. We stress that the present paper is the second of two companion papers derived from \cite{antonelli2020rectifiable}. For a more thorough introduction we refer the reader to the introductions of \cite{antonelli2020rectifiable, antonelli2020rectifiable2}.

A Carnot group $\mathbb{G}$ is a simply connected nilpotent Lie group, whose Lie algebra is stratified and generated by its first layer. Carnot groups are a generalization of Euclidean spaces, and we remark that (quotients of) Carnot groups arise as the infinitesimal models of sub-Riemannian manifolds and their geometry, even at an infinitesimal scale, might be very different from the Euclidean one. From now on $\mathbb G$ will be a fixed Carnot group endowed with an arbitrary left-invariant homogeneous distance $d$, see \cref{sec:Prel}, and we recall that any two such distances are bi-Lipschitz equivalent. These groups have finite Hausdorff dimension, that is commonly denoted by $Q$, with respect to any homogeneous left-invariant distance. We recall that on $\mathbb G$ we have a natural family of dilations $\{\delta_{\lambda}\}_{\lambda>0}$, see \cref{sec:Prel}, which gives raise to a homogeneous structure on $\mathbb G$ with respect to which we can perform blow-ups.
We recall here the notion of $\mathscr{P}$-rectifiable measure.

\begin{definizione}[$\mathscr{P}$-rectifiable measures]\label{def:PhRectifiableMeasureINTRO}
Fix a natural number $1\leq h\leq Q$. A Radon measure $\phi$ on $\mathbb G$ is said to be {\em $\mathscr{P}_h$-rectifiable} if for $\phi$-almost every $x\in \mathbb{G}$ we have \begin{itemize}
    \item[(i)]$0<\Theta^h_*(\phi,x)\leq\Theta^{h,*}(\phi,x)<+\infty$,
    \item[(\hypertarget{due}{ii})]$\mathrm{Tan}_h(\phi,x) \subseteq \{\lambda\mathcal{S}^h\llcorner \mathbb V(x):\lambda\geq 0\}$, where $\mathbb V(x)$ is a homogeneous subgroup of $\mathbb G$ of Hausdorff dimension $h$,
\end{itemize}
where $\Theta^h_*(\phi,x)$ and $\Theta^{h,*}(\phi,x)$ are, respectively, the lower and the upper $h$-density of $\phi$ at $x$, see \cref{def:densities}, $\mathrm{Tan}_h(\phi,x)$ is the set of $h$-tangent measures to $\phi$ at $x$, see \cref{def:TangentMeasure}, and $\mathcal{S}^h$ is the spherical Hausdorff measure of dimension $h$, see \cref{def:HausdorffMEasure}. Furthermore, we say that $\phi$ is {\em $\mathscr{P}_h^*$-rectifiable} if (\hyperlink{due}{ii})  is replaced with the weaker
\begin{equation*}
    \mathrm{(ii)^*}\,\, \mathrm{Tan}_h(\phi,x) \subseteq \{\lambda \mathcal{S}^h\llcorner \mathbb V: ~\lambda \geq 0, \text{ $\mathbb V$ is a homogeneous } \text{subgroup}\text{  of $\mathbb G$ of Hausdorff dimension $h$}\}.
\end{equation*}
If we impose more regularity on the tangents we can define different subclasses of $\mathscr{P}$-rectifiable or $\mathscr{P}^*$-rectifiable measures, see \cref{def:SubclassesPh} for details. We notice that, a posteriori, in the aforementioned definitions we can and will restrict to $\lambda>0$, see \cref{rem:AboutLambda=0}.
\end{definizione}

\textit{Additional remark.} 
This is the second of two companion papers derived from \cite{antonelli2020rectifiable}. The present work consists of an elaboration of Sections 2 and 5 of the Preprint \cite{antonelli2020rectifiable}, while the first of the two (that will appear as \cite{antonelli2020rectifiableA}) is an elaboration of Sections 2, 3, 4, and 6 of \cite{antonelli2020rectifiable}. In order to avoid redundancy, we wrote the complete proofs of some of the ancillary results of Section 2 of \cite{antonelli2020rectifiable} only in the first of the two. Therefore, here we omit some of the proofs of the basic results, but we will always give precise references to the complete proofs that are contained in the Preprint on arXiv \cite{antonelli2020rectifiable}. We finally remark the results presented here are independent from those proved in Sections 3, 4, and 6 of \cite{antonelli2020rectifiable}. Therefore the present work is completeley indepent from the other.

\subsection{Results}\label{mainresultsintro}
The main contribution of this paper is a co-normal Marstrand-Mattila rectifiability criterion for $\mathscr{P}$-rectifiable measures in the setting of Carnot groups. Already in the Euclidean case, it is not trivial to prove that a $\mathscr{P}^*$-rectifiable measure is rectifiable, see \cite[Theorem 5.1]{DeLellis2008RectifiableMeasures}, and \cite[Theorem 16.7]{Mattila1995GeometrySpaces}. Proving that a $\mathscr{P}^*_{Q-1}$-rectifiable measure in a Carnot group of Hausdorff dimension $Q$ is supported on the countable union of $C^1_{\mathrm H}$-regular hypersurfaces is a challenging problem that has been solved by the second-named author in \cite[Theorem 3]{MarstrandMattila20}. 

In this paper we adapt Preiss's techniques in \cite{Preiss1987GeometryDensities} to prove that $\mathscr{P}^*$-rectifiable measures with \textbf{co-normal} tangents, i.e., with \textbf{tangents that admit at least one normal complementary subgroup}, are $\mathscr{P}$-rectifiable, see \cref{thm:MMconormale}. This means that in this co-normal case, even if the tangents at a point might not be the same at different scales, then a posteriori the tangent is unique almost everywhere. We recall that when we say that a homogeneous subgroup $\mathbb V$ of a Carnot group $\mathbb G$ {\em admits a complementary subgroup}, we mean that there exists a homogeneous subgroup $\mathbb L$ such that $\mathbb G=\mathbb V\cdot \mathbb L$ and $\mathbb V\cap\mathbb L=\{0\}$. 
\begin{teorema}[Co-normal Marstrand--Mattila rectifiability criterion]\label{thm:MMconormaleIntro}
Let $\mathbb G$ be a Carnot group endowed with an arbitrary left-invariant homogeneous distance. Let $\phi$ be a $\mathscr{P}_h^*$-rectifiable measure on $\mathbb G$ \textbf{with tangents that admit at least one normal complementary subgroup}. Then $\phi$ is a $\mathscr{P}_h$-rectifiable measure. 
Moreover, there are countably many homogeneous subgroups $\mathbb{V}_i$ of homogeneous dimension $h$, and Lipschitz maps $\Phi_i:A_i\subseteq \mathbb V_i\to \mathbb{G}$, where $A_i$'s are compact, such that
$$\phi\Big(\mathbb{G}\setminus  \bigcup_{i\in\N} \Phi_i(A_i)\Big)=0.$$
\end{teorema}

Let us notice that the converse of \cref{thm:MMconormaleIntro} holds as well. Namely if $\phi$ is a Radon measure on $\mathbb G$ with positive $h$-lower density and finite $h$-upper density $\phi$-almost everywhere, and there are countably many homogeneous subgroups $\mathbb{V}_i$ of homogeneous dimension $h$, and Lipschitz maps $\Phi_i:A_i\subseteq \mathbb V_i\to \mathbb{G}$, where $A_i$'s are compact, such that
$$\phi\Big(\mathbb{G}\setminus  \bigcup_{i\in\N} \Phi_i(A_i)\Big)=0,$$
hence $\phi$ is $\mathscr{P}_h$-rectifiable (and a fortiori $\mathscr{P}_h^*$-rectifiable). The proof is done first by a classical reduction to measures of the type $\mathcal{S}^h\llcorner\Gamma$, and hence using the Rademacher theorem, and the area formula, which hold for the maps $\Phi_i$. The resulting reasoning is exactly the same as in the last part of the proof of \cref{thm:MMconormale}.

As we said above, a Marstrand--Mattila rectifiability criterion for codimension-one rectifiable measures in arbitrary Carnot groups has been proved by the second named author in \cite{MarstrandMattila20}. The techniques used in \cite{MarstrandMattila20} are likely to be adapted to show the same result in the more general co-horizontal case. Apart from the cases discussed here, and the result in \cref{thm:MMconormaleIntro}, we presently do not know if a Marstrand--Mattila rectifiability criterion holds in the generality of $\mathscr{P}^*$-rectifiable measure with complemented tangents. We believe that such a result could be really challenging because of the lack of regularity of the projections map in the general case.  

We remark that we are able to prove \cref{thm:MMconormaleIntro} because of the following two key observations: whenever $\mathbb V$ admits a normal complementary subgroup $\mathbb L$, then the projection $P_{\mathbb V}:\mathbb G\to\mathbb V$ related to the splitting $\mathbb G=\mathbb V\cdot \mathbb L$ is a Lipschitz homogeneous homomorphism, see \cref{prop:projhom}, and moreover $\mathbb V$ is a Carnot subgroup, see \cite[Remark 2.1]{AM20}. This allows us to adapt Preiss's machinery in \cite{Preiss1987GeometryDensities} not without some difficulties that are essentially due to the fact that, on the contrary with respect to the Euclidean setting, we do not have a canonical choice of a normal complementary subgroup to $\mathbb V$ when there is at least one. 
We also stress that, for the Marstrand-Mattila rectifiability criterion, the assumption on the strictly positive lower density is necessary already in the Euclidean case, see \cite[5.9]{Preiss1987GeometryDensities}.

The hypotheses of \cref{thm:MMconormaleIntro} are satisfied whenever we have a $\mathscr{P}_h^*$-rectifiable measure with horizontal tangents. Thus, the previous co-normal Marstrand-Mattila rectifiability criterion, jointly with the result of \cite[Theorem 1.3]{ChousionisONGROUP}, can be used to give the proof of Preiss's theorem for measures with one-density in the Heisenberg group $\mathbb H^1$ endowed with the Koranyi norm. For the sake of clarity, let us recall that if we identify $\mathbb H^1\equiv \mathbb R^3 =\{(x,t):x\in\mathbb R^2,t\in\mathbb R\}$ through exponential coordinates, then the {\em Koranyi norm} is $\|(x,t)\|:=(\|x\|_{\mathrm{eu}}^4+t^2)^{1/4}$, where $\|\cdot\|_{\mathrm{eu}}$ is the standard Euclidean norm.

\begin{teorema}[One-dimensional Preiss's theorem in $\mathbb H^1$]\label{thm:Preiss1Intro}
Let $\mathbb H^1$ be the first Heisenberg group endowed with the Koranyi norm. Let $\phi$ be a Radon measure on $\mathbb H^1$ such that the one-density $\Theta^1(\phi,x)$ exists positive and finite for $\phi$-almost every $x\in\mathbb H^1$. Then $\phi$-almost all of $\mathbb H^1$ can be covered with countably many images $\Phi_i(A_i)$ of Lipschitz functions $\Phi_i:A_i\subseteq \mathbb R\to\mathbb H^1$, and moreover $\phi$ is absolutely continuous with respect to the one-dimensional Hausdorff measure $\mathcal{H}^1$. 
\end{teorema}
\begin{proof}
From the fact that the one-density exists at $\phi$-almost every $x\in\mathbb H^1$ we deduce that at $\phi$-almost every $x\in\mathbb H^1$ the tangent measures are uniform measures, see \cite[Proposition 2.2]{Merlo1}. Then from \cite[Theorem 1.3]{ChousionisONGROUP} we get that the tangent measures, at $\phi$-almost every $x\in\mathbb H^1$, are $\mathcal{S}^1\llcorner L$, where $L$ is a horizontal line. Finally from \cref{thm:MMconormaleIntro}, since every horizontal line admits a normal complementary subgroup, we get the first part of the sought conclusion. The absolute continuity is a consequence of \cref{prop:MutuallyEthetaGamma}.
\end{proof}

Let us notice that \cref{thm:Preiss1Intro} is one of the few cases in which Preiss's theorem \cite{Preiss1987GeometryDensities} is nowadays known to hold beyond the Euclidean space. The characterization of the $k$-rectifiability of a measure through the existence of the $k$-density in Euclidean spaces was one of the great achievement of Geometric Measure Theory \cite{Preiss1987GeometryDensities}. Another Preiss's type result has been proved by A. Lorent \cite{Lorent} in $\ell_\infty^3$. Recently, the second named author has accomplished to prove the analogue of \cref{thm:Preiss1Intro} for the 3-density, which requires a deeper understanding of 3-uniform measures in the first Heisenberg group $\mathbb H^1$, see \cite{Merlo1, MarstrandMattila20}.

A related result to \cref{thm:Preiss1Intro} in the broad generality of metric spaces is contained in \cite{PreissTiser}. Nevertheless we stress that here we prove \cref{thm:Preiss1Intro} in the general setting of Radon measures and we ask no bound on the density, just its existence: namely, we prove that whenever the 1-density of a Radon measure exists on $\mathbb H^1$ endowed with the Koranyi norm, hence we have that it is rectifiable. We remark that, even if we take advantage of the fact that the classification of the 1-uniform measures on $\mathbb H^1$ was known from \cite{ChousionisONGROUP}, the result in \cref{thm:Preiss1Intro} is non-trivial, since it requires the Marstrand--Mattila rectifiability criterion in \cref{thm:MMconormaleIntro}.

Let us remark that the previous \cref{thm:Preiss1Intro} is the last step needed to completely solve in $\mathbb H^1$ the implication (i)$\Rightarrow$(ii) of the density problem formulated in \cite[page 50]{MarstrandMattila20}. Let us explain this and give a scheme here. If in $\mathbb H^1$ endowed with the Koranyi norm we have a Radon measure $\phi$ such that there exists $\alpha\geq 0$ for which the $\alpha$-density $\Theta^\alpha(\phi,x)$ exists positive and finite for $\phi$-almost every $x\in\mathbb H^1$ we first get that $\alpha$ is an integer, see \cite[Theorem 1.1]{Chousionis2015MarstrandsGroup}. Thus the only non-trivial cases are
\begin{itemize}
    \item $\alpha=1$. In this case $\phi$ is $\mathscr{P}_1$-rectifiable, see \cref{thm:MMconormaleIntro}. Moreover we can cover $\phi$-almost all of $\mathbb H^1$ with countably many images of Lipschitz maps from subsets of $\mathbb R$ to $\mathbb H^1$. Note that we can improve the latter conclusion. Indeed, we can cover $\phi$-almost all of $\mathbb H^1$ with countably many images of $C^1_{\mathrm H}$-functions defined on \textbf{open} subsets of $\mathbb R$ to $\mathbb H^1$. This last improvement comes from Pansu-Rademacher theorem for Lipschitz maps between Carnot groups, see \cite{Pansu}, and the Whitney exstension theorem proved in \cite[Theorem 6.5]{JuilletSigalottiPliability}. 
    \item $\alpha=2$. In this case $\phi$ is $\mathscr{P}_2$-rectifiable, see \cite[Theorem 3.7]{MarstrandMattila20}. This means that the tangent measure is $\phi$-almost everywhere unique and it is a Haar measure of the vertical line in $\mathbb H^1$.
    \item $\alpha=3$. In this case $\phi$ is $\mathscr{P}_3$-rectifiable, see \cite{Merlo1}, and \cite[Theorem 4]{MarstrandMattila20}. Moreover we can cover $\phi$-almost all of $\mathbb H^1$ with countably many $C^1_{\mathrm H}$-hypersurfaces, see \cite[Theorem 4]{MarstrandMattila20}.
\end{itemize}
As it is clear from the list above, an interesting line of investigation could be a finer study of the structure of $\mathscr{P}_2$-rectifiable measures in $\mathbb H^1$.

\vspace{0.3cm}
\textbf{Acknowledgments}: The first author is partially supported by the European Research Council (ERC Starting Grant 713998 GeoMeG `\emph{Geometry of Metric Groups}'). The second author is supported by the Simons Foundation Wave Project,  grant 601941, GD.

\section{Preliminaries}\label{sec:Prel}
\subsection{Carnot Groups}\label{sub:Carnot}
In this subsection we briefly introduce some notations on Carnot groups that we will extensively use throughout the paper. For a detailed account on Carnot groups we refer to \cite{LD17}.

A Carnot group $\mathbb{G}$ of step $\kappa$ \label{num:step} is a simply connected Lie group whose Lie algebra $\mathfrak g$ admits a stratification $\mathfrak g=V_1\, \oplus \, V_2 \, \oplus \dots \oplus \, V_\kappa$. We say that $V_1\, \oplus \, V_2 \, \oplus \dots \oplus \, V_\kappa$ is a {\em stratification} of $\mathfrak g$ if $\mathfrak g = V_1\, \oplus \, V_2 \, \oplus \dots \oplus \, V_\kappa$,
$$
[V_1,V_i]=V_{i+1}, \quad \text{for any $i=1,\dots,\kappa-1$}, \quad  \quad [V_1,V_\kappa]=\{0\}, \quad \text{and}\,\, V_\kappa\neq\{0\},
$$ 
where $[A,B]:=\mathrm{span}\{[a,b]:a\in A,b\in B\}$. We call $V_1$ the \emph{horizontal layer} of $\mathbb G$. We denote by $n$ the topological dimension of $\mathfrak g$, by $n_j$ the dimension of $V_j$ for every $j=1,\dots,\kappa$.
Furthermore, we define $\pi_i:\mathbb{G}\to V_i$ to be the projection maps on the $i$-th strata. 
We will often shorten the notation to $v_i:=\pi_iv$.

For a Carnot group $\mathbb G$, the exponential map $\exp :\mathfrak g \to \mathbb{G}$ is a global diffeomorphism from $\mathfrak g$ to $\mathbb{G}$.
Hence, if we choose a basis $\{X_1,\dots , X_n\}$ of $\mathfrak g$,  any $p\in \mathbb{G}$ can be written in a unique way as $p=\exp (p_1X_1+\dots +p_nX_n)$. This means that we can identify $p\in \mathbb{G}$ with the $n$-tuple $(p_1,\dots , p_n)\in \R^n$ and the group $\mathbb{G}$ itself with $\R^n$ endowed with the group operation $\cdot$ determined by the Baker-Campbell-Hausdorff formula. From now on, we will always assume that $\mathbb{G}=(\R^n,\cdot)$ and, as a consequence, that the exponential map $\exp$ acts as the identity.

For any $p\in \mathbb{G}$, we define the left translation $\tau _p:\mathbb{G} \to \mathbb{G}$ as
\begin{equation*}
q \mapsto \tau _p q := p\cdot q.
\end{equation*}
As already remarked above, the group operation $\cdot$ is determined by the Campbell-Hausdorff formula, and it has the form (see \cite[Proposition 2.1]{step2})
\begin{equation*}
p\cdot q= p+q+\mathscr{Q}(p,q), \quad \mbox{for all }\, p,q \in  \R^n,
\end{equation*} 
where $\mathscr{Q}=(\mathscr{Q}_1,\dots , \mathscr{Q}_\kappa):\R^n\times \R^n \to V_1\oplus\ldots\oplus V_\kappa$, and the $\mathscr{Q}_i$'s have the following properties. For any $i=1,\ldots \kappa$ and any $p,q\in \mathbb{G}$ we have\label{tran}
\begin{itemize}
    \item[(i)]$\mathscr{Q}_i(\delta_\lambda p,\delta_\lambda q)=\lambda^i\mathscr{Q}_i(p,q)$,
    \item[(ii)] $\mathscr{Q}_i(p,q)=-\mathscr{Q}_i(-q,-p)$,
    \item[(iii)] $\mathscr{Q}_1=0$ and $\mathscr{Q}_i(p,q)=\mathscr{Q}_i(p_1,\ldots,p_{i-1},q_1,\ldots,q_{i-1})$.
\end{itemize}
Thus, we can represent the product $\cdot$ as
\begin{equation}\label{opgr}
p\cdot q= (p_1+q_1,p_2+q_2+\mathscr{Q}_2(p_1,q_1),\dots ,p_\kappa +q_\kappa+\mathscr{Q}_\kappa (p_1,\dots , p_{\kappa-1} ,q_1,\dots ,q_{\kappa-1})). 
\end{equation}

The stratificaton of $\mathfrak{g}$ carries with it a natural family of dilations $\delta_\lambda :\mathfrak{g}\to \mathfrak{g}$, that are Lie algebra automorphisms of $\mathfrak{g}$ and are defined by\label{intrdil}
\begin{equation}
     \delta_\lambda (v_1,\dots , v_\kappa):=(\lambda v_1,\lambda^2 v_2,\dots , \lambda^\kappa v_\kappa), \quad \text{for any $\lambda>0$},
     \nonumber
\end{equation}
where $v_i\in V_i$. The stratification of the Lie algebra $\mathfrak{g}$  naturally induces a gradation on each of its homogeneous Lie sub-algebras $\mathfrak{h}$, i.e., sub-algebras that are $\delta_{\lambda}$-invariant for any $\lambda>0$, that is
\begin{equation}
    \mathfrak{h}=V_1\cap \mathfrak{h}\oplus\ldots\oplus V_\kappa\cap \mathfrak{h}.
    \label{eq:intr1}
\end{equation}
We say that $\mathfrak h=W_1\oplus\dots\oplus W_{\kappa}$ is a {\em gradation} of $\mathfrak h$ if $[W_i,W_j]\subseteq W_{i+j}$ for every $1\leq i,j\leq \kappa$, where we mean that $W_\ell:=\{0\}$ for every $\ell > \kappa$.
Since the exponential map acts as the identity, the Lie algebra automorphisms $\delta_\lambda$ can be read also as group automorphisms of $\mathbb{G}$.

\begin{definizione}[Homogeneous subgroups]\label{homsub}
A subgroup $\mathbb V$ of $\mathbb{G}$ is said to be \emph{homogeneous} if it is a Lie subgroup of $\mathbb{G}$ that is invariant under the dilations $\delta_\lambda$.
\end{definizione}

We recall the following basic terminology: a {\em horizontal subgroup} of a Carnot group $\mathbb G$ is a homogeneous subgroup of it that is contained in $\exp(V_1)$; a {\em Carnot subgroup} $\mathbb W=\exp(\mathfrak h)$ of a Carnot group $\mathbb G$ is a homogeneous subgroup of it such that the first layer $V_1\cap\mathfrak h$ of the grading of $\mathfrak h$ inherited from the stratification of $\mathfrak g$ is the first layer of a stratification of $\mathfrak h$.

Homogeneous Lie subgroups of $\mathbb{G}$ are in bijective correspondence through $\exp$ with the Lie sub-algebras of $\mathfrak{g}$ that are invariant under the dilations $\delta_\lambda$. 
For any Lie algebra $\mathfrak{h}$ with gradation $\mathfrak h= W_1\oplus\ldots\oplus W_{\kappa}$, we define its \emph{homogeneous dimension} as
$$\text{dim}_{\mathrm{hom}}(\mathfrak{h}):=\sum_{i=1}^{\kappa} i\cdot\text{dim}(W_i).$$
Thanks to \eqref{eq:intr1} we infer that, if $\mathfrak{h}$ is a homogeneous Lie sub-algebra of $\mathfrak{g}$, we have $\text{dim}_{\mathrm{hom}}(\mathfrak{h}):=\sum_{i=1}^{\kappa} i\cdot\text{dim}(\mathfrak{h}\cap V_i)$. We introduce now the class of homogeneous and left-invariant distances.

\begin{definizione}[Homogeneous left-invariant distance]
A metric $d:\mathbb{G}\times \mathbb{G}\to \R$ is said to be homogeneous and left invariant if for any $x,y\in \mathbb{G}$ we have
\begin{itemize}
    \item[(i)] $d(\delta_\lambda x,\delta_\lambda y)=\lambda d(x,y)$ for any $\lambda>0$,
    \item[(ii)] $d(\tau_z x,\tau_z y)=d(x,y)$ for any $z\in \mathbb{G}$.
\end{itemize}
\end{definizione}

We remark that two homogeneous left-invariant distances on a Carnot group are always bi-Lipschitz equivalent.
It is well-known that the Hausdorff dimension (for a definition of Hausdorff dimension see for instance \cite[Definition 4.8]{Mattila1995GeometrySpaces}) of a graded Lie group $\mathbb G$ with respect to an arbitrary left-invariant homogeneous distance coincides with the homogeneous dimension of its Lie algebra. For a reference for the latter statement, see \cite[Theorem 4.4]{LDNG19}. \textbf{From now on, if not otherwise stated, $\mathbb G$ will be a fixed Carnot group}. We recall that a \emph{homogeneous norm} $\|\cdot\|$ on $\mathbb G$ is a function $\|\cdot\|:\mathbb G\to [0,+\infty)$ such that $\|\delta_\lambda x\|=\lambda\|x\|$ for every $\lambda>0$ and $x\in\mathbb G$; $\|x\cdot y\|\leq \|x\|+\|y\|$ for every $x,y\in\mathbb G$; and $\|x\|=0$ if and only if $x=0$. We introduce now a distinguished homogeneous norm on $\mathbb G$.

\begin{definizione}[Smooth-box metric]\label{smoothnorm}
For any $g\in \mathbb{G}$, we let
$$\lVert g\rVert:=\max\{\varepsilon_1\lvert g_1\rvert,\varepsilon_2\lvert g_2\rvert^{1/2},\ldots, \varepsilon_{\kappa}\lvert g_\kappa\rvert^{1/{\kappa}}\},$$
where $\varepsilon_1=1$ and $\varepsilon_2,\ldots \varepsilon_{\kappa}$ are suitably small parameters depending only on the group $\mathbb{G}$. For the proof of the fact that we can choose the $\varepsilon_i$'s in such a way that $\lVert\cdot\rVert$ is a homogeneous norm on $\mathbb{G}$ that induces a left-invariant homogeneous distance we refer to \cite[Section 5]{step2}. 
\end{definizione}
Given an arbitrary homogeneous norm $\|\cdot\|$ on $\mathbb G$, the distance $d$ induced by $\|\cdot\|$ is defined as follows
$$
d(x,y):=\lVert x^{-1}\cdot y\rVert.
$$
Vice-versa, given a homogeneous left-invariant distance $d$, it induces a homogeneous norm through the equality $\|x\|:=d(x,e)$ for every $x\in\mathbb G$, where $e$ is the identity element of $\mathbb G$.

Given a homogeneous left-invariant distance $d$ we let $U(x,r):=\{z\in \mathbb{G}:d(x,z)<r\}$ be the open metric ball relative to the distance $d$ centred at $x$ and with radius $r>0$. The closed ball will be denoted with $B(x,r):=\{z\in \mathbb{G}:d(x,z)\leq r\}$. Moreover, for a subset $E\subseteq \mathbb G$ and $r>0$, we denote with $B(E,r):=\{z\in\mathbb G:\dist(z,E)\leq r\}$ the {\em closed $r$-tubular neighborhood of $E$}  and with $U(E,r):=\{z\in\mathbb G:\dist(z,E)< r\}$ the {\em open $r$-tubular neighborhood of $E$}.

\begin{definizione}[Hausdorff Measures]\label{def:HausdorffMEasure}
Throughout the paper we define the $h$-dimensional {\em spherical Hausdorff measure} relative to a left invariant homogeneous metric $d$ as\label{sphericaldhausmeas}
$$
\mathcal{S}^{h}(A):=\sup_{\delta>0}\inf\bigg\{\sum_{j=1}^\infty  r_j^h:A\subseteq \bigcup_{j=1}^\infty B(x_j,r_j),~r_j\leq\delta\bigg\},
$$
for every $A\subseteq \mathbb G$.
We define the $h$-dimensional {\em Hausdorff measure}\label{hausmeas} relative to $d$ as
$$
\mathcal{H}^h(A):=\sup_{\delta>0}\inf \left\{\sum_{j=1}^{\infty} 2^{-h}(\diam E_j)^h:A \subseteq \bigcup_{j=1}^{\infty} E_j,\, \diam E_j\leq \delta\right\},
$$
for every $A\subseteq \mathbb G$.
We define the $h$-dimensional {\em centered Hausdorff measure} relative to $d$ as\label{centredhausmeas}
$$
\mathcal{C}^{h}(A):=\underset{E\subseteq A}{\sup}\,\,\mathcal{C}_0^h(E),
$$
for every $A\subseteq \mathbb G$, 
where
$$
\mathcal{C}^{h}_0(E):=\sup_{\delta>0}\inf\bigg\{\sum_{j=1}^\infty  r_j^h:E\subseteq \bigcup_{j=1}^\infty B(x_j,r_j),~ x_j\in E,~r_j\leq\delta\bigg\},
$$
for every $E\subseteq \mathbb G$.
We stress that $\mathcal{C}^h$ is an outer measure, and thus it defines a Borel regular measure, see \cite[Proposition 4.1]{EdgarCentered}, and that the measures $\mathcal{S}^h,\mathcal{H}^h,\mathcal{C}^h$ are all equivalent measures, see \cite[Section 2.10.2]{Federer1996GeometricTheory} and \cite[Proposition 4.2]{EdgarCentered}.
\end{definizione}

\begin{definizione}[Hausdorff distance]\label{def:Haus}
Given a left-invariant homogeneous distance $d$ on $\mathbb G$, for any couple of sets $A,B\subseteq \mathbb{G}$, we define the \emph{Hausdorff distance} of $A$ from $B$ as
$$d_{H,\mathbb G}(A,B):=\max\Big\{\sup_{x\in A}\text{dist}(x,B),\sup_{y\in B}\text{dist}(A,y)\Big\},$$
where 
$$
\text{dist}(x,A):=\inf_{y\in A} d(x,y),
$$
for every $x\in\mathbb G$ and $A\subseteq \mathbb G$.
\end{definizione}

\subsection{Densities and tangents of Radon measures}

Throughout the rest of  the paper we will always assume that $\mathbb{G}$ is a fixed Carnot group endowed with an arbitrary left-invariant homogeneous distance $d$. Some of the forthcoming results will be proved in the particular case in which $d$ is the distance induced by the distinguished homogeneous norm defined in \cref{smoothnorm}, and we will stress this when it will be the case. 

The homogeneous, and thus Hausdorff, dimension with respect to $d$ will be denoted with $Q$. Furthermore as discussed in the previous subsection, we will assume without loss of generality that $\mathbb{G}$ coincides with $\R^n$ endowed with the product induced by the Baker-Campbell-Hausdorff formula relative to $\text{Lie}(\mathbb{G})$.

\begin{definizione}[Weak convergence of measures]\label{def:WeakConvergence}
Given a family $\{\phi_i\}_{i\in\N}$ of Radon measures on $\mathbb{G}$ we say that $\phi_i$ weakly converges to a Radon measure $\phi$, and we write $\phi_i\rightharpoonup \phi$, if
$$
\int fd \phi_i \to \int fd\phi, \qquad\text{for any } f\in C_c(\mathbb G).
$$
\end{definizione}

\begin{definizione}[Tangent measures]\label{def:TangentMeasure}
Let $\phi$ be a Radon measure on $\mathbb G$. For any $x\in\mathbb G$ and any $r>0$ we define the measure
$$
T_{x,r}\phi(E):=\phi(x\cdot\delta_r(E)), \qquad\text{for any Borel set }E.
$$
Furthermore, we define $\mathrm{Tan}_{h}(\phi,x)$, the $h$-dimensional tangents to $\phi$ at $x$, to be the collection of the Radon measures $\nu$ for which there is an infinitesimal sequence $\{r_i\}_{i\in\N}$ such that
$$r_i^{-h}T_{x,r_i}\phi\rightharpoonup \nu.$$
\end{definizione}
\begin{osservazione}(Zero as a tangent measure)\label{rem:TangentZero}
We remark that our definition potentially admits the zero measure as a tangent measure, as in \cite{DeLellis2008RectifiableMeasures}, while the definitions in \cite{Preiss1987GeometryDensities} and \cite{MatSerSC} do not.
\end{osservazione}

\begin{definizione}[Lower and upper densities]\label{def:densities}
If $\phi$ is a Radon measure on $\mathbb{G}$, and $h>0$, we define
$$
\Theta_*^{h}(\phi,x):=\liminf_{r\to 0} \frac{\phi(B(x,r))}{r^{h}},\qquad \text{and}\qquad \Theta^{h,*}(\phi,x):=\limsup_{r\to 0} \frac{\phi(B(x,r))}{r^{h}},
$$
and we say that $\Theta_*^{h}(\phi,x)$ and $\Theta^{h,*}(\phi,x)$ are the lower and upper $h$-density of $\phi$ at the point $x\in\mathbb{G}$, respectively. Furthermore, we say that measure $\phi$ has $h$-density if
$$
0<\Theta^h_*(\phi,x)=\Theta^{h,*}(\phi,x)<\infty,\qquad \text{for }\phi\text{-almost any }x\in\mathbb{G}.
$$
\end{definizione}

Lebesgue theorem holds for measures with positive lower density and finite upper density, and thus local properties are stable under restriction to Borel subsets.

\begin{proposizione}\label{prop:Lebesuge}
Suppose $\phi$ is a Radon measure on $\mathbb{G}$ with $0<\Theta^h_*(\phi,x)\leq \Theta^{h,*}(\phi,x)<\infty $ for $\phi$-almost every $x\in \mathbb{G}$. Then, for any Borel set $B\subseteq \mathbb{G}$ and for $\phi$-almost every $x\in B$ we have
$$\Theta^h_*(\phi\llcorner B,x)=\Theta^h_*(\phi,x),\qquad \text{and}\qquad\Theta^{h,*}(\phi\llcorner B,x)=\Theta^{h,*}(\phi,x).$$
\end{proposizione}

\begin{proof}
This is a direct consequence of Lebesgue differentiation Theorem of \cite[page 77]{HeinonenKoskelaShanmugalingam}, that can be applied since $(\mathbb G,d,\phi)$ is a Vitali metric measure space due to \cite[Theorem 3.4.3]{HeinonenKoskelaShanmugalingam}.
\end{proof}

We stress that whenever the $h$-lower density of $\phi$ is stricly positve and the $h$-upper density of $\phi$ is finite $\phi$-almost everywhere, the set $\mathrm{Tan}_h(\phi,x)$ is nonempty for $\phi$-almost every $x\in\mathbb G$, see \cite[Proposition 1.12]{MarstrandMattila20}. The following proposition has been proved in \cite[Proposition 1.13]{MarstrandMattila20}.

\begin{proposizione}[Locality of tangents]\label{prop:LocalityOfTangent}
Let $h>0$, and let $\phi$ be a Radon measure such that for $\phi$-almost every $x\in\mathbb G$ we have
$$
0<\Theta^h_*(\phi,x)\leq \Theta^{h,*}(\phi,x)<\infty.
$$
Then for every $\rho\in L^1(\phi)$ that is nonnegative $\phi$-almost everywhere we have
$\mathrm{Tan}_h(\rho\phi,x)=\rho(x)\mathrm{Tan}_h(\phi,x)$ for $\phi$-almost every $x\in\mathbb G$. More precisely the following holds: for $\phi$-almost every $x\in\mathbb G$ then
\begin{equation}
    \begin{split}
        \text{if $r_i\to 0$ is such that}\quad r_i^{-h}T_{x,r_i}\phi \rightharpoonup \nu\quad \text{then}\quad r_i^{-h}T_{x,r_i}(\rho\phi)\rightharpoonup \rho(x)\nu.  \\
    \end{split}
\end{equation}
\end{proposizione}

Let us introduce a useful split of the support of a Radon measure $\phi$ on $\mathbb G$. 
\begin{definizione}\label{def:EThetaGamma}
Let $\phi$ be a Radon measure on $\mathbb G$ that is supported on the compact set $K$. For any $\vartheta,\gamma\in\N$ we define
\begin{equation}
    E(\vartheta,\gamma):=\big\{x\in K:\vartheta^{-1}r^h\leq \phi(B(x,r))\leq \vartheta r^h\text{ for any }0<r<1/\gamma\big\}.
    \label{eq:A1}
\end{equation}
\end{definizione}

\begin{proposizione}\label{prop:cpt}
For any $\vartheta,\gamma\in\N$, the set $E(\vartheta,\gamma)$ defined in \cref{def:EThetaGamma} is compact.
\end{proposizione}

\begin{proof}
This is \cite[Proposition 1.14]{MarstrandMattila20}.
\end{proof}

\begin{proposizione}\label{prop::E}
Assume $\phi$ is a Radon measure supported on the compact set $K$ such that $
0<\Theta^h_*(\phi,x)\leq \Theta^{h,*}(\phi,x)<\infty
$ for $\phi$-almost every $x\in\mathbb G$. Then $\phi(\mathbb{G}\setminus \bigcup_{\vartheta,\gamma\in\N} E(\vartheta,\gamma))=0$.
\end{proposizione}

\begin{proof}
Let $w\in K\setminus \bigcup_{\vartheta,\gamma} E(\vartheta,\gamma)$ and note that this implies that either $\Theta^h_*(\phi,x)=0$ or $\Theta^{h,*}(\phi,x)=\infty$. Since $
0<\Theta^h_*(\phi,x)\leq \Theta^{h,*}(\phi,x)<\infty
$ for $\phi$-almost every $x\in\mathbb G$, this concludes the proof.
\end{proof}

We recall here a useful proposition about the structure of Radon measures.
\begin{proposizione}[{\cite[Proposition 1.17 and Corollary 1.18]{MarstrandMattila20}}]\label{prop:MutuallyEthetaGamma}
Let $\phi$ be a Radon measure supported on a compact set on $\mathbb{G}$ such that $0<\Theta^h_*(\phi,x)\leq \Theta^{h,*}(\phi,x)<\infty $ for $\phi$-almost every $x\in \mathbb{G}$. For every $\vartheta,\gamma\in \mathbb N$ we have that $\phi\llcorner E(\vartheta,\gamma)$ is mutually absolutely continuous with respect to $\mathcal{S}^h\llcorner E(\vartheta,\gamma)$.
\end{proposizione}

\subsection{Intrinsic Grassmannian in Carnot groups}

We give the definition of the intrinsic Grassmannian on Carnot groups and introduce the classes of complemented and co-normal homogeneous subgroups.

\begin{definizione}[Intrinsic Grassmanian on Carnot groups]\label{def:Grassmannian}
For any $1\leq h\leq Q$, we define $\G(h)$ to be the family of homogeneous subgroups $\mathbb V$ of $\mathbb{G}$ that have Hausdorff dimension $h$. 

Let us recall that if $\mathbb{V}$ is a homogeneous subgroup of $\mathbb{G}$, any other homogeneous subgroup such that
$$
\mathbb{V}\cdot\mathbb{L}=\mathbb{G}\qquad \text{and}\qquad \mathbb{V}\cap \mathbb{L}=\{0\}.
$$
is said to be a \emph{complement} of $\mathbb{G}$. Finally, we let
\begin{itemize}
    \item[(i)] $\G_c(h)$ to be the subfamily of those $\mathbb V\in\G(h)$ that have a complement and we will refer to $\G_c(h)$ as the $h$-dimensional \emph{complemented} Grassmanian,
    \item[(ii)] $\G_\unlhd(h)$ the subfamily of those $\mathbb V\in\G_c(h)$ having a \textbf{normal} complement and we will refer to $\G_\unlhd(h)$ as the $h$-dimensional \emph{co-normal} Grassmanian.
\end{itemize}
\end{definizione}
Let us introduce the stratification vector of a homogeneous subgroup.
\begin{definizione}[Stratification vector]\label{def:stratification}
Let $h\in\{1,\ldots,Q\}$ and for any $\mathbb{V}\in\G(h)$ we denote with $\mathfrak{s}(\mathbb{V})$ the vector
$$
\mathfrak{s}(\mathbb{V}):=(\dim(V_1\cap\mathbb{V}),\ldots, \dim(V_\kappa\cap \mathbb{V})),
$$
that with abuse of language we call the {\em stratification}, or the {\em stratification vector}, of $\mathbb{V}$. Furthermore, we define
$$
\mathfrak{S}(h):=\{\mathfrak{s}(\mathbb{V})\in\N^\kappa:\mathbb{V}\in \G(h)\}.
$$
We remark that the cardinality of $\mathfrak{S}(h)$ is bounded by $\prod_{i=1}^\kappa (\dim V_i +1)$ for any $h\in\{1,\ldots,Q\}$. We remark that sometimes in the literature what we call stratification vector is referred to as growth vector.
\end{definizione}

\begin{definizione}[$\mathfrak{s}$-\emph{co}-\emph{normal} Grassmannian]
For any $\mathfrak{s}\in\mathfrak{S}(h)$ we let
$$\G_\unlhd^\mathfrak{s}(h):=\{\mathbb{V}\in\G_\unlhd(h):\mathfrak{s}(\mathbb{V})=\mathfrak{s} \},$$
and we will refer to $\G_\unlhd^\mathfrak{s}(h)$ as the $\mathfrak{s}$-\emph{co}-\emph{normal} Grassmannian.
\end{definizione}

We now collect in the following some properties of the Grassmanians introduced above. The simple proofs are omitted and can be found in the Preprint \cite{antonelli2020rectifiable}.

\begin{proposizione}[Compactness of the Grassmannian, {\cite[Proposition 2.7]{antonelli2020rectifiable}}]\label{prop:CompGrassmannian}
For any $1\leq h\leq Q$ the function
$$
d_{\mathbb G}(\mathbb W_1,\mathbb W_2):=d_{H,\mathbb G}(\mathbb W_1\cap B(0,1),\mathbb W_2\cap B(0,1)),
$$
with $\mathbb W_1,\mathbb W_2\in \G(h)$, is a distance on $\G(h)$. Moreover $(\G(h),d_{\mathbb G})$ is a compact metric space. 
\end{proposizione}

\begin{proposizione}[{\cite[Proposition 2.8]{antonelli2020rectifiable}}]\label{prop:htagliato}
There exists a constant $\hbar_\mathbb{G}>0$, depending only on $\mathbb{G}$, such that if $\mathbb{W},\mathbb{V}\in\G(h)$ and $d_{\mathbb{G}}(\mathbb{V},\mathbb{W})\leq \hbar_{\mathbb{G}}$, then
$\mathfrak{s}(\mathbb{V})=\mathfrak{s}(\mathbb{W})$.
\end{proposizione}

\begin{proposizione}[{\cite[Proposition 2.9]{antonelli2020rectifiable}}]\label{prop:haar}
    Suppose $\mathbb{V}\in\G(h)$ is a homogeneous subgroup of topological dimension $d$. Then $\mathcal{S}^h\llcorner \mathbb{V}$, $\mathcal{H}^h\llcorner\mathbb{V}$, $\mathcal{C}^h\llcorner\mathbb{V}$ and $\mathcal{H}^d_{\mathrm{eu}}\llcorner \mathbb{V}$ are Haar measures of $\mathbb{V}$. Furthermore, any Haar measure $\lambda$ of $\mathbb{V}$ is $h$-homogeneous in the sense that
    $$\lambda(\delta_r(E))=r^h\lambda(E),\qquad\text{for any Borel set }E\subseteq \mathbb{V}.$$
\end{proposizione}

We now introduce the projections related to a splitting $\mathbb G=\mathbb V\cdot\mathbb L$ of the group.

\begin{definizione}[Projections related to a splitting]\label{def:Projections}
For any $\mathbb V\in \G_c(h)$, if we choose a complement $\mathbb L$, we can find two unique elements $g_{\mathbb V}:=P_\mathbb V g\in \mathbb V$ and $g_{\mathbb L}:=P_{\mathbb L}g\in \mathbb L$ such that
$$
g=P_\mathbb V (g)\cdot P_{\mathbb L}(g)=g_{\mathbb V}\cdot g_{\mathbb L}.
$$
We will refer to $P_{\mathbb V}(g)$ and $P_{\mathbb L}(g)$ as the \emph{splitting projections}, or simply {\em projections}, of $g$ onto $\mathbb V$ and $\mathbb L$, respectively.
\end{definizione}

We recall here below a very useful fact on splitting projections.
\begin{proposizione}\label{prop:InvarianceOfProj}
Let us fix $\mathbb V\in\G_c(h)$ and $\mathbb L$ two complementary homogeneous subgroups of a Carnot group $\mathbb G$. Then, for any $x\in\mathbb{G}$ the map $\Psi:\mathbb{V}\to\mathbb{V}$ defined as $\Psi(z):=P_\mathbb{V}(xz)$ is invertible and it has unitary Jacobian. As a consequence $\mathcal{S}^h(P_{\mathbb V}(\mathcal{E}))=\mathcal{S}^h(P_{\mathbb V}(xP_{\mathbb V}(\mathcal{E})))=\mathcal{S}^h(P_{\mathbb V}(x\mathcal{E}))$ for every $x\in \mathbb G$ and $\mathcal{E}\subseteq \mathbb G$ Borel.
\end{proposizione}

\begin{proof}
The first part is a direct consequence of \cite[Proof of Lemma 2.20]{FranchiSerapioni16}. For the second part it is sufficient to use the first part and the fact that for every $x,y\in\mathbb G$ we have $P_{\mathbb V}(xy)=P_{\mathbb V}(xP_{\mathbb V}y)$.
\end{proof}

The following proposition holds for the distance $d$ induced by the norm introduced in \cref{smoothnorm}. The proof is omitted and can be found in the Preprint \cite{antonelli2020rectifiable}.
\begin{proposizione}[{\cite[Proposition 2.11]{antonelli2020rectifiable}}]\label{prop:TopDimMetricDim}
Let $\mathbb G$ be a Carnot group endowed with the homogeneous norm $\|\cdot\|$ introduced in \cref{smoothnorm}. Let $\mathbb W\in\G(h)$ be a homogeneous subgroup of Hausdorff dimension $h$ and of topological dimension $d$. Then
\begin{enumerate}
\item[(i)] there exists a constant $\newC\label{c:2}:=\oldC{c:2}(\mathfrak{s}(\mathbb W))$ such that for any $p\in \mathbb{W}$ and any $r>0$ we have
\begin{equation}
    \mathcal{H}_{\mathrm{eu}}^{d}\left(B(p,r)\cap \mathbb W\right)=\oldC{c:2}r^{h},
\end{equation}
\item[(ii)] there exists a constant $\beta(\mathbb W)$ such that $    \mathcal{C}^{h}\llcorner \mathbb W =\beta(\mathbb{W})\mathcal{H}_{\mathrm{eu}}^{d}\llcorner \mathbb W$,
\item[(iii)] $\beta(\mathbb{W})=\mathcal{H}^d_{\mathrm{eu}}\llcorner \mathbb{W}(B(0,1))^{-1}$ and in particular $\beta(\mathbb{W})=\beta(\mathfrak{s}(\mathbb{W}))$.
\end{enumerate}
\end{proposizione}

\begin{osservazione}\label{rem:Ch1}
We stress here for future references that the proof of item (iii) of \cref{prop:TopDimMetricDim}, see \cite{antonelli2020rectifiable}, follows from the following fact, whose proof is in \cite[Proof of Proposition 2.11]{antonelli2020rectifiable}. It holds that whenever an arbitrary Carnot group $\mathbb G$ is endowed with an arbitrary left-invariant homogeneous distance $d$, then for  every homogeneous subgroup $\mathbb W\subseteq\mathbb G$ of Hausdorff dimension $h$, we have that 
 \begin{equation}
 \mathcal{C}^h(\mathbb W\cap B(0,1))=1.
 \end{equation}
\end{osservazione}

Let $\|\cdot\|$ be a homogeneous norm on $\mathbb G$. A function $\varphi:\mathbb{G}\to\R$ is said to be {\em radially symmetric with respect to $\|\cdot\|$} if there is a function $g:[0,\infty)\to\R$, called \emph{profile function} such that $\varphi(x)=g(\lVert x\rVert)$.

\begin{proposizione}
\label{unif}
Let $\varphi:\mathbb{G}\to\R$ be a radially symmetric function with respect to a homogeneous norm $\|\cdot\|$ on $\mathbb G$, and let $g$ be its profile function. Let $\mathbb V\in \G(h)$. Then the following holds
$$\int \varphi d\mathcal{C}^{h}\llcorner \mathbb{V}=h\int s^{h-1}g(s)ds.$$  
\end{proposizione}

\begin{proof}
It suffices to prove the proposition for positive simple functions, since the general result follows by Beppo Levi's convergence theorem. Thus suppose $\mathbb{V}$ has topological dimension $d$ and let $\varphi(z):=\sum_{i=1}^N a_i \chi_{B(0,r_i)}(z)$ and note that thanks to \cref{rem:Ch1} for any $\mathbb{V}\in\G(h)$ we have that $\mathcal{C}^h\llcorner \mathbb V(B(0,r_i))=r_i^h$, and then
\begin{equation}
\begin{split}
        \int \varphi(z)d\mathcal{C}^{h}\llcorner \mathbb{V}&=\sum_{i=1}^N a_i\mathcal{C}^{h}\llcorner \mathbb{V}(B(0,r_i))=\sum_{i=1}^N a_ir_i^{h}\\
        &=h\sum_{i=1}^N a_i\int_0^{r_i}s^{h-1}ds=h\int\sum_{i=1}^N a_i s^{h-1}\chi_{[0,r_i]}(s)ds
        =h\int s^{h-1}g(s)ds.
        \nonumber
\end{split}
\end{equation}
\end{proof}

Let us conclude this subsection with two Propositions about the projection maps.
\begin{proposizione}[Corollary 2.15 of \cite{FranchiSerapioni16}]\label{prop:distvsproj}If $\mathbb{V}$ and $\mathbb{L}$ are two complementary subgroups, there exists a constant $\newC\label{C:split}(\mathbb{V},\mathbb{L})$ such that for any $g\in\mathbb{G}$ we have
\begin{equation}\label{eqn:SeriuosProj}
\oldC{C:split}(\mathbb{V},\mathbb{L})\lVert P_\mathbb{L}(g)\rVert\leq \dist(g,\mathbb{V})\leq \lVert P_{\mathbb L} (g)\rVert, \qquad\text{for any } g\in\mathbb G.
\end{equation}
In the following, whenever we write $\oldC{C:split}(\mathbb V,\mathbb L)$, we are choosing the supremum of all the constants such that inequality \eqref{eqn:SeriuosProj} is satisfied.
\end{proposizione}

\begin{proposizione}\label{cor:2.2.19}
For any $\mathbb{V}\in \G_c(h)$ with complement $\mathbb L$ there is a constant $ \newC\label{ProjC}(\mathbb{V},\mathbb L)>0$ such that for any $p\in\mathbb{G}$ and any $r>0$ we have
$$\mathcal{S}^{h}\llcorner \mathbb V\big(P_\mathbb{V}(B(p,r))\big)=\oldC{ProjC}(\mathbb{V},\mathbb L)r^{h}.$$
Furthermore, for any Borel set $A\subseteq \mathbb{G}$ for which $\mathcal{S}^{h}(A)<\infty$, we have
\begin{equation}
    \mathcal{S}^{h}\llcorner \mathbb{V}(P_\mathbb{V}(A))\leq 2\oldC{ProjC}(\mathbb{V},\mathbb L)\mathcal{S}^{h}(A).
    \label{eq:n520}
\end{equation}
\end{proposizione}

\begin{proof}
The existence of such $\oldC{ProjC}(\mathbb V,\mathbb L)$ is yielded by \cite[Lemma 2.20]{FranchiSerapioni16}. 
Suppose $\{B(x_i,r_i)\}_{i\in\N}$ is a countable covering of $A$ with closed balls for which $\sum_{i\in\N} r_i^{h}\leq 2\mathcal{S}^{h}(A)$. Then
\begin{equation}
    \mathcal{S}^{h}(P_\mathbb{V}(A))\leq \mathcal{S}^{h}\Big(P_\mathbb{V}\Big(\bigcup_{i\in\N}B(x_i,r_i)\Big)\Big)\leq \oldC{ProjC}(\mathbb{V},\mathbb L)\sum_{i\in\N} r_i^{h}\leq 2 \oldC{ProjC}(\mathbb{V},\mathbb L)\mathcal{S}^{h}(A).
    \nonumber
\end{equation}
\end{proof}

\subsection{Cones over homogeneous subgroups and cylinder with co-normal axis}\label{sub:Cones}
In this subsection, we introduce the intrinsic cone $C_{\mathbb W}(\alpha)$ and the notion of $C_{\mathbb W}(\alpha)$-set, and prove some of their properties. In this subsection $\mathbb G$ will be a fixed Carnot group endowed with an arbitrary homogeneous norm $\|\cdot\|$ that induces a left-invariant homogeneous distance $d$.

\begin{definizione}[Intrinsic cone]\label{def:Cone}
For any  $\alpha>0$ and $\mathbb W\in \G(h)$, we define the cone $C_{\mathbb W}(\alpha)$ as
$$C_\mathbb W(\alpha):=\{w\in\mathbb{G}:\dist(w,\mathbb W)\leq \alpha\|w\|\}.$$
\end{definizione}

\begin{definizione}[$C_{\mathbb W}(\alpha)$-set]\label{def:CvAset}
Given $\mathbb W\in \G(h)$, and $\alpha>0$, we say that a set $E\subseteq \mathbb G$ is a {\em $C_{\mathbb W}(\alpha)$-set} if
$$
E\subseteq p\cdot C_{\mathbb W}(\alpha), \qquad  \text{for any } p\in E. 
$$
\end{definizione}

The following two Lemmata can be found in the Preprint \cite{antonelli2020rectifiable}. The proof of the first is written here for the reader's convenience, since it will be evoked later on, while the proof of the second is omitted.
\begin{lemma}[{\cite[Lemma 2.15]{antonelli2020rectifiable}}]\label{lem:DistanceOfCones}
For any $\mathbb W_1,\mathbb W_2\in \G(h)$, $\varepsilon>0$ and $\alpha>0$ if $d_\mathbb{G}(\mathbb W_1,\mathbb W_2)<\varepsilon/4$, then
$$C_{\mathbb W_1}(\alpha) \subseteq C_{\mathbb W_2}(\alpha+\varepsilon).$$
\end{lemma}

\begin{proof}
We prove that any $z\in C_{\mathbb W_1}(\alpha)$ is contained in the cone $C_{\mathbb W_2}(\alpha+\varepsilon)$. Thanks to the triangle inequality we infer
$$
\text{dist}(z,\mathbb W_2) \leq d(z,b)+\inf_{w\in\mathbb W_2} d (b,w), \qquad \text{for any } b\in\mathbb W_1.
$$
Thus, choosing $b^*\in\mathbb W_1$ in such a way that $d(z,b^*)=\text{dist}(z,\mathbb W_1)$, and evaluating the previous inequality at $b^*$ we get 
\begin{equation}\label{eqn:Estimate1}
\dist(z,\mathbb W_2) \leq \dist(z,\mathbb W_1) +\text{dist}(b^*,\mathbb W_2) \leq \alpha\|z\| + \dist(b^*,\mathbb W_2), 
\end{equation}
where in the second inequality we used $z\in C_{\mathbb W_1}(\alpha)$. 

Let us notice that, given $\mathbb W$ an arbitrary homogeneous subgroup of $\mathbb G$, $p\in \mathbb G$ an arbitrary point such that $p^*\in\mathbb W$ is one of the points at minimum distance from $\mathbb W$ to $p$, then the following inequality holds
\begin{equation}\label{eqn:Estimate2}
\|p^*\|\leq 2\|p\|.
\end{equation}
Indeed,
$$
\|p^*\|-\|p\| \leq \|(p^*)^{-1}\cdot p\| = d(p,\mathbb W)\leq \|p\| \Rightarrow \|p^*\|\leq 2\|p\|.
$$

Now, by homogeneity, since $b^*\in\mathbb W_1$ is the point at minimum distance from $\mathbb W_1$ of $z$, we get that $D_{1/\|z\|}(b^*)$ is the point at minimum distance from $\mathbb W_1$ of $D_{1/\|z\|}(z)$. Thus, since $\|D_{1/\|z\|}(z)\|=1$, from \eqref{eqn:Estimate2} we get that $\|D_{1/\|z\|}(b^*)\|\leq 2$. Finally we obtain
\begin{equation}\label{eqn:Estimate3}
\begin{split}
\text{dist}(b^*,\mathbb W_2)&=\|z\|\text{dist}\big(D_{1/\|z\|}(b^*),\mathbb W_2\big)=\|z\|\text{dist}\big(D_{1/\|z\|}(b^*),\mathbb W_2\cap B(0,4)\big) \leq \\
&\leq \|z\|d_H(\mathbb W_1\cap B(0,4),\mathbb W_2\cap B(0,4)) \\ &=4\|z\|d_H(\mathbb W_1\cap B(0,1),\mathbb W_2\cap B(0,1)) < \varepsilon \|z\|, 
\end{split}
\end{equation}
where the first equality follows from the homogeneity of the distance, and the second is a consequence of the fact that $\|D_{1/\|z\|}(b^*)\|\leq 2$, and thus, from \eqref{eqn:Estimate2}, the point at minimum distance of $D_{1/\|z\|}(b^*)$ from $\mathbb W_2$ has norm bounded above by 4; the third inequality comes from the definition of Hausdorff distance, the fourth equality is true by homogeneity and the last inequality comes from the hypothesis $d_{\mathbb G}(\mathbb W_1,\mathbb W_2)<\varepsilon/4$. Joining \eqref{eqn:Estimate1}, and \eqref{eqn:Estimate3} we get $z\in C_{\mathbb W_2}(\alpha+\varepsilon)$, that was what we wanted. 
\end{proof}

\begin{lemma}[{\cite[Lemma 2.16]{antonelli2020rectifiable}}]\label{lemma:LCapCw=e}
Let $\mathbb V\in \G_c(h)$, and let $\mathbb L$  be a complementary subgroup of $\mathbb V$. There exists $\newep\label{ep:Cool}:=\oldep{ep:Cool}(\mathbb V,\mathbb L)>0$ such that 
$$
\mathbb L\cap C_{\mathbb V}(\oldep{ep:Cool})=\{0\}.
$$
Moreover we can, and will, choose $\oldep{ep:Cool}(\mathbb V,\mathbb L):=\oldC{C:split}(\mathbb V,\mathbb L)/2$.
\end{lemma}

\begin{osservazione}\label{rem:alphaandepsilon1}
Let $\mathbb V\in \G_c(h)$ and let $\mathbb L$ be a complement of $\mathbb V$. Let us notice that if there exists $\alpha>0$ such that $\mathbb L\cap C_{\mathbb V}(\alpha)=\{0\}$, then $\oldC{C:split}(\mathbb V,\mathbb L)\geq \alpha$. Ineed it is enough to prove that $\alpha\|P_{\mathbb L}(g)\|\leq \dist(g,\mathbb V)$ for every $g\in\mathbb G$. If $g\in\mathbb V$ the latter in equality is trivially verified. Hence suppose by contradiction that there exists $g\notin\mathbb V$ such that $\alpha\|P_{\mathbb L}(g)\|>\dist(g,\mathbb V)$. Since $\dist(g,\mathbb V)=\dist(P_{\mathbb L}(g),\mathbb V)$ we conclude that $P_{\mathbb L}(g)\in\mathbb L\cap C_{\mathbb V}(\alpha)=\{0\}$, that is a contradiction since $g\notin \mathbb V$.
\end{osservazione}

We will not use the following proposition in the paper, but it is worth mentioning it.

\begin{proposizione}\label{prop:ComplGrassmannianOpen}
The family of the complemented subgroups $\G_c(h)$ is an open subset of $\G(h)$. 
\end{proposizione}

\begin{proof}
Fix a $\mathbb W\in \G_c(h)$ and let $\mathbb{L}$ be one complementary subgroup of $\mathbb W$ and set $\varepsilon<\min\{\oldep{ep:Cool}(\mathbb{V},\mathbb{L}),\hbar_{\mathbb G}\}$. Then, if $\mathbb{W}^\prime\in\G(h)$ is such that $d_\mathbb{G}(\mathbb{W},\mathbb{W}^\prime)<\varepsilon/4$, \cref{lem:DistanceOfCones} implies that $\mathbb{W}^\prime\subseteq C_{\mathbb{W}}(\varepsilon)$ and in particular $$\mathbb{L}\cap\mathbb{W}^\prime\subseteq \mathbb{L}\cap C_{\mathbb{W}}(\varepsilon)=\{0\}.$$
Moreover, since $\varepsilon<\hbar_\mathbb{G}$ from \cref{prop:htagliato}, we get that $\mathbb W'$ has the same stratification of $\mathbb W$ and thus the same topological dimension. This, jointly with the previous equality and the Grassmann formula, means that $(\mathbb{W}'\cap V_i)+(\mathbb{L}\cap V_i)=V_i$ for every $i=1,\dots,\kappa$. This, jointly with the fact that $\mathbb L\cap\mathbb W'=\{0\}$, implies that $\mathbb L$ and $\mathbb W'$ are complementary subgroups in $\mathbb G$ due to the triangular structure of the product $\cdot$ on $\mathbb G$, see \eqref{opgr}. For an alternative proof of the fact that $\mathbb L$ and $\mathbb W'$ are complementary subgroups,  see also \cite[Lemma 2.7]{JNGV20}.
\end{proof}

We now prove a Proposition that will be useful in the next Section.
\begin{proposizione}\label{injcompl}
Let $\mathbb{W}\in \G_c(h)$ and assume $\mathbb{L}$ is one of the complementary subgroups of $\mathbb W$. Any other subgroup $\mathbb{V}\in \G(h)$ on which $P_\mathbb{W}$ is injective and satisfying the identity $\mathfrak{s}(\mathbb{V})=\mathfrak{s}(\mathbb{W})$ is contained in  $\G_c(h)$ and admits $\mathbb L$ as a complement.
\end{proposizione}

\begin{proof}
The hypothesis $\mathfrak{s}(\mathbb{V})=\mathfrak{s}(\mathbb{W})$ implies that $\mathbb{V}$ and $\mathbb{W}$ have the same topological dimension. If by contradiction there exists a $0\neq y\in \mathbb{L}\cap \mathbb{V}$, then
$$P_\mathbb{W}(y)=0=P_\mathbb{W}(0).$$
This however is not possible since we assumed that $P_\mathbb{W}$ is injective on $\mathbb V$. The fact that $\mathbb{L}\cap \mathbb{V}=\{0\}$ concludes the proof by the same argument we used in the proof of the previous \cref{prop:ComplGrassmannianOpen}.
\end{proof}

The following definition of intrinsically Lipschitz functions is equivalent to the classical one in \cite[Definition 11]{FranchiSerapioni16} because the cones in \cite[Definition 11]{FranchiSerapioni16} and the cones $C_{\mathbb V}(\alpha)$ are equivalent whenever $\mathbb V$ admits a complementary subgroup, see \cite[Proposition 3.1]{FranchiSerapioni16}.
\begin{definizione}[Intrinsically Lipschitz functions]\label{def:iLipfunctions}
Let $\mathbb{W}\in \G_c(h)$ and assume $\mathbb{L}$ is a complement of $\mathbb{W}$ and let $E\subseteq \mathbb{W}$ be a subset of $\mathbb{V}$. Let $\alpha>0$. A function $f:E\to \mathbb{L}$ is said to be an \emph{$\alpha$-intrinsically Lipschitz function} if $\text{graph}(f):=\{v\cdot f(v):v\in E\}$ is a $C_\mathbb{W}(\alpha)$-set. A function $f:E\to \mathbb{L}$ is said to be an \emph{intrinsically Lipschitz function} if there exists $\alpha>0$ such that $f$ is an $\alpha$-intrinsically Lipschitz function.
\end{definizione}

The next Proposition can be found in the Preprint \cite{antonelli2020rectifiable} and thus we omit its proof.
\begin{proposizione}[{\cite[Proposition 2.19]{antonelli2020rectifiable}}, see also \cite{FranchiSerapioni16}]\label{prop:ConeAndGraph}
Let us fix $\mathbb W\in\G_c(h)$ with complement $\mathbb L$. If $\Gamma\subset\mathbb G$ is a $C_{\mathbb W}(\alpha)$-set for some $\alpha\leq \oldep{ep:Cool}(\mathbb W,\mathbb L)$, then the map $P_{\mathbb W}:\Gamma\to\mathbb W$ is injective. As a consequence $\Gamma$ is the intrinsic graph of an intrinsically Lipschitz map defined on $P_{\mathbb W}(\Gamma)$.
\end{proposizione}

We conclude this subsection with a more careful study of the co-normal Grassmanian. These results will turn out to be fundamental when approaching the Marstrand-Mattila rectifiability criterion in \cref{sec:MMconormale}.

\begin{proposizione}\label{prop:grasscompiffC3}
For any $\mathfrak{s}\in\mathfrak{S}(h)$ the function $\mathfrak{e}:\G_\unlhd^\mathfrak{s}(h)\to\R$ defined as
\begin{equation}\label{eqn:mathfrake}
\mathfrak{e}(\mathbb{V}):=\sup\{\oldep{ep:Cool}(\mathbb{V},\mathbb{L}):\mathbb{L}\text{ is a normal complement of }\mathbb{V}\},
\end{equation}
is lower semicontinuous. Moreover the following conclusion holds
\begin{itemize}
    \item[] if $\mathscr{G}\subseteq \G_\unlhd^{\mathfrak{s}}(h)$ is compact with respect to $d_{\mathbb G}$, then there exists a $\mathfrak{e}_{\mathscr G}>0$ such that $\mathfrak{e}(\mathbb{V})\geq \mathfrak{e}_{\mathscr G}$ for any $\mathbb{V}\in{\mathscr{G}}\subseteq \G_\unlhd^{\mathfrak{s}}(h)$.
\end{itemize}
\end{proposizione}

\begin{proof}
Let us prove that the function $\mathfrak e$ is lower semincontinuous. Since $\oldep{ep:Cool}(\mathbb V,\mathbb L)=\oldC{C:split}(\mathbb V,\mathbb L)/2$, see \cref{lemma:LCapCw=e}, it is enough to prove the proposition with $2\mathfrak e(\mathbb V)$ instead of $\mathfrak e(\mathbb V)$, and with $\oldC{C:split}(\mathbb V,\mathbb L)$ instead of $\oldep{ep:Cool}(\mathbb V,\mathbb L)$. Let us fix $\mathbb{V}\in\G_\unlhd^\mathfrak{s}(h)$ and $0<\varepsilon<\mathfrak e(\mathbb V)$, and denote with $\mathbb{L}$ one of the normal complement subgroups of $\mathbb{V}$ for which $\oldC{C:split}(\mathbb V,\mathbb L)> 2\mathfrak{e}(\mathbb{V})-\varepsilon$. For any $\mathbb{W}\in\G_\unlhd^\mathfrak{s}(h)$ thanks to \cref{lem:DistanceOfCones} we have
\begin{equation}\label{eqn:COOOONI}
C_\mathbb{W}(\oldC{C:split}(\mathbb V,\mathbb L)-4d_{\mathbb G}(\mathbb{V},\mathbb{W})-\varepsilon)\subseteq C_\mathbb{V}(\oldC{C:split}(\mathbb V,\mathbb L)-\varepsilon),
\end{equation}
whenever $d_{\mathbb G}(\mathbb V,\mathbb W)$ is small enough. Therefore if $d_{\mathbb G}(\mathbb{V},\mathbb{ W})$ is sufficiently small, the latter inclusion and the same proof as in \cref{lemma:LCapCw=e} imply that
$\mathbb{L}\cap \mathbb{W}\subseteq \mathbb{L}\cap C_\mathbb{V}(\oldC{C:split}(\mathbb{V},\mathbb{L})-\varepsilon)=\{0\}$. Since $\mathbb L\cap \mathbb W=\{0\}$, $\mathbb L$ and $\mathbb V$ are complementary subgroups and $\mathbb V$ and $\mathbb W$ have the same stratification vector, and thus the same topological dimension, we have that  $\mathbb{L}$ is a complement of $\mathbb{W}$ for the same argument used in the proof of \cref{prop:ComplGrassmannianOpen}. Thus, taking \eqref{eqn:COOOONI} into account we get that $\mathbb L\cap C_{\mathbb W}(\oldC{C:split}(\mathbb V,\mathbb L)-4d_{\mathbb G}(\mathbb V,\mathbb W)-\varepsilon)=\{0\}$ and thus, from \cref{rem:alphaandepsilon1}, we get that $\oldC{C:split}(\mathbb W,\mathbb L)\geq \oldC{C:split}(\mathbb V,\mathbb L)-4d_{\mathbb G}(\mathbb V,\mathbb W)-\varepsilon$ whenever $d_{\mathbb G}(\mathbb V,\mathbb W)$ is sufficiently small. This implies that 
$$
2\mathfrak e(\mathbb W)\geq \oldC{C:split}(\mathbb W,\mathbb L)\geq \oldC{C:split}(\mathbb V,\mathbb L)-4d_{\mathbb G}(\mathbb V,\mathbb W)-\varepsilon \geq 2\mathfrak e(\mathbb V)-4d(\mathbb V.\mathbb W)-2\varepsilon,
$$
whenever $d_{\mathbb G}(\mathbb V,\mathbb W)$ is small enough,
and thus
$$\liminf_{d_{\mathbb{G}}(\mathbb{W}, \mathbb{V})\to 0}\mathfrak{e} (\mathbb{W})\geq \mathfrak{e} (\mathbb{V})-\varepsilon,$$
from which the lower semicontinuity follows due to the arbitrariness of $\varepsilon$. The conclusion in item (i) follows since $\mathscr G\subseteq \G_\unlhd^\mathfrak{s}(h)$ is compact and $\mathfrak e$ is lower semincontinuous. 
\end{proof}

\begin{osservazione}
We observe that in the previous proposition we did not use the fact $\mathbb L$ is normal, but we stated the proposition in this specific case since we are going to use this formulation in the paper. The same proof works in the more general case when $\mathbb V\in \G_c^{\mathfrak s}(h)$ and $\mathfrak e (\mathbb V)=\sup\{\oldep{ep:Cool}(\mathbb V,\mathbb L):\mathbb L\,\,\text{is a complement of}\,\,\mathbb V\}$.
\end{osservazione}

\begin{proposizione}\label{lip:const:proj:conorm}
Let $C>0$ and $\mathbb V\in \G_\unlhd^\mathfrak{s}(h)$ be such that $\mathfrak e(\mathbb V)\geq C$. Then there exists a normal complement $\mathbb{L}$ of $\mathbb{V}$ such that 
\begin{equation}
\lVert P_\mathbb{V}(g)\rVert\leq (1+2/C)\|g\|,\qquad\text{and}\qquad\lVert P_\mathbb{L}(g)\rVert\leq (2/C)\lVert g\rVert, \qquad \text{for all $g\in\mathbb G$,}
\end{equation}
provided $P_\mathbb{V}$ and $P_\mathbb{L}$ are the projections relative to the splitting $\mathbb{G}=\mathbb{V}\mathbb{L}$.
\end{proposizione}

\begin{proof}
Thanks to the definition of $\mathfrak e(\mathbb V)$, see \eqref{eqn:mathfrake}, there exists a normal complementary subgroup $\mathbb{L}$ of $\mathbb V$ such that $\oldep{ep:Cool}(\mathbb V,\mathbb L)\geq C/2$. Thus, from \cref{lemma:LCapCw=e}, we get $\mathbb{L}\cap C_\mathbb{V}(C/2)=\{0\}$. This implies, arguing as in \cref{rem:alphaandepsilon1}, that for any $g\in \mathbb{G}$ we have
\begin{equation}
    C\lVert P_\mathbb{L}(g)\rVert/2\leq \dist(\mathbb{V},P_\mathbb{L}(g))=\dist(\mathbb{V},g)\leq \lVert g\rVert.
    \label{eq:::num2}
\end{equation}
Furthermore, thanks to the triangle inequality we have
$$\lVert g\rVert\geq \lVert P_\mathbb{V}(g)\rVert-\lVert P_\mathbb{L}(g)\rVert\geq\lVert P_\mathbb{V}(g)\rVert-(2/C)\lVert g\rVert,$$
thus concluding the proof of the proposition.
\end{proof}

\begin{proposizione}\label{prop:projhom}
Let $C>0$ and $\mathbb V\in \G_\unlhd^\mathfrak{s}(h)$ be such that $\mathfrak e(\mathbb V)\geq C$. Let $\mathbb{L}$ be a normal complementary subgroup to $\mathbb V$ as in \cref{lip:const:proj:conorm}. Then the projection $P_\mathbb{V}:\mathbb{G}\to\mathbb{V}$ related to the splitting $\mathbb G=\mathbb V\cdot\mathbb L$ is a $(1+2/C)$-Lipschitz homogeneous homomorphism.
\end{proposizione}

\begin{proof}
Thanks to the fact that $\mathbb{L}$ is normal, we have that for any $x,y\in\mathbb{G}$ the following equality holds
$$P_\mathbb{V}(xy)=P_\mathbb{V}(x_\mathbb{V}x_\mathbb{L}y_\mathbb{V}y_\mathbb{L})=P_\mathbb{V}(x_\mathbb{V}y_\mathbb{V}\cdot y_\mathbb{V}^{-1}x_\mathbb{L}y_\mathbb{V}\cdot y_\mathbb{L})=P_\mathbb{V}(x)P_\mathbb{V}(y).$$
Since $P_\mathbb{V}$ is always an homogeneous map, we infer that $P_\mathbb{V}$ is a homogeneous homomorphism. Moreover, from \cref{lip:const:proj:conorm} we have that
$$
\|P_{\mathbb V}(g)\|\leq (1+2/C)\|g\|,
$$
for every $g\in\mathbb{G}$. Hence from the fact that $P_{\mathbb V}$ is a homomorphism we have
$$
\|P_{\mathbb V}(x)^{-1}P_{\mathbb V}(y)\|= \|P_{\mathbb V}(x^{-1}y)\|\leq (1+2/C)\|x^{-1}y\|, 
$$
for every $x,y\in\mathbb{G}$ and thus $P_{\mathbb V}$ is $(1+2/C)$-Lipschitz.
\end{proof}
\begin{osservazione}
Notice that in the proof of the above proposition we proved that whenever $\mathbb L$ is normal, then $P_{\mathbb V}$ is a homomorphism.
\end{osservazione}

\begin{definizione}[Cylinder]\label{def:CYLINDER}
Let $\mathbb V,\mathbb L$ be two complementary subgroups of $\mathbb G$. For any $u\in \mathbb{G}$, and $r>0$ we define
$$
T(u,r):=P_\mathbb{V}^{-1}(P_\mathbb{V}(B(u,r))).$$
\end{definizione}

In the following proposition we study the structure of cylinders $T(\cdot,\cdot)$ when $\mathbb L$ is normal.

\begin{proposizione}\label{prop:structurecyl}
Let $C>0$ and $\mathbb V\in \G_\unlhd^\mathfrak{s}(h)$ be such that $\mathfrak e(\mathbb V)\geq C$. Let $\mathbb{L}$ be a normal complementary subgroup to $\mathbb V$ as in \cref{lip:const:proj:conorm}. Thus, for any $u\in \mathbb{G}$ we have $T(u,r)=P_\mathbb{V}(u)\delta_rT(0,1)$. Furthermore, we have
$$T(u,r)\subseteq P_\mathbb{V}(u)\delta_rP_\mathbb{V}^{-1}(B(0,(1+2/C))\cap \mathbb{V})=P_\mathbb{V}^{-1}(B(P_\mathbb{V}(u),(1+2/C)r)\cap \mathbb{V}).$$
Finally, for any $h\in\mathbb{L}$ we have
$B(uh,r)\subseteq T(u,r)$.
\end{proposizione}

\begin{proof}
First of all, we note that thanks to \cref{prop:projhom} we have that $w\in P_\mathbb{V}(B(u,r))$ if and only if there exists a $v\in B(0,1)$ such that $w=P_\mathbb{V}(u)\delta_rP_\mathbb{V}(v)$. Therefore, given $u\in\mathbb G$ and $r>0$, we have that $y\in T(u,r)$ if and only if $y=P_\mathbb{V}(u)\delta_rP_\mathbb{V}(v)h$ for some $v\in B(0,1)$ and $h\in \mathbb{L}$. Thus we conclude that
$T(u,r)=P_\mathbb{V}(u)\delta_rT(0,1)$ for every $u\in\mathbb G$ and $r>0$.

Secondly, thanks to \cref{prop:projhom} we infer that $P_\mathbb{V}(B(0,1))\subseteq \mathbb{V}\cap B(0,(1+2/C))$ and thus combining such inclusion with the first part of the proposition we deduce that
$$T(u,r)\subseteq P_\mathbb{V}(u)\delta_rP_\mathbb{V}^{-1}(B(0,(1+2/C))\cap \mathbb{V})=P_\mathbb{V}^{-1}(B(P_\mathbb{V}(u),(1+2/C)r)\cap \mathbb{V}),$$
where the last equality is true since $P_{\mathbb V}$ is a homogeneous homomorphism. 
Finally, thanks to the first part of the proposition, for any $u\in \mathbb{V}$ and any $h\in \mathbb{L}$ we have
$$B(uh,r)\subseteq T(uh,r)=T(u,r),$$
and this concludes the proof of the proposition.
\end{proof}

\subsection{Rectifiable measures in Carnot groups}
In what follows we are going to define the class of $h$-flat measures on a Carnot group and then we will give proper definitions of rectifiable measures on Carnot groups. Again we recall that throughout this subsection $\mathbb G$ will be a fixed Carnot group endowed with an arbitrary left-invariant homogeneous distance.

\begin{definizione}[Flat measures]\label{flatmeasures}
For any $h\in\{1,\ldots,Q\}$ we let $\mathfrak{M}(h)$ to be the {\em family of flat $h$-dimensional measures} in $\mathbb{G}$, i.e.
$$\mathfrak{M}(h):=\{\lambda\mathcal{S}^h\llcorner \mathbb W:\text{ for some }\lambda> 0 \text{ and }\mathbb W\in\G(h)\}.$$ 
Furthermore, if $G$ is a subset of the $h$-dimensional Grassmanian $\G(h)$, we let $\mathfrak{M}(h,G)$ to be the set \begin{equation}
    \mathfrak{M}(h,G):=\{\lambda\mathcal{S}^h\llcorner \mathbb{W}:\text{ for some }\lambda> 0\text{ and }\mathbb{W}\in G\}.
    \label{Gotico(h)}
\end{equation}
We stress that in the previous definitions we can use any of the Haar measures on $\mathbb W$, see \cref{prop:haar}, since they are the same up to a constant.
\end{definizione}

\begin{definizione}[$\mathscr{P}_h$ and $\mathscr{P}_h^*$-rectifiable measures]\label{def:PhRectifiableMeasure}
Let $h\in\{1,\ldots,Q\}$. A Radon measure $\phi$ on $\mathbb G$ is said to be a $\mathscr{P}_h$-rectifiable measure if for $\phi$-almost every $x\in \mathbb{G}$ we have
\begin{itemize}
    \item[(i)]$0<\Theta^h_*(\phi,x)\leq\Theta^{h,*}(\phi,x)<+\infty$,
    \item[(\hypertarget{due}{ii})]there exists a $\mathbb{V}(x)\in\G(h)$ such that $\mathrm{Tan}_h(\phi,x) \subseteq \{\lambda\mathcal{S}^h\llcorner \mathbb V(x):\lambda\geq 0\}$.
\end{itemize}
Furthermore, we say that $\phi$ is $\mathscr{P}_h^*$-rectifiable if (\hyperlink{due}{ii})  is replaced with the weaker
\begin{itemize}
    \item[(ii)*] $\mathrm{Tan}_h(\phi,x) \subseteq \{\lambda\mathcal{S}^h\llcorner \mathbb V:\lambda\geq 0\,\,\text{and}\,\,\mathbb V\in \G(h)\}$.
\end{itemize}
\end{definizione}
\begin{osservazione}(About $\lambda=0$ in \cref{def:PhRectifiableMeasure})\label{rem:AboutLambda=0}
It is readily noticed that, since in \cref{def:PhRectifiableMeasure} we are asking $\Theta^h_*(\phi,x)>0$ for $\phi$-almost every $x$, we can not have the zero measure as a tangent measure. As a consequence, a posteriori, we have that in item (ii) and item (ii)* above we can restrict to $\lambda>0$. We will tacitly work in this restriction from now on.

On the contrary, if we only know that for $\phi$-almost every $x\in\mathbb G$ we have 
\begin{equation}\label{eqn:YAYA}
\Theta^{h,*}(\phi,x)<+\infty, \qquad \text{and} \qquad  \mathrm{Tan}_h(\phi,x)\subseteq \{\lambda\mathcal{S}^h\llcorner\mathbb V(x):\lambda>0\},
\end{equation}
for some $\mathbb V(x)\in \G(h)$, hence $\Theta^h_*(\phi,x)>0$ for $\phi$-almost every $x\in\mathbb G$, and the same property holds with the item (ii)* above. Indeed, if at some $x$ for which \eqref{eqn:YAYA} holds we have $\Theta^h_*(\phi,x)=0$, then there exists $r_i\to 0$ such that $r_i^{-h}\phi(B(x,r_i))=0$. Since $\Theta^{h,*}(\phi,x)<+\infty$, up to subsequences (see \cite[Theorem 1.60]{AFP00}), we have $r_i^{-h}T_{x,r_i}\phi\to \lambda\mathcal{S}^h\llcorner\mathbb V(x)$, for some $\lambda>0$. Hence, by applying \cite[Proposition 1.62(b)]{AFP00} we conclude that $r_i^{-h}T_{x,r_i}\phi(B(0,1))\to \lambda\mathcal{S}^h\llcorner\mathbb V(x)(B(0,1))>0$, that is a contradiction.
\end{osservazione}

Throughout the paper it will be often convenient to restrict our attention to some subclasses of $\mathscr{P}_h$- and $\mathscr{P}_h^*$-rectifiable measures, imposing different restrictions on the algebraic nature of the tangents. More precisely we give the following definition.

\begin{definizione}[Subclasses of $\mathscr{P}_h$ and $\mathscr{P}_h^*$-rectifiable measures]\label{def:SubclassesPh}
Let $h\in\{1,\ldots,Q\}$. In the following we denote by $\mathscr{P}^{c}_h$ the family of those $\mathscr{P}_h$-rectifiable measures such that for $\phi$-almost every $x\in \mathbb{G}$ we have
$$\Tan_h(\phi,x)\subseteq \mathfrak{M}(h,\G_c(h)).$$
Furthermore, the family of those $\mathscr{P}_h^*$-rectifiable measures $\phi$ such that for $\phi$-almost any $x\in\mathbb{G}$ we have
\begin{itemize}
    \item[(i)]
    $\Tan_h(\phi,x)\subseteq \mathfrak{M}(h,\G_c(h))$ is denoted by $\mathscr{P}^{*,c}_h$,
    \item[(ii)] $\Tan_h(\phi,x)\subseteq \mathfrak{M}(h,\G_\unlhd(h))$ is denoted by $\mathscr{P}^{*,\unlhd}_h$,
    \item[(iii)] $\Tan_h(\phi,x)\subseteq \mathfrak{M}(h,\G_\unlhd^{\mathfrak s}(h))$ is denoted by $\mathscr{P}^{*,\unlhd,\mathfrak{s}}_h$.
\end{itemize}
\end{definizione}

\begin{proposizione}\label{prop:density}
Let $h\in\{1,\ldots,Q\}$ and assume $\phi$ is a Radon measure on $\mathbb G$.
If $\{r_i\}_{i\in\N}$ is an infinitesimal sequence such that $r_i^{-h}T_{x,r_i}\phi\rightharpoonup \lambda \mathcal{C}^h\llcorner \mathbb{V}$ for some $\lambda>0$ and $\mathbb{V}\in\G(h)$ then
$$\lim_{i\to\infty} \phi(B(x,r_i))/r_i^h=\lambda.$$
\end{proposizione}

\begin{proof}
Since $\mathcal{C}^h\llcorner \mathbb{V}(x)(\partial B(0,1))=0$, see e.g., \cite[Lemma 3.5]{JNGV20}, thanks to \cref{rem:Ch1} and to \cite[Proposition 1.62(b)]{AFP00}, we have
$$\lambda=\lambda \mathcal{C}^h\llcorner\mathbb{V}(x)(B(0,1))=\lim_{i\to\infty}\frac{T_{x,r_i}\phi(B(0,1))}{r_i^h}=\lim_{i\to\infty}\frac{\phi(B(x,r_i))}{r_i^h},$$
and this concludes the proof.
\end{proof}

The above proposition has the following immediate consequence.

\begin{corollario}\label{cordenstang}
Let $h\in\{1,\ldots,Q\}$ and assume $\phi$ is a $\mathscr{P}^*_h$-rectifiable. Then for $\phi$-almost every $x\in\mathbb{G}$ we have
$$\Tan_h(\phi,x)\subseteq \{\lambda\mathcal{C}^h\llcorner\mathbb{W}:\lambda\in[\Theta^h_*(\phi,x),\Theta^{h,*}(\phi,x)]\text{ and }\mathbb{W}\in\G(h)\}.$$
\end{corollario}

We introduce now a way to estimate how far two measures are.

\begin{definizione}\label{def:Fk}
Given $\phi$ and $\psi$ two Radon measures on $\mathbb G$, and given $K\subseteq \mathbb G$ a compact set, we define 
\begin{equation}
    F_K(\phi,\psi):= \sup\left\{\left|\int fd\phi - \int fd\psi\right|:f\in \mathrm{Lip}_1^+(K)\right\}.
    \label{eq:F}
\end{equation}
We also write $F_{x,r}$ for $F_{B(x,r)}$.
\end{definizione}

\begin{osservazione}\label{rem:ScalinfFxr}
With few computations that we omit, it is easy to see that $F_{x,r}(\phi,\psi)=rF_{0,1}(T_{x,r}\phi,T_{x,r}\psi)$. Furthermore, $F_K$ enjoys the triangular inequality, indeed if $\phi_1,\phi_2,\phi_3$ are Radon measures and $f\in\lip(K)$, then
\begin{equation*}
\begin{split}
\Big\lvert\int fd\phi_1-\int fd\phi_2\Big\rvert&\leq \Big\lvert\int fd\phi_1-\int fd\phi_3\Big\rvert+\Big\lvert\int fd\phi_3-\int fd\phi_2\Big\rvert\\ &\leq F_K(\phi_1,\phi_2)+F_K(\phi_2,\phi_3).
\end{split}
\end{equation*}
The arbitrariness of $f$ concludes that $F_K(\phi_1,\phi_2)\leq F_K(\phi_1,\phi_3)+F_K(\phi_3,\phi_2)$.
\end{osservazione}

The proof of the following criterion is contained in  \cite[Proposition 1.10]{MarstrandMattila20} and we omit the proof.
\begin{proposizione}\label{prop:WeakConvergenceAndFk}
Let $\{\mu_i\}$ be a sequence of Radon measures on $\mathbb G$. Let $\mu$ be a Radon measure on $\mathbb G$. The following are equivalent
\begin{enumerate}
    \item $\mu_i\rightharpoonup \mu$;
    \item $F_K(\mu_i,\mu)\to 0$, for every $K\subseteq \mathbb G$ compact. 
\end{enumerate}
\end{proposizione}

The following proposition is a consequence of the choice of the norm in \cref{smoothnorm}, since it is based on \cref{prop:TopDimMetricDim}.
\begin{proposizione}\label{prop:pianconv}
Let $\mathbb G$ be a Carnot group endowed with the smooth-box norm defined in \cref{smoothnorm}. Let $h\in\{1,\ldots,Q\}$ and suppose that $\{\mathbb{V}_i\}_{i\in\N}$ is a sequence of planes in $\G(h)$ converging in the metric $d_\mathbb{G}$ to some $\mathbb{V}\in\G(h)$. Then, $\mathcal{C}^{h}\llcorner \mathbb{V}_i\rightharpoonup \mathcal{C}^{h}\llcorner \mathbb{V}$.
\end{proposizione}

\begin{proof}
First of all note that \cref{prop:htagliato} implies that there exists a $i_0\in\N$ such that for any $i\geq i_0$ we have that $\mathbb{V}_i$ and $\mathbb{V}$ have the same stratification and thus the same topological dimension $d$. Since the $\mathbb{V}_i$'s have the same stratification if $i\geq i_0$, \cref{prop:TopDimMetricDim}(iii) implies that $\beta(\mathbb{V}_i)=\beta(\mathbb{V})$ for any $i\geq i_0$.
Thus, for any continuous function $f$ with compact support thanks to \cref{prop:TopDimMetricDim} we have\label{eq:hauseu}
\begin{equation}
   \lim_{i\to\infty} \int f d\mathcal{C}^{h}\llcorner \mathbb{V}_i-\int f d\mathcal{C}^{h}\llcorner \mathbb{V}= \lim_{i\to\infty} \beta(\mathbb{V})\bigg(\int f d\mathcal{H}^{d}_{\mathrm{eu}}\llcorner \mathbb{V}_i-\int f d\mathcal{H}^{d}_{\mathrm{eu}}\llcorner \mathbb{V}\bigg)=0,\nonumber
\end{equation}
where the last identity comes from the fact that $\mathcal{H}^{d}_{\mathrm{eu}}\llcorner \mathbb{V}_i\rightharpoonup \mathcal{H}^{d}_{\mathrm{eu}}\llcorner \mathbb{V}$.
\end{proof}

Now we are going to define a functional that in some sense tells how far is a measure from being flat around a point $x\in\mathbb G$ and at a certain scale $r>0$.

\begin{definizione}\label{def:metr}
For any $x\in\mathbb{G}$, any $h\in\{1,\ldots,Q\}$ and any $r>0$ we define the functional:
\begin{equation}\label{eqn:dxr}
    d_{x,r}(\phi,\mathfrak{M}(h)):=\inf_{\substack{\Theta>0,\\ \mathbb V\in \G(h)}} \frac{F_{x,r}(\phi,\Theta \mathcal{S}^{h}\llcorner x\mathbb V)}{r^{h+1}}. 
\end{equation}
Furthermore, if $G$ is a subset of the $h$-dimensional Grassmanian $\G(h)$, we also define
$$ d_{x,r}(\phi,\mathfrak{M}(h,G)):=\inf_{\substack{\Theta>0,\\ \mathbb V\in G }} \frac{F_{x,r}(\phi,\Theta \mathcal{S}^{h}\llcorner x\mathbb V)}{r^{h+1}}.$$

\end{definizione}
\begin{osservazione}\label{rem:dxrContinuous}
It is a routine computation to prove that, whenever $h\in \mathbb N$ and $r>0$ are fixed, the function $x\mapsto d_{x,r}(\phi,\mathfrak M(h,G))$ is a continuous function. The proof can be reached as in \cite[Item (ii) of Proposition 2.2]{MarstrandMattila20}. Moreover, from the invariance property in \cref{rem:ScalinfFxr} and \cref{prop:haar}, if in \eqref{eqn:dxr} we use the measure $\mathcal{C}^h\llcorner x\mathbb V$ we obtain the same functional.
\end{osservazione}

The proof of the next Proposition is in the Preprint \cite{antonelli2020rectifiable} and thus we omit it.

\begin{proposizione}[{\cite[Proposition 2.30]{antonelli2020rectifiable}}]\label{prop:TanAndDxr}
Let $\phi$ be a Radon measure on $\mathbb G$ satisfying item (i) in \cref{def:PhRectifiableMeasure}. Further, let $G$ be a subfamily of $\G(h)$ and let $\mathfrak{M}(h,G)$ be the set defined in \eqref{Gotico(h)}. If for $\phi$-almost every $x\in \mathbb G$ we have $\mathrm{Tan}_h(\phi,x)\subseteq \mathfrak{M}(h,G)$, then for $\phi$-almost every $x\in\mathbb G$ and every every $k>0$ we have
$$
\lim_{r\to 0} d_{x,kr}(\phi,\mathfrak{M}(h,G))=0.$$
\end{proposizione}

The following proposition is an adaptation of \cite[4.4(4)]{Preiss1987GeometryDensities} and it will be crucial in the proof of Marstrand-Mattila's rectifiability criterion in \cref{sec:MMconormale}. 

\begin{proposizione}\label{prop::4.4(4)}
Suppose that $h\in\{1,\ldots,Q\}$, $\phi$ is a  Radon measure supported on a compact set, and let $G\subseteq \G(h)$. If there exists an $x\in E(\vartheta,\gamma)$, a $\sigma\in(0,2^{-10(h+1)}\vartheta^{-1})$ and a $0<t<1/(2\gamma)$ such that
$$
d_{x,t}(\phi,\mathfrak{M}(h,G))\leq \sigma^{h+4},
$$
then there is a $\mathbb{V}\in G$ such that
\begin{itemize}
    \item[(i)] whenever $y,z\in B(x,t/2)\cap x\mathbb V$ and $\sigma t\leq r,s\leq t/2$ we have
    $$
\phi(B(y,r)\cap B(x\mathbb{V},\sigma^2t))\geq (1-2^{10(h+1)}\vartheta\sigma)(r/s)^h\phi(B(z,s));
$$
\item[(ii)] furthermore, if the plane $\mathbb V$ yielded by item (i) above admits a complementary normal subgroup $\mathbb L$, denote by $P_{\mathbb V}$ the splitting projection on $\mathbb V$ according to this splitting. Then for any $k>0$ with $\sigma k<2^{-10h} \vartheta^{-1}$, if we
define $T_\mathbb{V}(0,t/4k):=P_\mathbb{V}^{-1}(P_\mathbb{V}(B(0,t/4k)))$ we have
$$\phi(B(x,t/4)\cap xT_\mathbb{V}(0,t/4k))\leq (1+4\sigma(2kh+1))\mathcal{C}^h(P(B(0,1)))k^{-h}\phi(B(x,t/4)).$$
\end{itemize}
\end{proposizione}

\begin{proof}First of all, we notice that by the definition of $d_{x,t}(\phi,\mathfrak{M}(h,G))$ there exist $\mathbb{V}\in G$ and $\lambda>0$ such that
$$
F_{x,t}(\phi,\lambda\mathcal{C}^h\llcorner x\mathbb V)\leq \sigma^{h+3}t^{h+1}.$$
\paragraph{Proof of (i)}The key of the proof of item (i) is to show that for any $w\in B(x,t/2)\cap x\mathbb V$, any $\tau\in(0,t/2]$ and any $\rho\in(0,\tau]$ we have
\begin{align}
    \phi(B(w,\tau))&\leq \lambda\mathcal{C}^h\llcorner(x\mathbb V)(B(w,\tau+\rho))+\sigma^{h+3}t^{h+1}/\rho,\label{eq:1.1031}\\
    \lambda\mathcal{C}^h\llcorner(x\mathbb V)(B(w,\tau-\rho))&\leq \phi(B(w,\tau)\cap B(x\mathbb{V},\rho))+\sigma^{h+3}t^{h+1}/\rho.\label{eq:1.1032}
\end{align}
Before proving that \eqref{eq:1.1031} and \eqref{eq:1.1032} together imply the claim, we need to give a lower bound for $\lambda$. Since $x\in E(\vartheta,\gamma)$, with the choice $w=x$, $\tau=t/4$, and $\rho:=\sigma^2 t$ we have, from \eqref{eq:1.1031}, that the following inequality holds
\begin{equation}
\begin{split}
\vartheta^{-1}(t/4)^h&\leq \phi(B(x,t/4))\leq \lambda\mathcal{C}^h\llcorner(x\mathbb V)(B(x,(1/4+\sigma^2)t))+\sigma^{h+1}t^{h}\\ &=\lambda (1/4+\sigma^2)^ht^h+\sigma^{h+1}t^h,
\label{eq:chin4.4.4}
\end{split}
\end{equation}
where the last equality comes from \cref{rem:Ch1}.
Since we know that $\sigma\leq 1/(2^{10(h+1)}\vartheta)$, we infer that $\sigma^{h+1}\leq 1/(8^h\vartheta)$, and then from \eqref{eq:chin4.4.4} we infer
\begin{equation}\label{eq:estlambda}
    \vartheta^{-1}4^{-h} \leq \lambda(1/4+\sigma^2)^h+\sigma^{h+1} \quad \text{and in particular}\quad \lambda \geq \vartheta^{-1}2^{-3h},
\end{equation}
where we exploited the fact that $1/4+\sigma^2<1$, the fact that $\sigma^{h+1}\leq 1/(8^h\vartheta)$ and the fact that $4^{-h}-8^{-h}\geq 8^{-h}$.

Let us now prove that \eqref{eq:1.1031} and \eqref{eq:1.1032} imply the claim. Since by hypothesis $r,s\geq \sigma t$ with the choice $\rho=\sigma^2t$ we have $\rho<r,s$. Furthermore since $\sigma t\leq r,s\leq t/2$ and $y,z\in B(x,t/2)\cap x\mathbb V$, the bounds \eqref{eq:1.1031} and \eqref{eq:1.1032} imply
\begin{equation}
    \begin{split}
        \frac{\phi(B(y,r)\cap B(x\mathbb{V},\rho))}{\phi(B(z,s))}&\geq\frac{\lambda\mathcal{C}^h\llcorner(x\mathbb{V})(B(y,r-\rho))-\sigma^{h+3}t^{h+1}/\rho}{\lambda\mathcal{C}^h\llcorner(x \mathbb{V})(B(z,s+\rho))+\sigma^{h+3}t^{h+1}/\rho} \\&=\frac{r^h}{s^h}\frac{\lambda(1-\sigma^2t/r)^h-\sigma^{h+1}(t/r)^h}{\lambda (1+\sigma^2t/s)^{h}+\sigma^{h+1}(t/s)^h}\\
        &\geq\frac{r^h}{s^h}\frac{\lambda(1-\sigma)^h-\sigma^{h+1}(t/r)^h}{\lambda (1+\sigma)^{h}+\sigma^{h+1}(t/s)^h}\geq \frac{r^h}{s^h}\frac{\lambda(1-\sigma)^h-\sigma}{\lambda (1+\sigma)^{h}+\sigma},
        \nonumber
    \end{split}
\end{equation}
where the equality in the second line comes from \cref{rem:Ch1}, and we are using $\sigma t/r\leq 1$, and $\sigma t/s\leq 1$.
Since $2h\sigma\leq 1$, we have that $(1+\sigma)^h\leq 1+2h\sigma$, that can be easily proved by induction on $h$. This together with \eqref{eq:estlambda} and Bernoulli's inequality $(1-\sigma)^h\geq 1-\sigma h$ allows us to finally infer that
$$
\frac{\phi(B(y,r)\cap B(x\mathbb{V},\rho))}{\phi(B(z,s))}\geq \frac{r^h}{s^h} \frac{1-(\lambda h+1)\sigma/\lambda}{1+ (2h\lambda+1)\sigma/\lambda}\geq (1-2^{10(h+1)}\vartheta\sigma)\frac{r^h}{s^h},$$
where the last inequality comes from the fact that $\sigma\leq 1/2^{10(h+1)}\vartheta$, from \eqref{eq:estlambda} and some easy algebraic computations that we omit. An easy way to verify the last inequality is to show that $(1-\widetilde\alpha\sigma)/(1+\widetilde\beta\sigma)\geq 1-\widetilde\gamma\sigma$, where $\widetilde\alpha:=(\lambda h+1)/\lambda$, $\widetilde\beta:=(2h\lambda+1)/\lambda$ and $\widetilde\gamma:=2^{10(h+1)}\vartheta$, and observe that the latter inequality is implied by the fact that $\widetilde\alpha+\widetilde\beta-\widetilde\gamma\leq 0$.

Therefore, we are left to prove \eqref{eq:1.1031} and \eqref{eq:1.1032}. In order to prove \eqref{eq:1.1031}, we let $g(z):=\min\{1,\dist(z,\mathbb{G}\setminus U(w,\tau+\rho))/\rho\}$ and note that
\begin{align}
    \phi(B(w,\tau))&\leq \int g(z) d\phi(z)\leq \int g(z)d\lambda\mathcal{C}^h\llcorner (x\mathbb{V})(z)+\text{Lip}(g) F_{x,t}(\phi,\lambda\mathcal{C}^h\llcorner (x\mathbb{V}))\nonumber \\ &\leq \lambda\mathcal{C}^h\llcorner (x\mathbb{V})(B(w,\tau+\rho))+\sigma^{h+3}t^{h+1}/\rho.\nonumber
\end{align}
 On the the other hand, to prove \eqref{eq:1.1032} we let $h(z):=\min\{1,\dist(z,\mathbb{G}\setminus (U(w,\tau)\cap U(x\mathbb{V},\rho))
)/\rho\}$ and
\begin{equation}
\begin{split}
     \lambda\mathcal{C}^h\llcorner(x\mathbb{V})(B(w,\tau-\rho))&\leq \int h(z) d \lambda\mathcal{C}^h\llcorner(x\mathbb{V})(z) \\
     &\leq \int h(z)d\phi(z)+\text{Lip}(h) F_{x,t}(\phi,\lambda\mathcal{C}^h\llcorner (x\mathbb{V}))\\&\leq \phi(B(w,\tau)\cap B(x\mathbb{V},\rho))+\sigma^{h+3}t^{h+1}/\rho.\nonumber
\end{split}
\end{equation}

\paragraph{Proof of (ii):} In this proof let us fix $\tau:=t/4$ and define the function $\ell(z):=\min\{1,\dist(z,\mathbb{G}\setminus U(U(x,\tau)\cap xT(0,\tau/k),\rho))/\rho\}$, where $0<\rho<\tau$. With this definition we have the following chain of inequalities
\begin{equation}
\begin{split}
    \phi(B(x,\tau)\cap xT(0,\tau/k))
    &\leq \int \ell(z)d \phi(z)
    \leq \int \ell(z)d\lambda \mathcal{C}^h\llcorner (x\mathbb{V})(z)+\text{Lip}(\ell) F_{x,t}(\phi,\lambda\mathcal{C}^h\llcorner (x\mathbb{V}))\\
    &\leq \lambda \mathcal{C}^h\llcorner (x\mathbb{V})(B(x,\tau+\rho)\cap xT(0,\tau/k+\rho))+4^{h+1}\sigma^{h+3}\tau^{h+1}/\rho\\
    &\leq \lambda \mathcal{C}^h\llcorner\mathbb V(P(B(0,1)))(\tau/k+\rho)^h+4^{h+1}\sigma^{h+3}\tau^{h+1}/\rho,\label{numbero1000}
\end{split}
\end{equation}
where the third inequality above comes from the fact that, according to the proof of \cref{prop:projhom}, the projection $P$ is a homomorphism, and then the following chain of equalities holds
\begin{equation}
\begin{split}
        P(B(T(0,\tau/k),\rho))=P(T(0,\tau/k)B(0,\rho))
        =P(B(0,t/k))P(B(0,\rho))=P(B(0,\tau/k+\rho)).
\end{split}
\end{equation}
Putting together \eqref{eq:1.1032} when specialized to the case $w=x$ and $\tau=t/4$, with \eqref{numbero1000} and \cref{rem:Ch1}, we infer that
\begin{equation}
    \begin{split}
        \frac{\phi(B(x,\tau)\cap xT(0,\tau/k))}{\phi(B(x,\tau))}\leq \frac{\lambda \mathcal{C}^h\llcorner\mathbb V(P(B(0,1)))(\tau/k+\rho)^h+4^{h+1}\sigma^{h+3}\tau^{h+1}/\rho}{\lambda (\tau-\rho)^h-4^{h+1}\sigma^{h+3}\tau^{h+1}/\rho}.
    \end{split}
\end{equation}
Since $\sigma^2<1$ we choose $\rho:=\sigma^2 \tau$ and note that since $\sigma k<2^{-10h}\vartheta^{-1}$, the previous inequality yields
\begin{equation}
    \begin{split}
        \frac{\phi(B(x,\tau)\cap xT(0,\tau/k))}{\phi(B(x,\tau))}&\leq \frac{\lambda \mathcal{C}^h(P(B(0,1)))(1/k+\sigma^2)^h+4^{h+1}\sigma^{h+1}}{\lambda (1-\sigma^2)^h-4^{h+1}\sigma^{h+1}}\\ &\leq (1+4\sigma(2kh+1))\mathcal{C}^h(P(B(0,1)))k^{-h},\nonumber
    \end{split}
\end{equation}
where we omit the computations that lead to the last inequality but we stress that we need $\mathcal{C}^h(P(B(0,1)))\geq 1$, that in turns comes from the fact that $P(B(0,1))\supseteq B(0,1)\cap\mathbb V$ and $\mathcal{C}^h(B(0,1)\cap\mathbb V)=1$, due to \cref{rem:Ch1}; and also the bound on $\lambda$ in \eqref{eq:estlambda}. The last inequality concludes the proposition.
\end{proof}

We prove the following compactness result that will be of crucial importance in the proof of the co-normal Marstrand-Mattila rectifiability criterion later on.

\begin{proposizione}\label{prop:CompactnessTangents}
Let $h\in\{1,\ldots,Q\}$ and assume $\phi$ is a $\mathscr{P}^{*}_h$-rectifiable measure. Then, for $\phi$-almost all $x\in\mathbb{G}$ the set $\Tan_h(\phi,x)$ is weak-$*$ compact.
\end{proposizione}

\begin{proof}
Since the statement of the Proposition does not depend on the choice of the left-invariant homogeneous distance on $\mathbb G$, we assume that $\mathbb G$ is endowed with the left-invariant homogeneous distance induced by the smooth-box norm in \cref{smoothnorm}.

Let $x\in \mathbb{G}$ be such that  $0<\Theta^h_*(\phi,x)\leq \Theta^{h,*}(\phi,x)<\infty$ and $\Tan_h(\phi,x)\subseteq \mathfrak{M}(h)$. We now prove that for any sequence $\{\lambda_j\mathcal{C}^h\llcorner \mathbb{V}_j\}_{j\in\N}\subseteq \Tan_h(\phi,x)$, there are a $\lambda>0$ and $\mathbb{V}\in\G(h)$ such that, up to non-relabelled subsequences we have
$$
\lambda_j\mathcal{C}^h\llcorner \mathbb{V}_j\rightharpoonup \lambda\mathcal{C}^h\llcorner \mathbb{V}.
$$
Indeed, thanks to \cref{cordenstang} we have that $\lambda_j\in [\Theta^h_*(\phi,x),\Theta^{h,*}(\phi,x)]$ for any $j\in\N$ and thus we can assume without loss of generality that 
$$\lambda_j\to \lambda\in [\Theta^h_*(\phi,x),\Theta^{h,*}(\phi,x)]$$
up to a non-relabelled subsequence. Furthermore, thanks to \cref{prop:CompGrassmannian} there exists a $\mathbb V\in\G(h)$ such that $\mathbb V_j\to \mathbb V$ with respect to the metric $d_\mathbb{G}$. Thus, thanks to \cref{prop:pianconv} and a simple computation that we omit, we conclude that
$$
\lambda_j\mathcal{C}^h\llcorner \mathbb{V}_j\rightharpoonup \lambda \mathcal{C}^h\llcorner \mathbb{V}.
$$
Since we assumed $\{\lambda_j\mathcal{C}^h\llcorner \mathbb{V}_j\}\subseteq \Tan_h(\phi,x)$ then, for any $j\in\N$ there is a sequence $\{r_\ell(j)\}_{\ell\in\N}$ such that
$$
r_\ell(j)^{-h}T_{x,r_\ell(j)}\phi\rightharpoonup \lambda_j\mathcal{C}^h\llcorner \mathbb{V}_j.
$$
Thus, \cref{prop:WeakConvergenceAndFk} implies that $\lim_{\ell\to\infty}F_{0,1}(r_\ell(j)^{-h}T_{x,r_\ell(j)}\phi,\lambda_j\mathcal{C}^h\llcorner \mathbb{V}_j)=0$,
and in particular for any $j\in\N$ there exists an $\ell_j\in\N$ such that defined $\mathfrak{r}_j:=r_{\ell_j}(j)$ we have
$$
F_{0,1}(\mathfrak{r}_j^{-h}T_{x,\mathfrak{r}_j}\phi,\lambda_j\mathcal{C}^h\llcorner \mathbb{V}_j)\leq 1/j.
$$
Since $\limsup_{j\to\infty}\mathfrak{r}_j^{-h}T_{x,\mathfrak{r}_j}\phi(B(0,r))\leq\Theta^{h,*}(\phi,x)r^h$ for any $r>0$, thanks to \cite[Corollary 1.60]{AFP00}, we can assume without loss of generality that there exists a Radon measure $\nu$ such that
$\mathfrak{r}_j^{-h}T_{x,\mathfrak{r}_j}\phi\rightharpoonup \nu$. On the other hand, by definition we have that $\nu\in\Tan_h(\phi,x)$ and thus by hypothesis on $\phi$ there is a $\eta>0$ and a $\mathbb{W}\in\G(h)$ such that $\nu=\eta\mathcal{C}^h\llcorner \mathbb{W}$.
This implies that for any $j\in\N$ we have
\begin{equation}
\begin{split}
F_{0,1}(\eta \mathcal{C}^h\llcorner \mathbb{W},\lambda  \mathcal{C}^h\llcorner \mathbb{V})&\leq F_{0,1}(\eta \mathcal{C}^h\llcorner \mathbb{W},\mathfrak{r}_j^{-h}T_{x,\mathfrak{r}_j}\phi)+F_{0,1}(\mathfrak{r}_j^{-h}T_{x,\mathfrak{r}_j}\phi,\lambda_j \mathcal{C}^h\llcorner \mathbb{V}_j) \\&+F_{0,1}(\lambda_j  \mathcal{C}^h\llcorner \mathbb{V}_j,\lambda  \mathcal{C}^h\llcorner \mathbb{V})\\
&\leq  F_{0,1}(\eta \mathcal{C}^h\llcorner \mathbb{W},\mathfrak{r}_j^{-h}T_{x,\mathfrak{r}_j}\phi)+1/j+F_{0,1}(\lambda_j  \mathcal{C}^h\llcorner \mathbb{V}_j,\lambda  \mathcal{C}^h\llcorner \mathbb{V}).
\end{split}
    \nonumber
\end{equation}
The arbitrariness of $j$ and \cref{prop:WeakConvergenceAndFk} implies that $F_{0,1}(\eta \mathcal{C}^h\llcorner \mathbb{W},\lambda  \mathcal{C}^h\llcorner \mathbb{V})=0$ and since flat measures are cones, see \cref{rem:ScalinfFxr}, we conclude that $\eta \mathcal{C}^h\llcorner \mathbb{W}=\lambda  \mathcal{C}^h\llcorner \mathbb{V}$.
This shows that $\lambda\mathcal{C}^h\llcorner \mathbb{V}\in \Tan_h(\phi,x)$ and then the proof is concluded.
\end{proof}


\section{Marstrand-Mattila rectifiability criterion for co-normal-\texorpdfstring{$\mathscr{P}_h^*$}{Lg}-rectifiable measures}\label{sec:MMconormale}

This chapter is devoted to the proof of the following result, which is a restatement of the main result of the paper in \cref{thm:MMconormaleIntro}.

\begin{teorema}[Co-normal Marstrand-Mattila rectifiability criterion]\label{thm:MMconormale}
Assume $\phi$ is a $\mathscr{P}_h^{*,\unlhd}$-rectifiable measure on a Carnot group $\mathbb G$. Then there are countably many $\mathbb{W}_i\in \G_\unlhd(h)$, compact sets $K_i\Subset \mathbb{W}_i$ and Lipschitz functions $f_i:K_i\to \mathbb{G}$ such that
$$
\phi(\mathbb{G}\setminus \bigcup_{i\in\N} f_i(K_i))=0.
$$
In particular $\phi$ is $\mathscr{P}_h^c$-rectifiable.
\end{teorema}

We briefly discuss the strategy of the proof of \cref{thm:MMconormale}, which is ultimately an adaptation of Preiss's technique in \cite[Section 4.4(4), Lemma 5.2, Theorem 5.3, and Corollary 5.4]{Preiss1987GeometryDensities} to our setting, see \cref{prop::4.4(4)}, \cref{prop::5.2}, and \cref{prop:PHIK>0}, respectively. In particular we show that whenever a Radon measure satisfies precise structure conditions, see \cref{prop::5.2}, that are always verified whenever $\phi$ is $\mathscr{P}^*_h$-rectifiable with \textbf{tangents that admit at least one normal complementary subgroup}, see \cref{prop::5.2satisfied}, then it is possible to find a Lipschitz function $f:K\Subset \mathbb V\to\mathbb G$, with $\mathbb V\in \G_\unlhd(h)$, such that $\phi(f(K))>0$. This implies that $\mathbb G$ can be covered $\phi$-almost all with $\cup_{i\in\mathbb N}f_i(K_i)$, where $f_i:K_i\Subset \mathbb V_i\to\mathbb G$ are Lipschitz functions, see the first part of the proof of \cref{thm:MMconormale}. 

The last part of \cref{thm:MMconormale} is reached from the first part and the following key observation: if a homogeneous subgroup of a Carnot group admits a normal complementary subgroup, then it is a Carnot subgroup, see \cite[Remark 2.1]{AM20}. Thus the maps $f_i$ are Lipschitz maps between Carnot groups and we can apply Pansu-Rademacher  theorem, see \cite{Pansu}, Magnani's area formula, see \cite{MagnaniPhD}, and a classical argument to conclude that $\mathcal{S}^h\llcorner f_i(K_i)$ is a $\mathscr{P}_h^c$-rectifiable measure, see the last part of the proof of \cref{thm:MMconormale}. From this latter observation, the proof of \cref{thm:MMconormale} is concluded.

\textbf{Throughout all this Section we let $\mathbb G$ be a Carnot group of homogeneous dimension $Q$ equipped with the smooth-box norm introduced in \cref{smoothnorm}. This does not result in a loss of generality since our aim is to prove \cref{thm:MMconormale} that is clearly independent on the choice of the particular left-invariant homogeneous distance on $\mathbb G$. So we may suppose that $\mathbb G$ is endowed with the left-invariant homogeneous distance induced by the smooth-box norm introduced in \cref{smoothnorm}}

\subsection{Rigidity of the stratification of \texorpdfstring{$\mathscr{P}_h^*$}{Lg}-rectifiable measures}

We let $\varphi:\mathbb G\to [0,1]$ be a positive, smooth, radially symmetric function with respect to $\|\cdot\|$, supported in $B(0,2)$, and such that $\varphi\equiv 1$ on $B(0,1)$. We shall denote by $g$ its profile function, that is defined right above the statement of \cref{unif}.

\begin{proposizione}\label{prop:gimel}
For any $h\in \{1,\ldots,Q\}$ there exists a constant $\gimel(\mathbb{G},h)=\gimel>0$ such that for any $\mathbb{V}\in\G(h)$ and any $\mathfrak{s}\in \mathfrak{S}(h)\setminus \{\mathfrak{s}(\mathbb{V})\}$, we have
$$
\inf_{\substack{\mathbb{W}\in \G(h)\\\mathfrak{s}(\mathbb{W})=\mathfrak{s}}}\int \varphi(z) \dist(z,\mathbb{W} ) d\mathcal{C}^h\llcorner\mathbb{V}>\gimel,$$
where the stratification vector $\mathfrak{s}(\cdot)$ was introduced in \cref{def:stratification}.
\end{proposizione}

\begin{proof}
Suppose by contradiction this is not the case. Thus there are two sequences $\{\mathbb{W}_i\}\subseteq  \G(h)$ and $\{\mathbb{V}_i\}\subseteq \G(h)$ such that for any $i\in\N$ we have $\mathfrak{s}(\mathbb{W}_i)\neq \mathfrak{s}(\mathbb{V}_i)$ and
\begin{equation}
\int \varphi(z) \dist(z,\mathbb{W}_i ) d\mathcal{C}^h\llcorner\mathbb{V}_i\leq 1/i.
    \label{eq:ineq:deg}
\end{equation}
Thanks to the pidgeonhole principle and the fact that $\mathfrak{S}(h)$, see \cref{def:stratification}, is a finite set we can assume up to passing to a non re-labelled subsequence that
$$
\mathfrak{s}(\mathbb{W}_i)=\mathfrak{s}_1\neq\mathfrak{s}_2=\mathfrak{s}(\mathbb{V}_i),\qquad \text{for any }i\in\N.$$
Furthermore, thanks to \cref{prop:CompGrassmannian}, we can also assume, up to passing to a non re-labelled subsequence, that
$$
\mathbb{W}_i\underset{d_\mathbb{G}}{\to} \mathbb{W}\in\G(h),\qquad\text{and}\qquad\mathbb{V}_i\underset{d_\mathbb{G}}{\to} \mathbb{V}\in\G(h).
$$
Furthermore, thanks to \cref{prop:htagliato}, we also deduce that
$$\mathfrak{s}(\mathbb{W})=\mathfrak{s}_1\neq \mathfrak{s}_2=\mathfrak{s}(\mathbb{V}).$$
In order to conclude the proof of the proposition we first note for any $\mathbb{U}\in \G(h)$ and any $R>0$, if $z\in B(0,R)$, then every element $u\in\mathbb{U}$ for which $\dist(z,\mathbb{U})=d(u,z)$ is contained in $B(0,2R)$.
The same argument as in \eqref{eqn:Estimate1} and \eqref{eqn:Estimate3} allows us to conclude that for every $z\in B(0,2)$ the following inequality holds
\begin{equation}
    \dist(z,\mathbb{W}_i)\geq \dist(z,\mathbb{W})-8d_\mathbb{G}(\mathbb{W},\mathbb{W}_i), \qquad \mbox{for all $i\in\mathbb N$}.
    \label{eq:ineq_disti1}
\end{equation}
Putting together \eqref{eq:ineq:deg} and \eqref{eq:ineq_disti1} thanks to \cref{unif} we infer
\begin{equation}
\begin{split}
     1/i&\geq \int \varphi(z) \dist(z,\mathbb{W}_i ) d\mathcal{C}^h\llcorner\mathbb{V}_i\geq \int \varphi(z) \dist(z,\mathbb{W} ) d\mathcal{C}^h\llcorner\mathbb{V}_i-8d_\mathbb{G}(\mathbb{W},\mathbb{W}_i)\int \varphi(z)  d\mathcal{C}^h\llcorner\mathbb{V}_i\\
     &=\int \varphi(z) \dist(z,\mathbb{W} ) d\mathcal{C}^h\llcorner\mathbb{V}_i-8d_\mathbb{G}(\mathbb{W},\mathbb{W}_i)h\int s^{h-1}g(s)ds.
\end{split}
\end{equation}
Therefore, since $\varphi(z) \dist(z,\mathbb{W} )$ is a continuous function with compact support, thanks to \cref{prop:pianconv} and sending $i$ to $+\infty$ in the previous inequality we conclude
$$
\int \varphi(z) \dist(z,\mathbb{W} ) d\mathcal{C}^h\llcorner\mathbb{V}=0.$$
In particular $\dist(z,\mathbb{W})=0$ for $\mathcal{S}^h\llcorner \mathbb{V}$-almost every $z\in\mathbb{V}$, and since both $\mathrm{Lie}(\mathbb{V})$ and $\mathrm{Lie}(\mathbb{W})$ are vector subspaces of $\mathrm{Lie}(\mathbb{G})$ we have $\mathbb{V}\subseteq \mathbb{W}$. On the one hand this allows us to infer that
$$\dim(V_i\cap \mathbb{V})\leq \dim(V_i\cap \mathbb{W}),\qquad \text{for any }i\in\{1,\ldots,\kappa\},$$
and on the other hand, since $\mathfrak{s}(\mathbb{V})\neq \mathfrak{s}(\mathbb{W})$, there must exist an $\ell\in\{1,\ldots,\kappa\}$ such that
$\dim(V_{\ell}\cap \mathbb{V})<\dim(V_{\ell}\cap \mathbb{W})$.
This however contradicts the fact that $\mathbb{W}\in \G(h)$, indeed
$$h=\dim_{\mathrm{hom}}\mathbb V=\sum_{i=1}^\kappa i\cdot\dim(V_i\cap \mathbb{V})<\sum_{i=1}^\kappa i\cdot\dim(V_i\cap \mathbb{W})=\dim_{\mathrm{hom}}(\mathbb{W}).$$
\end{proof}

\begin{proposizione}\label{F_scontinuous}
Let $\mathfrak{s}\in \mathfrak{S}(h)$. For any Radon measure $\psi$ we define 
$$
\mathscr{F}_{\mathfrak{s}}(\psi):=\inf_{\substack{\mathbb{W}\in \G(h)\\\mathfrak{s}(\mathbb{W})=\mathfrak{s}}}\int \varphi(z) \dist(z,\mathbb{W} ) d\psi.
$$
Then, the functional $\mathscr{F}_\mathfrak{s}:\mathcal{M}\to\R$ on Radon measures is continuous with respect to the weak-* topology in the duality with the functions with compact support on $\mathbb G$.
\end{proposizione}

\begin{proof}
Let $\psi_i\rightharpoonup\psi$ and note that for any $\mathbb{V}\in \G(h)$ for which $\mathfrak{s}(\mathbb{V})=\mathfrak{s}$, we have
\begin{equation}
  \lim_{i\to+\infty} \int \varphi(z) \dist(z,\mathbb{V} ) d\psi_i=\int \varphi(z) \dist(z,\mathbb{V} ) d\psi,
   \label{eq:convint}
\end{equation}
since $\varphi(z)\dist(z,\mathbb{V})$ is a continuous function with compact support. Let us first prove that
$$
\mathscr{F}_{\mathfrak{s}}(\psi)\leq \liminf_{i\to\infty} \mathscr{F}_{\mathfrak{s}}(\psi_i ).
$$
Indeed, if by contradiction $\mathscr{F}_{\mathfrak s}(\psi)>\liminf_{i\to\infty} \mathscr{F}_{\mathfrak{s}}(\psi_i )$, up to passing to a non re-labelled subsequence in $i$ that realizes the $\liminf$ and up to choosing a quasi-minimizer for $\mathscr{F}_{\mathfrak s}(\psi_i)$, we can find $\delta>0$, and $\mathbb W_i\in \G(h)$ with $\mathfrak{s}(\mathbb W_i)=\mathfrak s$ such that 
\begin{equation}\label{eqn:Leim}
\mathscr{F}_{\mathfrak s}(\psi)>\int\varphi(z)\mathrm{dist}(z,\mathbb W_i)d\psi_i+\delta, \qquad \mbox{for all $i\in\mathbb N$}.
\end{equation}
We can assume that $\mathbb W_i\to\mathbb W\in \G(h)$, with $\mathfrak s(\mathbb W)=\mathfrak s$, up to a non re-labelled subsequence, see \cref{prop:CompGrassmannian} and \cref{prop:htagliato}. Thus since $\psi_i\rightharpoonup \psi$ passing to the limit the right hand side of \eqref{eqn:Leim}\footnote{Setting $f_i(z):=\varphi(z)\dist(z,\mathbb W_i)$ and $f(z):=\varphi(z)\dist(z,\mathbb W)$ we notice that $f_i\to f$ uniformly on $B(0,2)$ since $\mathbb W_i\to\mathbb W$. Thus $|\int fd\psi-\int f_id\psi_i|\leq |\int f d\psi-\int f d\psi_i|+|\int fd\psi_i-\int f_id\psi_i|$ and the limit is zero because $\psi_i\rightharpoonup\psi$, $\sup_i\psi_i(B(0,2))<+\infty$ and $f_i\to f$ uniformly on $B(0,2)$.} 
we obtain $\mathscr{F}_{\mathfrak s}(\psi)>\int \varphi(z)\mathrm{dist}(z,\mathbb W)d\psi$, that is a contradiction with the definition of $\mathscr{F}_{\mathfrak s}$.
The proof of the proposition is concluded if we prove that
$$
\mathscr{F}_{\mathfrak{s}}(\psi)\geq \limsup_{i\to\infty} \mathscr{F}_{\mathfrak{s}}(\psi_i ).
$$
In order to prove the previous inequality let us fix $\varepsilon>0$ and $\mathbb{V}_\varepsilon\in\G(h)$ with $\mathfrak s(\mathbb V_\varepsilon)=\mathfrak s$ such that
\begin{equation}
    \int \varphi(z) \dist(z,\mathbb{V}_\varepsilon) d\psi-\varepsilon\leq\mathscr{F}_{\mathfrak{s}}(\psi).
    \label{eq:bddi}
\end{equation}
Putting together \eqref{eq:convint} and \eqref{eq:bddi}, we infer
\begin{equation}
\begin{split}
   \limsup_{i\to\infty}\mathscr{F}_\mathfrak{s}(\psi_i)-\varepsilon &\leq \limsup_{i\to\infty}\int \varphi(z) \dist(z,\mathbb{V}_\varepsilon) d\psi_i -\varepsilon \\ &=\int \varphi(z) \dist(z,\mathbb{V}_\varepsilon) d\psi-\varepsilon\leq\mathscr{F}_{\mathfrak{s}}(\psi).
   \end{split}
\end{equation}
The arbitrariness of $\varepsilon$ concludes the limsup inequality and thus the proof of the proposition.
\end{proof}

\begin{definizione}
For any $\mathcal{T}\subseteq \mathfrak{M}(h)$ we define $\mathfrak{s}(\mathcal{T})$ to be the set
$$\mathfrak{s}(\mathcal{T}):=\{\mathfrak{s}(\mathbb{V}):\text{there exists a non-null Haar measure of }\mathbb{V}\text{ in }\mathcal{T}\}.$$
Namely we are considering all the possible stratification vectors of the homogeneous subgroups that are the support of some element of $\mathcal{T}$.
\end{definizione}

\begin{teorema}\label{struct:strat}
Assume $\phi$ is a $\mathscr{P}^{*}_h$-rectifiable measure. Then, for $\phi$-almost every $x\in \mathbb{G}$ the set $\mathfrak{s}(\mathrm{Tan}_h(\phi,x))\subseteq \mathfrak{S}(h)$ is a singleton.
\end{teorema}

\begin{osservazione}
In the notation of the above proposition, since for $\phi$-almost every $x\in\mathbb{G}$ we have $\mathrm{Tan}_h(\phi,x)\subseteq \mathfrak{M}(h)$, the symbol $\mathfrak{s}(\mathrm{Tan}_h(\phi,x))$ is well defined $\phi$-almost everywhere.
\end{osservazione}

\begin{proof}
Suppose by contradiction there exists a point $x\in \mathbb{G}$ where
\begin{itemize}
    \item[(i)] $0<\Theta^h_*(\phi,x)\leq\Theta^{h,*}(\phi,x)<\infty$,
    \item[(ii)] $\mathrm{Tan}_h(\phi,x)\subseteq \mathfrak{M}(h)$,
    \item[(iii)] there are $\mathbb{V}_1,\mathbb{V}_2\in \G(h)$ with $\mathfrak{s}(\mathbb{V}_1)\neq \mathfrak{s}(\mathbb{V}_2)$ and $\lambda_1,\lambda_2\geq 0$ such that
    $\lambda_1\mathcal{C}^h\llcorner \mathbb{V}_1,\lambda_2\mathcal{C}^h\llcorner \mathbb{V}_2\in \mathrm{Tan}_h(\phi,x)$.
\end{itemize}
Assume that $\{r_i\}_{i\in\N}$ and $\{s_i\}_{i\in\N}$ are two infinitesimal sequences such that $r_i\leq s_i$ and for which
\begin{equation}
    \frac{T_{x,r_i}\phi}{r_i^h}\rightharpoonup \lambda_1\mathcal{C}^h\llcorner \mathbb{V}_1,\qquad \text{and}\qquad \frac{T_{x,s_i}\phi}{s_i^h}\rightharpoonup \lambda_2\mathcal{C}^h\llcorner \mathbb{V}_2.\nonumber
\end{equation}
Note that thanks to \cref{prop:density}, we have in particular that
$\Theta_*^h(\phi,x)\leq \lambda_1,\lambda_2\leq \Theta^{h,*}(\phi,x)$.
Throughout the rest of the proof we let $\mathfrak{s}:=\mathfrak{s}(\mathbb{V}_1)$ and we define
$$
f(r):=\inf_{\substack{\mathbb{W}\in \G(h)\\\mathfrak{s}(\mathbb{W})=\mathfrak{s}}}\int \varphi(z) \dist(z,\mathbb{W} ) d\frac{T_{x,r}\phi}{r^h}.
$$
Thanks to \cref{prop:gimel} and \cref{F_scontinuous} we infer that the function $f$ is continuous on $(0,\infty)$ and that
$$\lim_{i\to\infty} f(r_i)=0\qquad \text{and}\qquad \lim_{i\to\infty}f(s_i)>\gimel \lambda_2\geq \gimel \Theta^{h}_*(\phi,x).$$
Let us choose, for $i$ sufficiently large, $\sigma_i\in[r_i,s_i]$ in such a way that $f(\sigma_i)=\gimel\Theta^{h}_*(\phi,x)/2$ and $f(r)\leq \gimel\Theta^{h}_*(\phi,x)/2$ for any $r\in [r_i,\sigma_i]$. Up to passing to a non re-labelled subsequence, since $\phi$ is $\mathscr{P}^{*}_h$-rectifiable, we can assume that $\sigma_i^{-h}T_{x,\sigma_i}\rightharpoonup \lambda_3\mathcal{C}^h\llcorner \mathbb{V}_3$ for some $\lambda_3>0$ and some $\mathbb{V}_3\in\G(h)$. Thanks to \cref{prop:density}, we infer that $\Theta_*^h(\phi,x)\leq \lambda_3\leq \Theta^{h,*}(\phi,x)$ and thanks to the continuity of the functional $\mathscr{F}_\mathfrak{s}$ in \cref{F_scontinuous}, we conclude that
\begin{equation}
    \gimel\Theta^{h}_*(\phi,x)/2=\lim_{i\to\infty} f(\sigma_i)=\lim_{i\to\infty} \mathscr{F}_{\mathfrak{s}}(\sigma_i^{-h}T_{x,\sigma_i} \phi)=\lambda_3\mathscr{F}_{\mathfrak{s}}(\mathcal{C}^h\llcorner \mathbb{V}_3).
    \label{eqquifin}
\end{equation}
The chain of identities \eqref{eqquifin} together with the bounds on $\lambda_3$ imply
\begin{equation}
    0<\gimel\Theta^{h}_*(\phi,x)/2\Theta^{h,*}(\phi,x)\leq \mathscr{F}_{\mathfrak{s}}(\mathcal{C}^h\llcorner \mathbb{V}_3)\leq \gimel/2.
    \label{eqqui2}
\end{equation}
Since $\mathbb{V}_3\in\G(h)$, \eqref{eqqui2} on the one hand implies by means of \cref{prop:gimel} that $\mathfrak{s}(\mathbb{V}_3)=\mathfrak{s}$. On the other hand, since $\mathscr{F}_{\mathfrak{s}}(\mathcal{C}^h\llcorner \mathbb{V}_3)>0$, we have that $\mathfrak{s}(\mathbb{V}_3)\neq\mathfrak{s}$, resulting in a contradiction.
\end{proof}

\begin{definizione}
Assume $\phi$ is a $\mathscr{P}^{*}_h$-rectifiable measure. For every $x\in \mathbb{G}$ we define the map $\mathfrak{s}(\phi,x)\in \N^\kappa$ in the following way
\begin{equation}
    \begin{split}
        \mathfrak{s}(\phi,x):=\begin{cases}
        \mathfrak{s} &\text{if }\mathrm{Tan}_h(\phi,x)\subseteq \mathfrak{M}(h)\text{ and }\mathfrak{s}(\mathrm{Tan}_h(\phi,x))\text{ is the singleton }\{\mathfrak{s}\},\\
        0 &\text{otherwise}.
        \end{cases}
        \nonumber
    \end{split}
\end{equation}
\end{definizione}

\begin{osservazione}
The map $\mathfrak{s}(\phi,\cdot)$ is well defined and non-zero $\phi$-almost everywhere thanks to \cref{struct:strat}.
\end{osservazione}

\begin{proposizione}\label{propmeasuS}
Assume $\phi$ is a $\mathscr{P}^{*}_h$-rectifiable measure. Then, the map $x\mapsto \mathfrak{s}(\phi,x)$ is $\phi$-measurable.
\end{proposizione}

\begin{proof}
Let $\hbar_\mathbb{G}$ be the constant introduced in \cref{prop:htagliato}. Let us first prove that there exists $\widetilde\alpha:=\widetilde\alpha(\mathbb G)$ such that the following assertion holds
\begin{equation}\label{eqn:ClaimTildeAlpha}
\text{for any $1\leq h\leq Q$ and for any $\mathbb V,\mathbb W\in\G(h)$, if $\mathbb V\subseteq C_{\mathbb W}(\widetilde\alpha)$, then $d_{\mathbb G}(\mathbb V,\mathbb W)\leq \hbar_{\mathbb G}$.}
\end{equation}
Indeed, if this was not the case, we can find an $1\leq h\leq Q$ and sequences $\{\mathbb V_i\}$, $\{\mathbb W_i\}$ in $\G(h)$ such that $\mathbb V_i\subseteq C_{\mathbb W_i}(i^{-1})$ and for which $d_\mathbb G(\mathbb V_i,\mathbb W_i)>\hbar_\mathbb{G}$, for all $i\in\mathbb N$. Thus, up to non re-labelled subsequences, we can assume that $\mathbb V_i\to\mathbb V$ and $\mathbb W_i\to\mathbb W$, for some $\mathbb V,\mathbb W\in \G(h)$, thanks to \cref{prop:CompGrassmannian}. Thanks to the aforementioned convergences and the fact that $\mathbb V_i\subseteq C_{\mathbb W_i}(i^{-1})$ for every $i\in\mathbb N$ we deduce that $\mathbb V\subseteq \mathbb W$ and thus $\mathbb V=\mathbb W$ since they both have homogeneous dimension $h$. But this latter equality is readily seen to be in contradiction with the fact that $d_\mathbb G(\mathbb V_i,\mathbb W_i)>\hbar_\mathbb{G}$, for all $i\in\mathbb N$, since $\mathbb W_i\to\mathbb W$ and $\mathbb V_i\to\mathbb V$.

Let $\{\mathbb{V}_\ell\}_{\ell=1,\ldots, N}$ be a finite $\widetilde\alpha/3$-dense set in $\G(h)$, where $\widetilde\alpha$ is defined above. For any $r\in (0,1)\cap \Q$ and $\ell=1,\ldots, N$ we define the functions on $\mathbb G$
$$
f_{r,\ell}(x):=r^{-h}\phi(\{w\in B(x,r):\dist(x^{-1}w,\mathbb{V}_\ell)\geq\widetilde\alpha\lVert x^{-1}w\rVert\})=:r^{-h}\phi(I(x,r,\ell)).
$$
We claim that the functions $f_{r,\ell}$ are upper semicontinuous. Let $\{x_i\}_{i\in\N}$ be a sequence of points converging to some $x\in\mathbb{G}$ and without loss of generality we assume that $\lim_{i\to\infty}r^{-h}\phi({I(x_i,r,\ell)})$ exists. Since the sets $I(x_i,r,\ell)$ are contained in $B(x,1)$ provided $i$ is sufficiently big, we infer thanks to Fatou's Lemma that
\begin{equation}
\limsup_{i\to\infty}f_{r,\ell}(x_i)=\frac{1}{r^h}\limsup_{i\to\infty}\int \chi_{I(x_i,r,\ell)}(z)d\phi(z)\leq \frac{1}{r^h}\int\limsup_{i\to\infty} \chi_{I(x_i,r,\ell)}(z)d\phi(z).
\label{eq:::num20}
\end{equation}
Furthermore, since $x_i\to x$ and the sets $I(x_i,r,\ell)$ and $I(x,r,\ell)$ are closed, we have
$$\limsup_{i\to\infty} \chi_{I(x_i,r,\ell)}=\chi_{\limsup_{i\to+\infty} I (x_i,r,\ell)}\leq \chi_{I(x,r,\ell)},
$$ 
where the first equality is true in general. Then, from \eqref{eq:::num20}, we infer that
\begin{equation}
\limsup_{i\to\infty}f_{r,\ell}(x_i)\leq \frac{1}{r^h}\int\limsup_{i\to\infty} \chi_{I(x_i,r,\ell)}(z)d\phi(z)\leq \frac{1}{r^h}\int \chi_{I(x,r,\ell)}(z)d\phi(z)=f_{r,\ell}(x),
\nonumber
\end{equation}
and this concludes the proof that $f_{r,\ell}$ is upper semicontinuous.
This implies that the function
$$
f_\ell:=\liminf_{r\in\Q, r\to 0} f_{r,\ell},
$$ 
is $\phi$-measurable and as a consequence, since $\mathrm{Tan}_{h}(\phi,x)\subseteq \mathfrak{M}(h)$ for $\phi$-almost every $x\in\mathbb{G}$, we infer that the set
$$
B_\ell:=\{x\in\mathbb{G}:f_\ell(x)=0\}\cap \{x\in \mathbb{G}:\mathrm{Tan}_h(\phi,x)\subseteq \mathfrak{M}(h)\},
$$
is $\phi$-measurable as well.
If we prove that for $\phi$-almost any $x\in B_\ell$ there exists a non-zero Haar measure $\nu$ in $\mathrm{Tan}_h(\phi,x)$ relative to a homogeneous subgroup $\mathbb{V}$ of $\mathbb{G}$ such that $d_\mathbb{G}(\mathbb{V},\mathbb{V}_\ell)\leq \hbar_\mathbb{G}$, we infer that
\begin{equation}
    \mathfrak{s}(\mathrm{Tan}_h(\phi,x))=\{\mathfrak{s}(\mathbb{V}_\ell)\},\qquad\text{for }\phi\text{-almost any }x\in B_\ell,
    \label{eqguyes}
\end{equation}
and thus $\mathfrak{s}(\phi,x)=\mathfrak{s}(\mathbb{V}_\ell)$ for $\phi$-almost every $x\in B_\ell$. Indeed, if we are able to find such a measure $\nu$ relative to $\mathbb V$, \eqref{eqguyes} is an immediate consequence of the fact that if $d_\mathbb{G}(\mathbb{V},\mathbb{V}_\ell)\leq \hbar_{\mathbb G}$, \cref{prop:htagliato} implies that $\mathbb{V}$ and $\mathbb{V}_\ell$ have the same stratification; and the fact that, from \cref{struct:strat}, $\phi$-almost everywhere the tangent subgroups have the same stratification.

In order to construct such a non-zero Haar measure $\nu$, we fix a point $x\in B_\ell$ in the $\phi$-full-measure subset of $B_\ell$ such that the following conditions hold
\begin{itemize}
    \item[(i)] $0<\Theta^h_*(\phi,x)\leq \Theta^{h,*}(\phi,x)<\infty$,
    \item[(ii)] $\mathrm{Tan}_{h}(\phi,x)\subseteq \mathfrak{M}(h)$,
\end{itemize}
 and we let $\{r_i\}_{i\in\N}$ be an infinitesimal sequence of rational numbers such that $\lim_{i\to\infty}f_{r_i,\ell}(x)=0$.
 
Thanks to item (i) above and the compactness of measures, see \cite[Proposition 1.59]{AFP00}, we can find a non re-labelled subsequence of $r_i$ such that
$$
r_i^{-h}T_{x,r_i}\phi\rightharpoonup \nu.
$$
Such a $\nu$ belongs by definition to $\mathrm{Tan}_h(\phi,x)$ and thus there is a $\lambda>0$ and a $\mathbb{V}\in\G(h)$ such that $\nu=\lambda\mathcal{C}^h\llcorner \mathbb{V}$. Thanks to \cite[Proposition 2.7]{DeLellis2008RectifiableMeasures}, we infer that
\begin{equation}
\begin{split}
    \nu(\{w\in U(0,1):\dist(w,\mathbb{V}_\ell)>\widetilde\alpha\lVert w\rVert\})&\leq \liminf_{i\to\infty}r_i^{-h}T_{x,r_i}\phi(\{w\in U(0,1):  \dist(w,\mathbb{V}_\ell)>\widetilde\alpha\lVert w\rVert\}) \\
    &=\liminf_{i\to\infty}r_i^{-h}\phi(\{w\in U(x,r_i): \dist(x^{-1}w,\mathbb{V}_\ell)>\widetilde\alpha\lVert x^{-1}w\rVert\})=0,
    \nonumber
    \end{split}
\end{equation}
where the last identity comes from the choice of the sequence $r_i$. This shows in particular that 
$$
\mathbb{V}\subseteq \{w\in \mathbb G:\dist(w,\mathbb{V}_\ell)\leq \widetilde\alpha\lVert w\rVert\}=C_{\mathbb V_\ell}(\widetilde\alpha),
$$
and then, from \eqref{eqn:ClaimTildeAlpha} we conclude that $d_\mathbb G(\mathbb V,\mathbb V_\ell)\leq \hbar_{\mathbb G}$, that was what we wanted to prove.

An immediate consequence of \eqref{eqguyes} is that
\begin{equation}
    \text{if }\ell,m\in\{1,\ldots,N\}\text{ and }\mathfrak{s}(\mathbb{V}_\ell)\neq \mathfrak{s}(\mathbb{V}_m) \text{ then }\phi(B_\ell\cap B_m)=0.
    \label{eq:interBls}
\end{equation}
On the other hand, the $B_\ell$'s cover $\phi$-almost all $\mathbb G$. To prove this latter assertion, we note that since $\phi$ is $\mathscr{P}^{*}_h$-rectifiable, for $\phi$-almost all $x\in \mathbb{G}$ there is an infinitesimal sequence $r_i\to 0$, a $\lambda>0$ and a $\mathbb{V}\in \G(h)$ such that $r_i^{-h}T_{x,r_i}\phi\rightharpoonup \lambda\mathcal{C}^h\llcorner \mathbb{V}$. Since the set $\{\mathbb V_\ell:\ell=1,\ldots,N\}$ is $\widetilde\alpha/3$-dense in $\G(h)$, there must exist an $\ell\in\{1,\ldots,N\}$ such that
\begin{equation}\label{eqn:NONSOPIU}
\mathbb{V}\subseteq \{w\in\mathbb{G}:\dist(w,\mathbb{V}_\ell)< \widetilde\alpha\lVert w\rVert\}.
\end{equation}
This last inclusion follows since there exists $\ell$ such that $d_{\mathbb G}(\mathbb V,\mathbb V_\ell)\leq \widetilde\alpha/3$ and the observation that every point in $\partial B(0,1)\cap\mathbb V$ is such that every point at minimum distance of it from $\mathbb V_\ell$ is in $B(0,2)\cap\mathbb V_\ell$. The previous inclusion, jointly with \cite[Proposition 2.7]{DeLellis2008RectifiableMeasures}, implies that
\begin{equation}
\begin{split}
    f_\ell(x)=\liminf_{r\in\Q, r\to 0} f_{r,\ell}(x)&\leq \liminf_{i\to\infty}f_{r_i,\ell}(x) \\ &=\liminf_{i\to\infty}r_i^{-h}\phi(\{w\in B(x,r_i):\dist(x^{-1}w,\mathbb{V}_\ell)\geq \widetilde\alpha\lVert x^{-1}w\rVert\})\\
    &\leq\limsup_{i\to\infty}r_i^{-h}T_{x,r_i}\phi(\{w\in B(0,1):\dist(w,\mathbb{V}_\ell)\geq \widetilde\alpha\lVert w\rVert\})\\
    &\leq \lambda \mathcal{C}^h\llcorner \mathbb{V}(\{w\in B(0,1):\dist(w,\mathbb{V}_\ell)\geq \widetilde\alpha\lVert w\rVert\})=0,
    \label{eq:f_l0}
\end{split}
\end{equation}
where the last inequality is true since \eqref{eqn:NONSOPIU} holds. 
This proves that $x\in B_\ell$ and as a consequence that the $B_\ell$'s cover $\phi$-almost all $\mathbb{G}$. 

We are ready to prove the measurability of the map $x\mapsto \mathfrak{s}(\phi,x)$. Fix an $\mathfrak{s}\in\mathfrak{S}(h)$ and let $D(\mathfrak{s}):=\{x\in\mathbb{G}:\mathfrak{s}(\phi,x)=\mathfrak{s}\}\cap\bigcup_{\ell=1}^N B_\ell$. Since by the previous step the $B_\ell$'s cover $\phi$-almost all $\mathbb{G}$ we know that $\{x\in\mathbb{G}:\mathfrak{s}(\phi,x)=\mathfrak{s}\}\setminus \bigcup_{l=1}^N B_\ell$ is $\phi$-null and thus it is $\phi$-measurable. Furthermore, thanks to \eqref{eqguyes} and \eqref{eq:interBls} we know that up to $\phi$-null sets we have
$$D(\mathfrak{s})=\bigcup_{\mathfrak s\in\mathfrak S(h)}\{B_\ell:\mathfrak{s}(\mathbb{V}_\ell)=\mathfrak{s}\}.$$
Since the sets $B_\ell$ are $\phi$-measurable, this concludes the proof that $\{x\in\mathbb G:\mathfrak s(\phi,x)=\mathfrak s\}$ is $\phi$-measurable for every $\mathfrak s\in \mathfrak S(h)$, taking also into account that $\mathfrak{s}(\phi,\cdot)^{-1}(0)$ is $\phi$-null.
\end{proof}

\subsection{Proof of \cref{thm:MMconormale}}
This long and technical section is devoted to the proof of \cref{thm:MMconormale}. 
\begin{definizione}\label{defC6C7}
Let $C>0$ be a real number. Through the rest of this section we let
$$
\newC\label{C:proj}(C):=1+2/C,
$$
and
$$\newC\label{C:U}(C):=(10(1+\oldC{C:proj}))^{2(Q+10)}.$$
\end{definizione}

\begin{osservazione}\label{rk:contradiction5.2}
Let $\mathfrak s\in \mathfrak{S}(h)$ be fixed and let $\mathbb V\in G^{\mathfrak s}_\unlhd(h)$ with $\mathfrak e(\mathbb V)\geq C$, where $\mathfrak e$ is defined in \eqref{eqn:mathfrake}. Let $\mathbb L$ be a complement of $\mathbb V$ and $P:=P_{\mathbb V}$ the projection on $\mathbb V$ related to this splitting.  Note that with the previous choices of $\oldC{C:proj}$ and $\oldC{C:U}$, for any $h\in\{1,\ldots,Q\}$, thanks to \cref{lip:const:proj:conorm} and \cref{rem:Ch1}, we have
$$2(1+\oldC{C:proj})^h\mathcal{C}^h(P(B(0,1)))<\oldC{C:U}/2^{h+3},$$
since $\mathcal{C}^h(P(B(0,1)))\leq \mathcal{C}^h\llcorner\mathbb V(B(0,\oldC{C:proj}))=\oldC{C:proj}^h$.
\end{osservazione}

\begin{proposizione}\label{prop::5.2}
Let $h\in\{1,\ldots,Q\}$, $\mathfrak{s}\in\mathfrak{S}(h)$, and let $\mathscr{G}$ be a subset of $\G_\unlhd^{\mathfrak s}(h)$ such that there exists a constant $C>0$ for which
$$
\text{$\mathfrak e(\mathbb V)\geq C$ for all $\mathbb V\in\mathscr{G}$},
$$ 
where we recall that $\mathfrak e$ was defined in \eqref{eqn:mathfrake}. 
Further let $r>0$, $\varepsilon\in (0,5^{-h-5}\oldC{C:U}^{-3h}]$, $r_1:=(1-\varepsilon/h)r$, and $\mu:=2^{-7}h^{-3}\oldC{C:U}^{-5h}\varepsilon^2$, where $\oldC{C:proj}$ and $\oldC{C:U}$ are defined in terms of $C$ in \cref{defC6C7}.

Let $\phi$ be a Radon measure and let $z\in\supp(\phi)$. We define $Z(z,r_1)$ to be the set of the triplets $(x,s,\mathbb{V})\in B(z,\oldC{C:U}r_1)\times (0,\oldC{C:U}r]\times \G^{\mathfrak s}_\unlhd(h)$ such that
\begin{equation}
    \phi(B(y,t))\geq (1-\varepsilon)(t/\oldC{C:U}r)^h\phi(B(z,\oldC{C:U}r)),
    \label{eq:2.9mm}
\end{equation}
whenever $y\in B(x,\oldC{C:U}s)\cap x\mathbb{V}$ and $t\in [\mu s, \oldC{C:U}s]$. The geometric assumption we make on $\phi$ is that we can find a compact subset $E$ of $B(z,\oldC{C:U}r_1)$ such that $z\in E$, 
\begin{equation}
    \phi(B(z,\oldC{C:U}r_1)\setminus E)\leq \mu^{h+1}\oldC{C:U}^{-h}\phi(B(z,\oldC{C:U}r_1)),
    \label{eq:densityE}
\end{equation}
and such that for any $x\in E$ and every $s\in(0,\oldC{C:U}r-d(x,z)]$ there is a $\mathbb{V}\in\G^{\mathfrak s}_\unlhd(h)$ such that $(x,s,\mathbb{V})\in Z(z,r_1)$. Furthermore we assume that there exists $\mathbb{W}\in\mathscr G$ such that $(z,r,\mathbb{W})\in Z(z,r_1)$, and let us fix $\mathbb L$ a normal complementary subgroup of $\mathbb W$ such that \cref{prop:projhom} holds. Let us denote $P:=P_\mathbb{W}$ the projection on $\mathbb W$ related to the splitting $\mathbb G=\mathbb W\mathbb{L}$.

Let us recall that with the notation $T(u,r)$ we mean the cylinder with center $u\in\mathbb G$ and radius $r>0$ related to the projection $P=P_\mathbb W$, see \cref{def:CYLINDER}. For any $u\in P(B(z,r_1))$ let $s(u)\in[0,r]$ be the smallest number with the following property: for any $s(u)<s\leq r$ we have
\begin{enumerate}
    \item $E\cap T(u,s/4h)\neq \emptyset$, and
    \item $\phi\big(B(z,\oldC{C:U}r)\cap T(u,\oldC{C:proj}s))\leq \mu^{-h}(s/\oldC{C:U}r)^h\phi(B(z,\oldC{C:U}r))$.
\end{enumerate}
Finally, we define
\begin{itemize}
    \item[($\alpha$)] $\text{ }\text{ }A:=\{u\in P(B(z,r_1)):s(u)=0\}$,
    \item[($\beta$)] $A_1:=\Big\{u\in P(B(z,r_1)): s(u)>0,\text{ and }\phi\big(B(z,\oldC{C:U}r)\cap T(u,\oldC{C:proj}s(u))\big)\geq \varepsilon^{-1}(s(u)/\oldC{C:U}r)^h\phi(B(z,\oldC{C:U}r))\Big\}$,
    \item[($\gamma$)]$A_2:=\Big\{u\in P(B(z,r_1)): s(u)>0,\text{ and }\phi\big((B(z,\oldC{C:U}r)\setminus E)\cap T(u,s(u)/4h) \big)\geq 2^{-1}(s(u)/4h\oldC{C:U}r)^h\phi(B(z,\oldC{C:U}r))\Big\}$.
\end{itemize}
Then we have 
\begin{itemize}
    \item[(i)] $s(u)\leq  \oldC{C:U}h\mu r$ for every $u\in P(B(z,r_1))$,
    \item[(ii)] The function  $u\mapsto s(u)$ is lower semicontinuous on $P(B(z,r_1))$ and as a consequence $A$ is compact,
    \item[(iii)] $ P(B(z,r_1))\subseteq A\cup A_1\cup A_2$,
    \item[(iv)]$\mathcal{C}^h(P(B(z,r))\setminus A)\leq 5^{h+3}\oldC{C:U}^{3h}\mathcal{C}^h(P(B(0,1)))\varepsilon r^h$,
    \item[(v)] $P(E\cap P^{-1}(A))=A$, $\mathcal{S}^h(E\cap P^{-1}(A))>0$ and there is a constant $\mathfrak{C}>1$ such that
    $$\mathfrak{C}^{-1}\mathcal{S}^h(E\cap P^{-1}(A))\leq \phi(E\cap P^{-1}(A))\leq \mathfrak{C}\mathcal{S}^h(E\cap P^{-1}(A)).$$
    \end{itemize}
\end{proposizione}
\begin{proof} We prove each point of the proposition in a separate paragraph. For the sake of notation we write $Z:=Z(z,r_1)$, and without loss of generality we will always assume that $z=0$, since $P_{\mathbb W}$ is a homogeneous homomorphism, see \cref{prop:projhom}, and thus the statement is left-invariant.  Since it will be used here and there in the proof, we estimate $\phi(B(0,\oldC{C:U}r)\setminus B(0,\oldC{C:U}r_1))$. Since $(0,r,\mathbb W)\in Z$, we infer that
\begin{equation}
    \phi(B(0,\oldC{C:U} r_1))\geq (1-\varepsilon)(r_1/r)^h\phi(B(0,\oldC{C:U}r)).
    \nonumber
\end{equation}
This implies that
\begin{equation}
\begin{split}
      \phi(B(0,\oldC{C:U}r)\setminus B(0,\oldC{C:U}r_1))
      &= \phi(B(0,\oldC{C:U}r))-\phi(B(0,\oldC{C:U}r_1))\leq \phi(B(0,\oldC{C:U}r))(1-(1-\varepsilon)(r_1/r)^h) \\
      &=\phi(B(0,\oldC{C:U}r))(1-(1-\varepsilon)(1-\varepsilon/h)^h)\leq 2\varepsilon \phi(B(0,\oldC{C:U}r)),
\end{split}
  \label{bd:bd1}
\end{equation}
where in the last inequality we used that $h \mapsto (1-\varepsilon/h)^h$ is increasing.

\item\paragraph{Proof of (i):} Let $u\in P(B(0,r_1))$ and let $ \oldC{C:U}\mu h r<s\leq r$. Then
$$\phi(B(0,\oldC{C:U}r)\cap T(u, \oldC{C:proj}s))\leq \phi(B(0,\oldC{C:U}r))\leq \mu^{-h}(s/\oldC{C:U}r)^h\phi(B(0,\oldC{C:U}r)),$$
where the last inequality comes from the fact that $ \oldC{C:U}\mu h r<s$.
Defined $v:=u\delta_\mu(u^{-1})$, we immediately note that $v\in \mathbb{W}$ and that, from \cref{lip:const:proj:conorm}, $ d(v,u)=\mu d(u,0)\leq \oldC{C:proj}\mu r$. Furthermore, for every $\Delta\in B(0,\mu r)$ we have
\begin{equation}
    \begin{split}
    d(0,u\delta_\mu(u^{-1})\Delta)&\leq \mu\|u\|+\|u\|+\|\Delta\|\leq \mu \oldC{C:proj}r_1+\oldC{C:proj}r_1+\mu r \\
    &\leq (\oldC{C:proj}(1+\mu)+2\mu)r_1\leq \oldC{C:U}r_1,
    \label{incl:palli}
\end{split}
\end{equation}
where in the inequality above we used the fact that $r_1>r/2$, and $\oldC{C:U}>2(\oldC{C:proj}+1)>\oldC{C:proj}(1+\mu)+2\mu$.
Thus, on the one hand we have $B(v,\mu r)\subseteq B(u,(1+\oldC{C:proj})\mu r)$ and on the other, thanks to \eqref{incl:palli}, we deduce that
\begin{equation}\label{eqn:Deduzione}
B(v,\mu r)\subseteq B(0,\oldC{C:U}r_1).
\end{equation}
Since $(0,r,\mathbb W)\in Z$, this implies thanks to the definition of $Z$ and $E$ that
\begin{equation}
    \phi(B(v,\mu r))\geq (1-\varepsilon)\mu^h\oldC{C:U}^{-h} \phi(B(0,\oldC{C:U}r_1))> \phi(B(0,\oldC{C:U}r_1)\setminus E).
    \label{eq:ineq:balls}
\end{equation}
Furthermore, thanks to \eqref{eqn:Deduzione}, \eqref{eq:ineq:balls} and the definition of $T(\cdot,\cdot)$, we also infer that
$$
\emptyset\neq E\cap B(v,\mu r) \subseteq E\cap B(u,(1+\oldC{C:proj})\mu r)\subseteq E\cap T(u,s/4h),
$$
where the last inclusion is true since $(1+\oldC{C:proj})\mu r\leq \oldC{C:U}\mu r /4 <s/(4h)$.
\paragraph{Proof of (ii):} Let $u\in P(B(0,r_1))$ and let $0<s\leq s(u)$. By definition of $s(u)$, up to eventually increasing $s$ such that it still holds $0<s\leq s(u)$, there are two cases. Either
\begin{equation}
    \phi(B(0,\oldC{C:U}r)\cap T(u,\oldC{C:proj}s))>(1+\tau)^h\mu^{-h}(s/{ \oldC{C:U}}r)^h\phi(B(0,\oldC{C:U}r)),
    \label{eq:defs(u)}
\end{equation}
 for some $\tau>0$ or
\begin{equation}
    E\cap T(u,s/4h)= \emptyset.
    \label{eq:defs(u)2}
\end{equation}
If $v\in P(B(0,r_1))$ is sufficiently close to $u$ then $s+\oldC{C:proj}^{-1}d(u,v)\leq (1+\tau)s$ and $s+\oldC{C:proj}^{-1}d(u,v)\leq r$, since $s(u)\leq r$ thanks to point (i). If \eqref{eq:defs(u)} holds, this implies that
\begin{equation}\label{eqn:EXT1}
\begin{split}
        \phi(B(0,\oldC{C:U}r)\cap T(v,\oldC{C:proj}(s+\oldC{C:proj}^{-1} d(u,v))))&>\phi(B(0,\oldC{C:U}r)\cap T(u,\oldC{C:proj}s))\\
        &\geq (1+\tau)^h\mu^{-h}(s/{ \oldC{C:U}}r)^h\phi(B(0,\oldC{C:U}r)) \\
        &\geq\mu^{-h} ((s+\oldC{C:proj}^{-1}d(u,v))/{ \oldC{C:U}}r)^h\phi(B(0,\oldC{C:U}r)),
\end{split}
\end{equation}
where the last inequality is true provided $d(u,v)$ is suitably small. On the other hand, if \eqref{eq:defs(u)2} holds, then
\begin{equation}\label{eqn:EXT2}
E\cap T(v,(s-4hd(u,v))/4h)\subseteq E\cap T(u,s/4h)=\emptyset.
\end{equation}
Taking into account \eqref{eqn:EXT1} and \eqref{eqn:EXT2}, this shows that $s(v)\geq\min\{s-4hd(u,v),s+\oldC{C:proj}^{-1}d(u,v)\}=s-4hd(u,v)$ provided $v$ is sufficiently close to $u$. This implies that
$\liminf_{v\to u}s(v)\geq s$ for any $s\leq s(u)$ for which at least one between \eqref{eq:defs(u)} and \eqref{eq:defs(u)2} holds. In particular, from the definition of $s(u)$, we deduce that there exists a sequence $s_i\to s(u)^-$ such that at each $s_i$ at least one between \eqref{eq:defs(u)} and \eqref{eq:defs(u)2} holds. In conclusion we infer
$$
\liminf_{v\to u} s(v)\geq s(u).
$$

\paragraph{Proof of (iii):} Suppose that $u\in P(B(0,r_1))\setminus (A\cup A_1)$. Since $u\not\in A\cup A_1$, then $s(u)>0$ and
\begin{equation}
    \phi\big(B(0,\oldC{C:U}r)\cap T(u,\oldC{C:proj}s(u))\big)< \varepsilon^{-1}(s(u)/{ \oldC{C:U}}r)^h\phi(B(0,\oldC{C:U}r)).
    \label{eq:contro}
\end{equation}
Thanks to the definition of $s(u)$, for any $0<s<s(u)$, up to eventually increasing $s$ in such a way that it still holds $0<s<s(u)$, we have either
\begin{equation}
    \phi(B(0,\oldC{C:U}r)\cap T(u,\oldC{C:proj}s))>\mu^{-h}(s/{ \oldC{C:U}}r)^h\phi(B(0,\oldC{C:U}r)),
    \label{eq:defs(u)prim}
\end{equation}
or
\begin{equation}
    E\cap T(u,s/4h)= \emptyset.
    \label{eq:defs(u)2prim}
\end{equation}
Let us assume that $\eqref{eq:defs(u)2prim}$ does not hold for some $s<s(u)$. Then $\eqref{eq:defs(u)2prim}$ does not hold for any $t$ such that $s\leq t<s(u)$. Thus, in this case, we deduce the existence of $t_i<s(u)$ such that $t_i\to s(u)$ for which \eqref{eq:defs(u)prim} holds. Thus we have
\begin{equation}
    \begin{split}
        \mu^{-h}(s(u)/{ \oldC{C:U}}r)^h\phi(B(0,\oldC{C:U}r))&=\lim_{i\to+\infty}\mu^{-h}(t_i/{ \oldC{C:U}}r)^h\phi(B(0,\oldC{C:U}r))\\
        &\leq \limsup_{i\to +\infty}\phi(B(0,\oldC{C:U}r)\cap T(u,\oldC{C:proj}t_i)) \\& \leq \phi(B(0,\oldC{C:U}r)\cap T(u,\oldC{C:proj}s(u))) \\ 
        &\leq \varepsilon^{-1}(s(u)/{ \oldC{C:U}}r)^h\phi(B(0,\oldC{C:U}r)),
    \end{split}
\end{equation}
that is a contradiction thanks to the choice of $\mu$ and $\varepsilon$. This proves that for any  $0<\rho<s(u)$ we have $E\cap  T(u,v/4h)=\emptyset$ and thus
\begin{equation}
    \begin{split}
        E\cap \text{int}(T(u,s(u)/4h))= \emptyset.
        \nonumber
    \end{split}
\end{equation}
Let us now define the constants
$$
s:=16hs(u)/\varepsilon,\qquad\text{and}\qquad\sigma:=(2h-1)\varepsilon/32h^2.
$$
Thanks to item (i), from which $s(u)\leq { \oldC{C:U}}h\mu r$, and from the very definition of $\mu$, we deduce that
\begin{equation}
        0<s(u)\leq s=16hs(u)/\varepsilon\leq r-r_1, \qquad \text{and}\qquad \mu\leq \sigma\leq 1.
\end{equation}
Thanks to the compactness of $E$ and the definition of $s(u)$ we have that 
$$E\cap T(u,s(u)/4h)\neq \emptyset.$$
Let us fix $x\in E\cap T(u,s(u)/4h)$ and assume  $\mathbb{V}\in  \G^{\mathfrak s}_\unlhd(h)$ to be such that $(x,s,\mathbb{V})\in Z$. We claim that
\begin{equation}
    \lVert P(x^{-1}y)\rVert\geq \sigma\lVert x^{-1}y\rVert,\qquad\text{for every }y\in x\mathbb V.
    \label{eq:lipcond}
\end{equation}
Assume by contradiction that there is a $y\in x\mathbb{V}$ such that $\lVert x^{-1}y\rVert=1$ and for which $\lVert P(x^{-1}y)\rVert<\sigma$. Let us fix $w\in B(0,\sigma s)$ and let $t\in\mathbb R$ be such that $\lvert t\rvert\leq\oldC{C:proj} s(u)/(4h\sigma)$. Then, we have
\begin{equation}
    d(0,x\delta_t(x^{-1}y)w)\leq d(0,x)+\lvert t\rvert \lVert x^{-1}y\rVert+\sigma s\leq d(0,x)+\frac{\oldC{C:proj}s(u)}{4h\sigma}+\sigma s.
    \label{eq:bdB(C7)}
\end{equation}
Thanks to the choice of the constants and item (i), according to which $s(u)\leq { \oldC{C:U}}h\mu r$, we infer that
\begin{equation}
\begin{split}
    \frac{\oldC{C:proj}s(u)}{4h\sigma}+\sigma s&\leq \oldC{C:proj}s(u)(1-1/2h+8h/((2h-1)\varepsilon))\\
    &\leq \oldC{C:proj}2^{-7}h^{-2}\varepsilon^2r(1-1/2h+8h/((2h-1)\varepsilon))\leq \oldC{C:proj}\varepsilon r/h,
    \label{eq:birdi1}
\end{split}
\end{equation}
where in the first inequality above we are using the fact that $\oldC{C:proj}\geq 1$, and in the second we are using the explicit expression $\mu=2^{-7}h^{-3}\oldC{C:U}^{-5h}\varepsilon^2$ and the fact that $\oldC{C:U}^{-5h+1}<1$.
Hence, since $x\in B(0,\oldC{C:U}r_1)$ putting together \eqref{eq:bdB(C7)} and \eqref{eq:birdi1} we infer that 
\begin{equation}\label{eqn:Estdeltat}
d(0,x\delta_t(x^{-1}y)w)\leq \oldC{C:U}r_1+\oldC{C:proj}\varepsilon r/h<\oldC{C:U}r,
\end{equation}
where the second inequality comes from the definition of $r_1$ and the fact that $\oldC{C:U}\geq \oldC{C:proj}$. As a consequence of the previous computations we finally deduce that
$$
B(x\delta_t(x^{-1}y),\sigma s)\subseteq B(0,\oldC{C:U}r), \qquad\text{for any }\lvert t\rvert\leq\oldC{C:proj} s(u)/(4h\sigma).
$$
We now prove that for any $\lvert t\rvert\leq \oldC{C:proj} s(u)/(4h\sigma)$ and any $w\in B(0,\sigma s)$, we have
\begin{equation}\label{eqn:TIME}
x\delta_t(x^{-1}y) w\in T(u,\oldC{C:proj}s(u)).
\end{equation}
Indeed, thanks to \cref{prop:projhom}, we have that $$P(x\delta_t(x^{-1}y)w)=P(x)\delta_t(P(x^{-1}y))P(w)$$ and thus since $x\in T(u,s(u)/4h)$ by means of \cref{prop:structurecyl} we infer that 
$$d(u,P(x))\leq \oldC{C:proj}s(u)/4h.$$
Thanks to this, and together with the fact that $\|P(w)\|\leq \oldC{C:proj}\sigma s$ due to \cref{lip:const:proj:conorm}, we can estimate
\begin{equation}
    \begin{split}
d(u,P(x)\delta_t(P(x^{-1}y))P(w))&\leq d(u,P(x))+\lvert t\rvert\lVert P(x^{-1}y)\rVert+\oldC{C:proj}\sigma s\\
&\leq \frac{\oldC{C:proj}s(u)}{4h}+\frac{\oldC{C:proj}s(u)}{4h}+\oldC{C:proj}\sigma s\leq\frac{\oldC{C:proj}s(u)}{2h}+\oldC{C:proj}\Big(1-\frac{1}{2h}\Big)s(u)
\leq \oldC{C:proj}s(u),
\nonumber
    \end{split}
\end{equation}
where in the second inequality of the last line we are using $\sigma s=s(u)(1-1/(2h))$.
Summing up, the above computations yield that
\begin{equation}\label{eqn:TIME2}
B(x\delta_t(x^{-1}y),\sigma s)\subseteq B(0,\oldC{C:U}r)\cap T(u,\oldC{C:proj}s(u)), \quad \text{ for any }\lvert t\rvert\leq \oldC{C:proj} s(u)/(4h\sigma).
\end{equation}
Now we are in a position to write the following chain of inequalities
\begin{equation}
\begin{split}
    \phi(B(0,\oldC{C:U}r)\cap T(u,\oldC{C:proj}s(u)))&\geq (2\sigma s)^{-1}\int_{-s(u)/4h\sigma}^{s(u)/4h\sigma}\phi(B(x\delta_t(x^{-1}y),\sigma s)) dt\\
    &\geq (2\sigma s)^{-1}(s(u)/2h\sigma)(1-\varepsilon)(\sigma s/r\oldC{C:U})^h\phi(B(0,\oldC{C:U}r))\\
    &= (1-\varepsilon)(1-1/2h)^h16h^2(2h-1)^{-2}\varepsilon^{-1}(s(u)/\oldC{C:U}r)^h\phi(B(0,\oldC{C:U}r))\\
    &\geq \varepsilon^{-1}(s(u)/\oldC{C:U}r)^h\phi(B(0,\oldC{C:U}r))
\end{split}
\label{eqn:ContR}
\end{equation}
where the first inequality is true by { applying Fubini theorem} to the function $F(t,z):=\chi_{B(0,\sigma s)}(\delta_t(y^{-1}x)x^{-1}z)$ on the domain $[-s(u)/(4h\sigma),s(u)/(4h\sigma)]\times \mathbb G$, and by noticing that when $|t|\leq s(u)/(4h\sigma)$ we have \eqref{eqn:TIME2}; the second inequality is true since $x\in E$ and then $(x, s,\mathbb V)\in Z$ for some $\mathbb V\in { \G_\unlhd^{\mathfrak s}(h)}$; and the last inequality is true since $(1-\varepsilon)(1-1/(2h))^h16h^2(2h-1)^{-2}\geq 1$. Since \eqref{eqn:ContR} is a contradiction with the assumption $u\notin A_1$ we get that \eqref{eq:lipcond} holds and thus $P\lvert_{\mathbb{V}}$ is injective, since it is also a homomorphism. Furthermore, since $\mathbb{V}$ has the same stratification as $\mathbb{W}$, \cref{injcompl} implies that $\mathbb{V}\mathbb{L}=\mathbb{G}$, where $\mathbb L$ is the chosen normal complement of $\mathbb W$. Thanks to \cite[Proposition 3.1.5]{FMS14}, there exists an intrinsically linear function $\ell:\mathbb{W}\to\mathbb{L}$ such that $\mathbb{V}=\text{graph}(\ell)$ and thus $P\lvert_\mathbb{V}$ is also surjective. In particular we can find a $w\in x\mathbb{V}$ in such a way that $P(w)=u$ and, by using \eqref{eq:lipcond} and $d(u,P(x))\leq \oldC{C:proj}s(u)/4h$, that follows from \cref{prop:structurecyl}, and the fact that $P$ is a homogeneous homomorphism, we conclude that the following inequality holds
\begin{equation}\label{eqn:xw}
\lVert x^{-1}w\rVert\leq\sigma^{-1}\lVert P(x)^{-1}P(w)\rVert=\sigma^{-1}\lVert P(x)^{-1}u\rVert\leq \frac{\oldC{C:proj}s(u)}{4h\sigma}.
\end{equation}
We now claim that the inclusion
\begin{equation}
    U(w,s(u)/4h)\subseteq (B(0,\oldC{C:U}r)\setminus E)\cap \text{int}(T(u,s(u)/4h)),
    \label{eq:claimfinale3}
\end{equation}
concludes the proof of item (iii). Indeed, we have $(x,s,\mathbb V)\in Z$, and since $w\in B(x,\oldC{C:U}s)\cap x\mathbb V$, see \eqref{eqn:xw}, and we have $\mu s\leq s(u)/4h\leq \oldC{C:U}s$, we infer, by approximation and using the hypothesis, that
\begin{equation}
    \phi(U(w,s(u)/4h))\geq (1-\varepsilon)(s(u)/4h\oldC{C:U}r)^h\phi(B(0,\oldC{C:U}r)).
    \label{eq:claimfinale3.2}
\end{equation}
Putting together \eqref{eq:claimfinale3} and \eqref{eq:claimfinale3.2} we deduce that
\begin{equation}
    \phi\big((B(0,\oldC{C:U}r)\setminus E)\cap \text{int}(T(u,s(u)/4h))\big)\geq (1-\varepsilon)(s(u)/4h\oldC{C:U}r)^h\phi(B(0,\oldC{C:U}r)).
    \nonumber
\end{equation}
and thus $u\in A_2$, which proves item (iii). In order to prove the inclusion \eqref{eq:claimfinale3} we note that since $\lVert x^{-1}w\rVert\leq \oldC{C:proj}s(u)/(4h\sigma)$, see \eqref{eqn:xw}, we have thanks to the same computation we performed in \eqref{eq:bdB(C7)}, \eqref{eq:birdi1}, and \eqref{eqn:Estdeltat}, that $B(w,s(u)/(4h))\subseteq B(0,\oldC{C:U}r)$. Furthermore, since $P(w)=u$ the inclusion \eqref{eq:claimfinale3} follows thanks to the fact that $B(w,s(u)/4h)\subseteq T(u,s(u)/4h)$, see \cref{prop:structurecyl}, and the fact that $\text{int}(T(u,s(u)/4h))\cap E=\emptyset$.

\paragraph{Proof of (iv):} Let $\tau>1$. Thanks to \cite[Theorem 2.8.4]{Federer1996GeometricTheory}, we deduce that there exists a countable set $D\subseteq A_1$ such that the following two hold
\begin{itemize}
    \item[($\alpha)$] $\{B(w,\oldC{C:proj}^2s(w))\cap\mathbb W:w\in D\}$ is a disjointed subfamily of $\{B(w,\oldC{C:proj}^2s(w))\cap\mathbb W:w\in A_1\}$,
    \item[($\beta)$] for any $w\in A_1$ there exists a $u\in D$ such that $B(w,\oldC{C:proj}^2s(w))\cap B(u,\oldC{C:proj}^2s(u))\cap\mathbb W\neq \emptyset$ and $s(w)\leq \tau s(u)$.
\end{itemize} 

Furthermore, if we define for every $u\in A_1$ the set
\begin{equation}\label{eqn:Bhat}
\begin{split}
    \hat{B}(u,\oldC{C:proj}^2s(u)):=\bigcup\{B(w,\oldC{C:proj}^2s(w))\cap\mathbb W:w\in A_1,
    ~B(u,\oldC{C:proj}^2s(u))\cap B(w,\oldC{C:proj}^2s(w))\cap\mathbb W\neq \emptyset,~s(w)\leq \tau s(u) \},
\end{split}
\end{equation}
we have, thanks to \cite[Corollary 2.8.5]{Federer1996GeometricTheory}, that $A_1\subseteq \bigcup_{u\in A_1} B(u,\oldC{C:proj}^2s(u))\cap\mathbb W\subseteq \bigcup_{w\in D} \hat{B}(w,\oldC{C:proj}^2s(w))$. An easy computation based on the triangle inequality, which we omit, leads to the following inclusion
\begin{equation}
    \hat{B}(u,\oldC{C:proj}^2s(u))\subseteq \mathbb{W}\cap B(u,(1+2\tau)\oldC{C:proj}^2s(u)), \qquad \text{for every $u\in A_1$}.
    \label{eq:inclussion}
\end{equation}
Since $D\subseteq A_1$, and since $T(u,\oldC{C:proj}s(u))\subseteq P^{-1}(B(u,\oldC{C:proj}^2s(u))\cap \mathbb W)$ for every $u\in A_1$, see \cref{prop:structurecyl}, we conclude, by exploiting the fact that $\{B(w,\oldC{C:proj}^2s(w))\cap\mathbb W:w \in D\}$ is a disjointed family, the following inequality
$$\phi(B(0,\oldC{C:U}r))\geq \sum_{u\in D}\phi(B(0,\oldC{C:U}r)\cap T(u,\oldC{C:proj}s(u)))\geq \varepsilon^{-1}\sum_{u\in D}(s(u)/\oldC{C:U}r)^h\phi(B(0,\oldC{C:U}r)),$$
where the last inequality above comes from the fact that $D\subseteq A_1$. The above inequality can be rewritten as $\sum_{u\in D}s(u)^h\leq \oldC{C:U}^h \varepsilon r^h$. In particular, thanks to \cref{rem:Ch1}, and \eqref{eq:inclussion} we infer that
\begin{equation}
\begin{split}
        \mathcal{C}^h(A_1)&\leq \sum_{u\in D}\mathcal{C}^h(B(u,(1+2\tau)\oldC{C:proj}^2s(u))\cap\mathbb W)\\
        &=\oldC{C:proj}^{2h}(1+2\tau)^h\sum_{u\in D}s(u)^h\leq \oldC{C:proj}^{2h}\oldC{C:U}^h(1+2\tau)^h\varepsilon r^h.
\end{split}
        \label{eq:bdes1}
\end{equation}

With a similar argument we used to prove the existence of $D$, we can construct a countable set $D^\prime\subseteq A_2$ such that the family $\{B(u,\oldC{C:proj}s(u)/4h)\cap\mathbb W:u\in D^\prime\}$ is disjointed and the family $\{\hat{B}(u,\oldC{C:proj}s(u)/4h):u\in D^\prime\}$, constructed as in \eqref{eqn:Bhat}, covers $A_2$. In a similar way as in \eqref{eq:inclussion} we have $\hat{B}(u,\oldC{C:proj}s(u)/(4h))\subseteq \mathbb{W}\cap B(u,(1+2\tau)\oldC{C:proj}s(u)/4h)$ for every $u\in A_2$. Moreover, since $$T(u,s(u)/4h)\subseteq P^{-1}(B(u,\oldC{C:proj}s(u)/4h)\cap \mathbb W),$$
for every $u\in A_2$, see \cref{prop:structurecyl}, we conclude by exploiting the fact that $\{B(u,\oldC{C:proj}s(u)/(4h))\cap\mathbb W:w \in D'\}$ is a disjointed family, the following inequality
\begin{equation}
    \begin{split}
        \phi(B(0,\oldC{C:U}r)\setminus E)&\geq \sum_{u\in D^\prime}\phi((B(0,\oldC{C:U}r)\setminus E)\cap T(u,s(u)/4h))\\
        &\geq 2^{-1}\phi(B(0,\oldC{C:U}r))\sum_{u\in D^\prime}(s(u)/4h\oldC{C:U}r)^h,
    \end{split}
\end{equation}
where the last inequality holds since $D'\subseteq A_2$.
From the previous inequality, \eqref{bd:bd1}, and the fact that $0\in E$, we infer that
\begin{equation}
    \begin{split}
        \sum_{u\in D^\prime}(s(u)/4h\oldC{C:U}r)^h&\leq\frac{2\phi(B(0,\oldC{C:U}r)\setminus E)}{\phi(B(0,\oldC{C:U}r))}\\
        &\leq2\cdot\frac{\phi(B(0,\oldC{C:U}r)\setminus B(0,\oldC{C:U}r_1))+\phi(B(0,\oldC{C:U}r_1)\setminus E)}{\phi(B(0,\oldC{C:U}r))} \\
        &\leq2\cdot\frac{2\varepsilon \phi(B(0,\oldC{C:U}r)) +\mu^{h+1}\oldC{C:U}^{-h}\phi(B(0,\oldC{C:U}r))}{\phi(B(0,\oldC{C:U}r))}\leq 10\varepsilon.
        \end{split}
\end{equation}
Consequently, we deduce that
\begin{equation}
\begin{split}
     \mathcal{C}^h(A_2)&\leq \sum_{u\in D^\prime} \mathcal{C}^h(\mathbb{W}\cap B(u,(1+2\tau)\oldC{C:proj}s(u)/4h))\\
     &=(1+2\tau)^h\oldC{C:proj}^{h}\sum_{u\in D^\prime} (s(u)/4h)^h\leq 10(1+2\tau)^h\oldC{C:proj}^{h}\oldC{C:U}^h \varepsilon r^h.
\end{split}
    \label{eq:bdes2}
\end{equation}
Finally, putting together \eqref{eq:bdes1}, \eqref{eq:bdes2}, item (iii) of this proposition, and \cref{rem:Ch1}, we conclude the following inequality
\begin{equation}
\begin{split}
        \mathcal{C}^h(P(B(0,r))\setminus A)&\leq \mathcal{C}^h(P(B(0,r))\setminus P(B(0,r_1)))+\mathcal{C}^h(A_1)+\mathcal{C}^h(A_2)\\
        &\leq \mathcal{C}^h(P(B(0,1)))r^h(1-(1-\varepsilon/h)^h)+\oldC{C:proj}^{2h}\oldC{C:U}^h(1+2\tau)^h\varepsilon r^h+10(1+2\tau)^h\oldC{C:proj}^{h}\oldC{C:U}^h \varepsilon r^h\\
        &\leq 50(1+2\tau)^h\oldC{C:U}^{3h}\mathcal{C}^h(P(B(0,1))) \varepsilon r^h,
\end{split}
\nonumber
\end{equation}
where in the last inequality we used that $1\leq \oldC{C:proj}\leq \oldC{C:U}$, and that $\mathcal{C}^h(P(B(0,1))\geq 1$ since $P(B(0,1))\supseteq B(0,1)\cap\mathbb W$ and $\mathcal{C}^h(B(0,1)\cap\mathbb W)=1$, thanks to \cref{rem:Ch1}.
With the choice $\tau=2$, item (iv) follows.

\paragraph{Proof of (v):} Let $u\in A$ and note that since $s(u)=0$, for any $s>0$ we have that
$$
E\cap T(u,s/4h)\neq \emptyset.$$
Since the sets $E\cap T(u,s/4h)$ are compact we infer the following equality thanks to the finite intersection property
$$
\emptyset\neq E\cap \bigcap_{s>0}T(u,s/4h)=E\cap P^{-1}(u).
$$
This implies that $u\in P(E\cap P^{-1}(u))$ for every $u\in A$, and as a consequence $A\subseteq P(E\cap P^{-1}(A))$. Since the inclusion $ P(E\cap P^{-1}(A))\subseteq A$ is obvious we finally infer that
$A=P(E\cap P^{-1}(A))$. Moreover, thanks to item (iv) and to the choice of $\varepsilon<5^{-h-5}\oldC{C:U}^{-3h}$, we conclude that $\mathcal{S}^h(A)>0$ thanks to the fact that $\mathcal{C}^h\llcorner\mathbb W$ and $\mathcal{S}^h\llcorner\mathbb W$ are equivalent, see \cref{prop:haar}, and thanks to the following chain of inequalities
\begin{equation}
\begin{split}
     \mathcal{C}^h(A)&\geq \mathcal{C}^h(P(B(0,r)))-\mathcal{C}^h(P(B(0,r))\setminus A)\\
     &\geq \mathcal{C}^h(P(B(0,1)))r^h-5^{h+3}\oldC{C:U}^{3h}\mathcal{C}^h(P(B(0,1))) \varepsilon r^h\geq\frac{24}{25}r^h.
\end{split}
    \nonumber
\end{equation}
Thanks to the fact that $P$ is $\oldC{C:proj}$-Lipschitz, see \cref{prop:projhom}, we further infer that
$$0<\mathcal{S}^h(A)=\mathcal{S}^h(P(E\cap P^{-1}(A)))\leq \oldC{C:proj}^{h}\mathcal{S}^h(E\cap P^{-1}(A)).$$
For any $s$ sufficiently small and $u\in A$, by definition of $s(u)$ and $A$, we have the following chain of inequalities
$$\phi(B(x,\oldC{C:proj}s))\leq \phi\big(B(0,\oldC{C:U}r)\cap T(u,\oldC{C:proj}s))\leq \mu^{-h}(s/\oldC{C:U}r)^h\phi(B(0,\oldC{C:U}r)),$$
whenever $x\in E\cap P^{-1}(u)$, where the first inequality comes from the fact that $x\in E\subseteq B(0,\oldC{C:U}r_1)$, and \cref{prop:structurecyl}. Finally by \cite[2.10.17(2)]{Federer1996GeometricTheory} and the previous inequality we infer
\begin{equation}
    \phi\llcorner (E\cap P^{-1}(A))\leq \oldC{C:proj}^{-h}\oldC{C:U}^{-h}\mu^{-h}\frac{\phi(B(0,\oldC{C:U}r))}{r^{h}}\mathcal{S}^h\llcorner (E\cap P^{-1}(A)).
    \label{eq:bdhaus1}
\end{equation}
On the other hand, if we assume $x\in E$ and $s$ sufficiently small, we have $(x,s,\mathbb{V})\in Z$ for some $\mathbb{V}\in { \G_\unlhd^\mathfrak{s}(h)}$. This implies that, by using the very definition of $Z$, that
$$\phi(B(x,s))\geq (1-\varepsilon)(s/\oldC{C:U}r)^h\phi(B(0,\oldC{C:U}r)),$$
and thus by \cite[2.10.19(3)]{Federer1996GeometricTheory}, we have
\begin{equation}
\phi\llcorner E\geq (1-\varepsilon)\frac{\phi(B(0,\oldC{C:U}r))}{(\oldC{C:U}r)^{h}}\mathcal{S}^h\llcorner E.
    \label{eq:bdhaus2}
\end{equation}
Putting together \eqref{eq:bdhaus1} and \eqref{eq:bdhaus2}, we conclude the proof of item (v).
\end{proof}

\begin{proposizione}\label{lemma:Gx}
Let $\phi$ be a $\mathscr{P}^{*,\unlhd}_h$-rectifiable measure such that there exists an $\mathfrak s\in\mathbb N^\kappa$ for which for $\phi$-almost every $x\in\mathbb G$ we have 
\begin{equation}\label{eqn:LEIHOLD}
\mathrm{Tan}_h(\phi,x)\subseteq \{\lambda\mathcal{S}^h\llcorner\mathbb V:\lambda >0\,\,\text{and}\,\,\mathbb V\in\G_\unlhd^\mathfrak{s}(h)\}.
\end{equation}
Then, the set
\begin{equation}\label{eqn:DefinitionGx}
\mathscr{G}(x):=\{\mathbb V\in \G_\unlhd^\mathfrak{s}(h):\text{there exists}\,\,\Theta>0\,\,\text{such that}\,\,\Theta\mathcal{S}^h\llcorner\mathbb V\in \mathrm{Tan}_h(\phi,x)\},
\end{equation}
is a compact subset of $\G_\unlhd^\mathfrak{s}(h)$ for all $x\in\mathbb G$ for which \eqref{eqn:LEIHOLD} holds,
and the sets 
\begin{equation}\label{eqn:Gc}
\mathscr G_C:=\{x\in\mathbb{G}:\mathfrak{e}(\mathbb V)\in(C,\infty)\,\,\text{for every}\,\, \mathbb V\in \mathscr{G}(x)\},
\end{equation}
where $\mathfrak e$ is defined in \eqref{eqn:mathfrake}, are $\phi$-measurable for any $C>0$.
\end{proposizione}

\begin{proof}
The fact that $\mathscr{G}(x)$ is compact is an immediate consequence of \cref{prop:CompactnessTangents}, the compactness of the Grassmannian in \cref{prop:CompGrassmannian}, and the convergence result in \cref{prop:pianconv}.
For any $\lambda,k,r>0$ define the function $\mathcal{M}_{\lambda,k,r}(x,\mathbb{V}):\mathbb{G}\times \G_\unlhd^\mathfrak{s}(h)\to \R$ as
$$
\mathcal{M}_{\lambda,k,r}(x,\mathbb{V}):=F_{0,k}(r^{-h}T_{x,r}\phi,\lambda\mathcal{C}^h\llcorner \mathbb{V}),
$$
where $F_{0,k}$ is defined in \cref{def:Fk}. We claim that, for any choice of the parameters, the function $\mathcal{M}_{\lambda,k,r}$ is continuous when $\mathbb{G}\times \G_\unlhd^\mathfrak{s}(h)$ is endowed with respect to the topology induced by the metric $d+d_{\mathbb{G}}$. Indeed, assume $\{x_i\}_{i\in\N}\subseteq \mathbb{G}$ and $\{\mathbb{V}_i\}\subseteq \G_\unlhd^\mathfrak{s}(h)$ are two sequences converging to $x\in\mathbb G$ and $\mathbb{V}\in\G_\unlhd^\mathfrak{s}(h)$ respectively. Thanks to the triangle inequality we have
\begin{equation}
\begin{split}
     \limsup_{i\to\infty}\lvert \mathcal{M}_{\lambda,k,r}(x,\mathbb{V})-\mathcal{M}_{\lambda,k,r}(x_i,\mathbb{V}_i)\rvert
     &\leq \limsup_{i\to\infty}\Big(\lvert \mathcal{M}_{\lambda,k,r}(x,\mathbb{V})-\mathcal{M}_{\lambda,k,r}(x_i,\mathbb{V})\rvert \\ &+\lvert \mathcal{M}_{\lambda,k,r}(x_i,\mathbb{V})-\mathcal{M}_{\lambda,k,r}(x_i,\mathbb{V}_i)\rvert\Big)\\
     &\leq \limsup_{i\to\infty}F_{0,k}(r^{-h}T_{x,r}\phi,r^{-h}T_{x_i,r}\phi)+\limsup_{i\to\infty}F_{0,k}(\lambda\mathcal{C}^h\llcorner \mathbb{V},\lambda\mathcal{C}^h\llcorner \mathbb{V}_i)\\
     &\leq \limsup_{i\to\infty}r^{-(h+1)}d(x,x_i)\phi(B(x,kr+d(x,x_i))) \\
     &+\limsup_{i\to\infty}F_{0,k}(\lambda\mathcal{C}^h\llcorner \mathbb{V},\lambda\mathcal{C}^h\llcorner \mathbb{V}_i)=0,
     \nonumber
\end{split}
\end{equation}
where the inequality in the fourth line comes from a simple computation that we omit
and the last identity comes from \cref{prop:pianconv}. This in particular implies that the function
$$
\mathcal{M}(x,\mathbb{V}):=\sup_{\substack{ k>0\\ k\in\Q}}\inf_{\substack{\lambda>0\\ \lambda\in\Q}}\liminf_{\substack{r\to 0\\r\in \Q}}\frac{\mathcal{M}_{\lambda,k,r}(x,\mathbb{V})}{k^{h+1}},
$$
is Borel measurable. 

We now claim that for $\phi$-almost every $x\in\mathbb{G}$ we have that $\mathbb{V}\in \mathscr{G}(x)$ if and only if $\mathcal{M}(x,\mathbb{V})=0$. Indeed if $\mathbb{V}\in\mathscr{G}(x)$, there is a $\lambda>0$ and an infinitesimal sequence $\{r_i\}_{i\in\N}$ such that $\lim_{i\to\infty}F_{0,k}(r_i^{-h}T_{x,r_i}\phi,\lambda\mathcal{C}^h\llcorner \mathbb{V})=0$ for any $k>0$, see \cref{prop:WeakConvergenceAndFk}. However, by the scaling properties of $F$, see \cref{rem:ScalinfFxr}, we can choose an another infinitesimal sequence $\{s_i\}_{i\in\N}\subseteq \Q$ such that $r_i/s_i\to 1$, and then $\lim_{i\to\infty}F_{0,k}(s_i^{-h}T_{x,s_i}\phi,\lambda\mathcal{C}^h\llcorner \mathbb{V})=0$ for every $k>0$ as well, proving the first half of the claim.
Viceversa, if $\mathcal{M}(x,\mathbb{V})=0$, then for any $j\in\N$ there exists a $\lambda_j>0$, with $\lambda_j\in\mathbb Q$, and an infinitesimal sequence $\{r_i(j)\}\subseteq \Q$ such that $\lim_{i\to\infty} F_{0,1}(r_i(j)^{-h}T_{x,r_i(j)}\phi,\lambda_j\mathcal{C}^h\llcorner \mathbb{V})\leq 1/j$. Since $0<\Theta_*^h(\phi,x)\leq \Theta^{h,*}(\phi,x)<\infty$ for $\phi$-almost every $x\in\mathbb{G}$, we can argue as in the last part of the proof of \cref{prop:CompactnessTangents} and hence we can assume without loss of generality that $\lambda_j$ converge to some non-zero $\lambda$ and that, for every $j\in\mathbb N$, there exists $i_j\in\mathbb N$ such that $r_{i_j}(j)$ is an infinitesimal sequence and $r_{i_j}(j)^{-h}T_{x,r_{i_j}(j)}\phi\rightharpoonup \lambda\mathcal{C}^h\llcorner\mathbb V$. This eventually concludes the proof of the claim. 

Furthermore, since $\mathfrak{e}$ by \cref{prop:grasscompiffC3} is lower semicontinuous on $\G_\unlhd^\mathfrak{s}(h)$, we know that for any $C>0$ the set $\mathbb{G}\times \{\mathbb{W}\in\G_\unlhd^\mathfrak{s}(h):\mathfrak{e}(\mathbb{W})\leq C\}$ is closed in $\mathbb{G}\times \G_\unlhd^\mathfrak{s}(h)$ and in particular, the set
\begin{equation}
    \begin{split}
        \mathcal{M}^{-1}(0)\cap \mathbb{G}\times \{\mathbb{W}\in\G_\unlhd^\mathfrak{s}(h):\mathfrak{e}(\mathbb{W})\leq C\}
        =\{(x,\mathbb V)\in \mathbb G\times \G_\unlhd^{\mathfrak s}(h)\,\,\text{such that}\,\,\mathcal{M}(x,\mathbb V)=0\,\,\text{and}\,\,\mathfrak e(\mathbb V)\leq C\},
    \end{split}
\end{equation}
is Borel. Now, since the projection on the first component of the above set is an analytic set, by the very definition of analytic sets, and since every analytic set is universally measurable, see for example \cite[Section 2.2.4]{KarouiTan13}, we get that the set $\{x\in\mathbb G\,\,\text{such that there exists $\mathbb V\in\G^{\mathfrak s}_\unlhd(h)$ with $\mathcal{M}(x,\mathbb V)=0$ and $\mathfrak e(\mathbb V)\leq C$}\}$ is $\phi$-measurable. In particular its complement, that is $\mathscr G_C$ up to $\phi$-null sets - since $\mathcal{M}(x,\mathbb V)=0$ if and only if $\mathbb V\in\mathscr G(x)$ for $\phi$-almost every $x\in\mathbb G$ - is $\phi$-measurable as well.
\end{proof}

\begin{proposizione}\label{prop::5.2satisfied}
Let  $h\in\{1,\ldots,Q\}$, $\mathfrak{s}\in\mathfrak{S}(h)$,
and $\phi$ be a $\mathscr{P}_h^{*,\unlhd}$-rectifiable measure supported on a compact set $K$ and for which for $\phi$-almost every $x\in\mathbb{G}$ we have
\begin{equation}\label{eqn:UNLHD}
\mathrm{Tan}_h(\phi,x)\subseteq \{\lambda\mathcal{S}^h\llcorner \mathbb{V}:\lambda>0\text{ and }\mathbb{V}\in\G^{\mathfrak s}_\unlhd(h)\}.
\end{equation}
Let us further assume that there exists a constant $C>0$ such that $\phi(\mathbb G\setminus \mathscr G_C)=0$, where $ \mathscr G_C$ is defined in \eqref{eqn:Gc}.
Throughout the rest of the statement and the proof we will always assume that $\oldC{C:proj}$ and $\oldC{C:U}$  are the constants introduced in \cref{defC6C7} in terms of $C$. Furthermore, let $\varepsilon\in (0,5^{-10(h+5)}\oldC{C:U}^{-3h}]$ and $\mu:=2^{-7}h^{-3}\oldC{C:U}^{-5h}\varepsilon^2$. 

Then, there are $\vartheta,\gamma\in\N$, a $\phi$-positive compact subset $E$ of $E(\vartheta,\gamma)$, and a point $z\in E\cap\mathscr G_C$ such that 
\begin{itemize}
    \item[(i)] There exists a $\rho_z>0$ for which $\phi(B(z,\oldC{C:U}\rho)\setminus E)\leq \mu^{h+1}\oldC{C:U}^{-h}\phi(B(z,\oldC{C:U}\rho))$ for any $0<\rho<\rho_z$;
    \item[(ii)] There exists an $r_0\in(0,5^{-10(h+5)}\oldC{C:U}^{-3h}\gamma^{-1}]$ such that for any $w\in E$ and any $0<\rho\leq\oldC{C:U}r_0$ we can find a $\mathbb{V}_{w,\rho}\in \G_\unlhd^{\mathfrak s}(h)$ such that $\mathfrak e(\mathbb V_{w,\rho})\geq C$, see \eqref{eqn:mathfrake}, and
    \begin{enumerate}
    \item $F_{w,4\oldC{C:U}\rho}(\phi,\Theta\mathcal{C}^h\llcorner w\mathbb{V}_{w,\rho})\leq (4^{-1}\vartheta^{-1}\oldC{C:U}^{-1}\mu)^{(h+3)}\cdot(4\oldC{C:U}\rho)^{h+1}$ for some $\Theta>0$,
    \item whenever $y\in B(w,\oldC{C:U}\rho)\cap w\mathbb{V}_{w,\rho}$ and $t\in [\mu \rho, \oldC{C:U}\rho]$ we have
$\phi(B(y,t))\geq (1-\varepsilon)(t/\oldC{C:U}\rho)^h\phi(B(w,\oldC{C:U}\rho))$,
\item There exists a normal complement $\mathbb L_{w,\rho}$ of $\mathbb V_{w,\rho}$ as in \cref{lip:const:proj:conorm}  such that
$$
(1-\varepsilon)\phi(B(w,\oldC{C:U}\rho)\cap w T_{\mathbb{V}_{w,\rho}}(0,\rho))
        \leq\oldC{C:U}^{-h}\mathcal{C}^h(P(B(0,1)))\phi(B(w,\oldC{C:U}\rho)),
$$
where $T_{\mathbb V_{w,\rho}}$ is the cylinder related to the splitting $\mathbb G=\mathbb V_{w,\rho}\cdot\mathbb L_{w,\rho}$, see \cref{def:CYLINDER};
\end{enumerate}
    \item[(iii)] There exists an infinitesimal sequence $\{{\rho_i(z)}\}_{i\in\N}\subseteq (0,\min\{r_0,\rho_z\}]$ such that for any $i\in\N$, any $w\in E$ and any $\rho\in(0,\oldC{C:U}{\rho_i(z)}]$ we have
    $\phi(B(w,\oldC{C:U}\rho))\geq (1-\varepsilon)(\rho/{\rho_i(z)})^h\phi(B(z,\oldC{C:U}\rho_i(z)))$.
    \end{itemize}
\end{proposizione}

\begin{proof}
For any positive $a,b\in\R$ we define ${F}(a,b)$ to be the set of those points in $K$ for which
$$
br^h\leq\phi(B(x,r)),\qquad\text{for any }r\in (0,a).
$$
One can prove, with the same argument used in the proof of \cref{prop:cpt}, see \cite[Proposition 1.14]{MarstrandMattila20}, that the sets ${F}(a,b)$ are compact. As a consequence, this implies that the sets
$$
\widetilde{{F}}(a,b):=\bigcap_{p=1}^\infty {F}(\oldC{C:U}a,(1-\varepsilon)b)\setminus F(\oldC{C:U}a/p,b),
$$
are Borel. Since $\phi$ is $\mathscr{P}^*_h$-rectifiable, $\mathbb{G}$ can be covered $\phi$-almost all by countably many sets $\widetilde{{F}}(a,b)$. Indeed, $\phi(\mathbb G\setminus\cup_{a,b\in \mathbb Q^+}\widetilde F(a,b))=0$ since $0<\Theta_*^h(\phi,x)<+\infty$ holds $\phi$-almost everywhere. In particular thanks to \cref{prop::E} we can find $a,b\in \R$ and $\vartheta,\gamma\in\N$ such that $\phi(\widetilde{{F}}(a,b)\cap E(\vartheta,\gamma))>0$. Since $\widetilde{{F}}(a,b)\cap E(\vartheta,\gamma)$ is measurable, there must exist a $\phi$-positive compact subset of  $\widetilde{{F}}(a,b)\cap E(\vartheta,\gamma)$ that we denote with $F$. Notice that since $\phi(\mathbb G\setminus \mathscr G_C)=0$ the set $F\cap\mathscr G_C$ is measurable and $\phi$-positive as well.

Let us denote by $\G^{\mathfrak s,C}_\unlhd(h)$ the set $\{\mathbb V\in \G^{\mathfrak s}_\unlhd(h)\,\,\text{such that}\,\,\mathfrak e(\mathbb V)\geq C\}$. Since by the very definition of $\mathscr G_C$ we have $\mathrm{Tan}_h(\phi,x)\subseteq \mathfrak{M}(h,\G^{\mathfrak s,C}_\unlhd(h))$ for $\phi$-almost every $x\in F\cap\mathscr G_C$, we infer that \cref{prop:TanAndDxr} together with Severini-Egoroff theorem, that can be applied since the functions $x\to d_{x,kr}(\phi,\mathfrak M(h,\G^{\mathfrak s,C}_\unlhd(h)))$ are continuous in $x$ for every $k,r>0$ - see \cref{rem:dxrContinuous} - yield a $\phi$-positive compact subset ${E}$ of $F\cap\mathscr G_C$ and an $r_0\leq5^{-10(h+5)}\oldC{C:U}^{-3h}\gamma^{-1}$ such that
 \begin{equation}\label{eq:::num4}
 \begin{split}
     d_{x,4\oldC{C:U}\rho}(\phi,\mathfrak{M}(h,\G^{\mathfrak s,C}_\unlhd(h)))\leq (4^{-1}\vartheta^{-1}\oldC{C:U}^{-1}\mu)^{(h+4)}\text{ for any }x\in E\text{ and any }0<\rho\leq \oldC{C:U}r_0.
 \end{split}
       \end{equation}
Let us fix $z$ to be a density point of $E$ with respect to $\phi$, and let us show that $E$ and $z$ satisfy the requirements of the proposition. First, by construction $E$ is $\phi$-positive and contained in $E(\vartheta,\gamma)$. Second, since $z$ is a density point of $E$, item (i) follows if we choose $\rho_z$ small enough. Moreover, the bound \eqref{eq:::num4} directly implies item (ii.1). Let us prove the remaining items.
 
Since $E\subseteq E(\vartheta,\gamma)$, $4\oldC{C:U}^2r_0<\gamma/2$ and $4^{-1}\vartheta^{-1}\oldC{C:U}^{-1}\mu\leq 2^{-10(h+1)}\vartheta$, \cref{prop::4.4(4)}(i) implies that for any $w\in E$ and any $0<\rho<\oldC{C:U}r_0$ - choosing $\sigma=4^{-1}\vartheta^{-1}\oldC{C:U}^{-1}\mu$ and $t=4\oldC{C:U}\rho$ in \cref{prop::4.4(4)} - there exists a $\mathbb{V}_{w,\rho}\in \G^{\mathfrak s,C}_\unlhd(h)$ such that
$$
\phi(B(y,r)\cap B(w\mathbb{V},4^{-1}\oldC{C:U}^{-1}\vartheta^{-2}\mu^2\rho))\geq (1-2^{10(h+1)}4^{-1}\oldC{C:U}^{-1}\mu)(r/s)^h\phi(B(v,s)),
$$
whenever $y,v\in B(w,2\oldC{C:U}\rho)\cap w\mathbb{V}_{w,\rho}$ and $\vartheta^{-1}\mu\rho\leq r,s\leq 2\oldC{C:U}\rho$. Since $$2^{10(h+1)}4^{-1}\oldC{C:U}^{-1}\mu\leq \varepsilon,$$ with the choices $s=\oldC{C:U}\rho$ and $v=w$, we finally infer
$$
\phi(B(y,r))\geq (1-\varepsilon)(r/\oldC{C:U}\rho)^h\phi(B(w,\oldC{C:U}\rho)),
$$
for any $\mu \rho\leq r\leq \oldC{C:U}\rho$ and any $y\in B(w,\oldC{C:U}\rho)\cap w\mathbb V_{w,\rho}$, and this proves item (ii.2). For any $w\in E$ and any $0<\rho<\oldC{C:U}r_0$ we choose one normal complement $\mathbb L_{w,\rho}$ of $\mathbb V_{w,\rho}$ as in \cref{lip:const:proj:conorm}, and we denote with $P:=P_{\mathbb V_{w,\rho}}$ the projection relative to this splitting. 
Eventually, \cref{prop::4.4(4)}(ii), with the choice $k:=\oldC{C:U}$, 
implies that for any $0<\rho<\oldC{C:U}r_0$  we have
\begin{equation}
\begin{split}
        \phi(B(w,\oldC{C:U}\rho)\cap w T_{\mathbb V_{w,\rho}}(0,\rho))
        &\leq (1+(2\oldC{C:U} h+1)\vartheta^{-1}\oldC{C:U}^{-1}\mu)\oldC{C:U}^{-h}\mathcal{C}^h(P(B(0,1)))\phi(B(w,\oldC{C:U}\rho))\\
        &\leq (1+\varepsilon)\oldC{C:U}^{-h}\mathcal{C}^h(P(B(0,1)))\phi(B(w,\oldC{C:U}\rho)),
        \label{eq:::num5}
\end{split}
\end{equation}
where the last inequality comes from the fact that $(2\oldC{C:U} h+1)\vartheta^{-1}\oldC{C:U}^{-1}\mu<\varepsilon$. Hence also item (ii.3) is verified. In order to verify item (iii), note that since $z\in E\subseteq \widetilde{F}(a,b)$ on the one hand then there is an infinitesimal sequence $\{\rho_i(z)\}_{i\in\N}$ such that
\begin{equation}
    \frac{\phi(B(z,\oldC{C:U}\rho_i(z)))}{(\oldC{C:U}\rho_i(z))^h}\leq b.
    \label{eq:Eab1}
\end{equation}
On the other hand for any $w\in E$, and any $0<\rho<a$ we have
\begin{equation}
   b\leq \frac{1}{1-\varepsilon}\frac{\phi(B(w,\oldC{C:U}\rho))}{(\oldC{C:U}\rho)^h}.
   \label{eq:Eab2}
\end{equation}
Putting together \eqref{eq:Eab1} and \eqref{eq:Eab2} we finally infer that for any $i\in\N$, any $w\in E$ and any $\rho\in(0,a)$ we have
$$
\frac{\phi(B(z,\oldC{C:U}\rho_i(z)))}{\rho_i(z)^h}\leq \frac{1}{1-\varepsilon}\frac{\phi(B(w,\oldC{C:U}\rho))}{\rho^h},
$$
concluding the proof of item (iii) and thus of the proposition.
\end{proof}

Before going on with the last part of the proof of \cref{thm:MMconormale}, we need a Proposition borrowed from the Preprint \cite{antonelli2020rectifiable} whose simple proof is omitted here. Before stating it, we give a couple of definitions. 

\begin{definizione}\label{def:Pixr}
Let us fix $x\in \mathbb G$, $r>0$ and $\phi$ a Radon measure on $\mathbb G$.
We define $\Pi_{\delta}(x,r)$ to be the subset of planes $\mathbb V\in\G(h)$ for which there exists a $\Theta>0$ such that
\begin{equation}\label{eqn:Previous}
F_{x,r}(\phi, \Theta \mathcal{S}^{h}\llcorner x\mathbb V)\leq 2\delta r^{h+1}.
\end{equation}
\end{definizione}

\begin{definizione}\label{eqn:DefinitionOfDeltaG}
For any $\vartheta\in\mathbb N$ we define $\delta_\mathbb{G}=\delta_{\mathbb G}(h,\vartheta):=\vartheta^{-1}2^{-(4h+5)}$.
\end{definizione}

\begin{proposizione}[{\cite[Proposition 3.1]{antonelli2020rectifiable}}]\label{prop1vsinfty}
Let $\phi$ be a Radon measure on $\mathbb G$ that is supported on a compact set, let $x\in E(\vartheta,\gamma)$, 
fix $\delta<\delta_\mathbb{G}$, where $\delta_{\mathbb G}$ is defined in \cref{eqn:DefinitionOfDeltaG}, and set $0<r<1/\gamma$. Then for every $\mathbb V\in \Pi_{\delta}(x,r)$, see \cref{def:Pixr}, we have
\begin{equation}
\sup_{w\in E(\vartheta,\gamma)\cap B(x,r/4)} \frac{\dist\big(w,x\mathbb V\big)}{r}\leq2^{1+1/(h+1)}\vartheta^{1/(h+1)}\delta^{1/(h+1)}=: \newC\label{C:b1}(\vartheta,h)\delta^{1/(h+1)}.
\end{equation}
\end{proposizione}

\begin{proposizione}\label{prop:PHIK>0}
Assume $\phi$ is a $\mathscr{P}_h^{*,\unlhd}$-rectifiable measure supported on a compact set $K$. Then, there exists a $\mathbb{W}\in \G_\unlhd(h)$, a compact set $K'\Subset \mathbb{W}$ and a Lipschitz function $f:K'\to \mathbb{G}$ such that $\phi(f(K'))>0$.
\end{proposizione}

\begin{proof}
\cref{struct:strat} implies that for $\phi$-almost every $x\in \mathbb{G}$ the elements of $\Tan_h(\phi,x)$ all share the same stratification vector. Furthermore, thanks to \cref{propmeasuS}, for any $\mathfrak{s}\in \mathfrak{S}(h)$ the set
$\mathscr{T}_\mathfrak{s}:=\{x\in K: \mathfrak{s}(\phi,x)=\mathfrak{s}\}$ is $\phi$-measurable. Thus, if we prove that for any $\mathfrak{s}\in \mathfrak{S}(h)$ there exists a Lipschitz function as in the thesis of the proposition whose image has positive $\phi\llcorner \mathscr{T}_\mathfrak{s}$-measure, the proposition is proved since the sets $\mathscr{T}_\mathfrak{s}$ cover $\phi$-almost all $K$ and since the locality of tangents hold, see \cref{prop:LocalityOfTangent}. Thanks to this argument, we can assume without loss of generality that there exists a $\mathfrak{s}\in\mathfrak{S}(h)$ such that for $\phi$-almost every $x\in K$ we have $\mathfrak{s}(\phi,x)=\mathfrak{s}$. 

Let us further reduce ourselves to the setting in which there exists a constant $C>0$ such that $\phi(\mathbb G\setminus \mathscr G_C)=0$, where $\mathscr G_C$ is defined in \eqref{eqn:Gc}. Thanks to \cref{lemma:Gx}, we know that for $\phi$-almost every $x\in\mathbb G$ the set $\mathscr{G}(x)$ defined in \eqref{eqn:DefinitionGx} is compact. Hence, taking item (i) of \cref{prop:grasscompiffC3} into account, for $\phi$-almost every $x\in\mathbb G$ there exists a constant $C(x)>0$ such that $\mathfrak e(\mathbb V)\geq C(x)$ for every $\mathbb V\in \mathscr{G}(x)$. This readily implies that 
$$
\phi(\mathbb G\setminus \cup_{n\in\mathbb N}\mathscr G_ {1/n})=0.
$$
Hence, since $\mathscr G_{1/n}$ is $\phi$-measurable for every $n\in\mathbb N$, see \cref{lemma:Gx}, we can reduce, with the same argument used in the previous paragraph, to deal with the case in which there exists $C>0$ such that $\phi(\mathbb G\setminus \mathscr G_C)=0$. 

Let $\oldC{C:proj}:=\oldC{C:proj}(C)$ and $\oldC{C:U}:=\oldC{C:U}(C)$ be defined as in \cref{defC6C7}, and let $\widetilde\varepsilon\leq 5^{-10(h+5)}\oldC{C:U}^{-3h}$, and $\widetilde\mu:=2^{-7}h^{-3}\oldC{C:U}^{-5h}\widetilde\varepsilon^2$. Let $E\subseteq K$ be the compact set and $z\in E\cap \mathscr G_C$ the point yielded by \cref{prop::5.2satisfied} with respect to $\widetilde\varepsilon,\widetilde\mu$. Furthermore let $\widetilde\varepsilon\leq\varepsilon\leq 5^{-h-5}\oldC{C:U}^{-3h}$, and $\mu:=2^{-7}h^{-3}\oldC{C:U}^{-5h}\varepsilon^2$ such that $(1-\widetilde\varepsilon)^2\geq (1-\varepsilon)$. We define
$$
r:=\rho_1(z),\qquad \text{and}\qquad r_1:=(1-\varepsilon/h)r,$$
where $\rho_1(z)$ is the first term of the sequence $\{\rho_i(z)\}_{i\in\N}$ yielded by item (iii) of \cref{prop::5.2satisfied}.

Let us check that the compact set $E\cap B(z,\oldC{C:U}r_1)$ satisfies the hypothesis of \cref{prop::5.2} with respect to the choiches $\varepsilon,\mu,r$. First of all, since $r<\rho_z$, item (i) of \cref{prop::5.2satisfied} implies that \eqref{eq:densityE} holds since $\widetilde\mu\leq \mu$. Secondly, since $r\leq r_0$, item (ii.2) of \cref{prop::5.2satisfied} implies that for any $w\in E$ and any $0<\rho<\oldC{C:U}r$ there exists a $\mathbb{V}_{w,\rho}\in \G_\unlhd^\mathfrak{s}(h)$ such that whenever $y\in B(w,\oldC{C:U}r)\cap w\mathbb{V}_{w,\rho}$ and $t\in [\mu \rho, \oldC{C:U}\rho]$ we have
$$
\phi(B(y,t))\geq (1-\widetilde\varepsilon)(t/\oldC{C:U}\rho)^h\phi(B(w,\oldC{C:U}\rho)).
$$
Furthermore, since $r=\rho_1(z)$, thanks to item (iii) of \cref{prop::5.2satisfied} we finally infer that for any $w\in E$ and any $0<\rho<\oldC{C:U}r$ we have
\begin{equation}
    \begin{split}
\phi(B(y,t))&\geq (1-\widetilde\varepsilon)(t/\oldC{C:U}\rho)^h\phi(B(w,\oldC{C:U}\rho))\geq (1-\widetilde\varepsilon)^2(t/\oldC{C:U}r)^h\phi(B(z,\oldC{C:U}r))\\
&\geq (1-\varepsilon)(t/\oldC{C:U}r)^h\phi(B(z,\oldC{C:U}r)),        
    \end{split}
\end{equation}
whenever $y\in B(w,\oldC{C:U}r)\cap w\mathbb{V}_{w,\rho}$ and $t\in [\mu \rho, \oldC{C:U}\rho]$. 
The above paragraph shows that the hypotheses of \cref{prop::5.2} are satisfied by $z$ and $E\cap B(z,\oldC{C:U}r_1)$ with the choices of $r,r_1,\varepsilon,\mu$ as above. 

\textbf{Throughout the rest of the proof $E$ will stand for $E\cap B(z,\oldC{C:U}r_1)$}, and in order to conclude the argument we will need to use the other two pieces of information yielded by \cref{prop::5.2satisfied}. Indeed, since $r< \oldC{C:U}r_0$, item (ii.3) of \cref{prop::5.2satisfied} implies that
\begin{equation}
    (1-\varepsilon)\phi(zT_{\mathbb{V}_{z,r}}(0,r)\cap B(z,\oldC{C:U}r))\leq \mathcal{C}^h(P(B(0,1))) \oldC{C:U}^{-h}\phi(B(z,\oldC{C:U}r)),
    \label{reg:0:cil}
\end{equation}
where $T$ is the cylinder related to the splitting $\mathbb G=\mathbb V_{z,r}\cdot \mathbb L_{z,r}$, and $\mathbb L_{z,r}$ is one normal complement to $\mathbb V_{z,r}$ chosen as in item (ii.3) of \cref{prop::5.2satisfied}.
Furthermore, thanks to item (ii.1) of \cref{prop::5.2satisfied} and the fact that $r<r_0$ we know that there exists $\Theta>0$ such that
\begin{equation}
   F_{z,4\oldC{C:U} r}(\phi,\Theta\mathcal{C}^h\llcorner z\mathbb{V}_{z,r})\leq (4^{-1}\vartheta^{-1}\oldC{C:U}^{-1}\mu)^{h+3}\cdot(4\oldC{C:U}r)^{h+1}.
    \label{eq::1vsinftye}
\end{equation}
The bound \eqref{eq::1vsinftye} together with \cref{prop1vsinfty}, that we can apply since $4\oldC{C:U}r\leq \gamma^{-1}$, and $2^{-1}(4^{-1}\vartheta^{-1}\oldC{C:U}^{-1}\mu)^{h+3}\leq \delta_{\mathbb{G}}$, where $\delta_\mathbb{G}$ was introduced in \cref{eqn:DefinitionOfDeltaG},  imply that
\begin{equation}
    \sup_{w\in E\cap B(z,\oldC{C:U} r)} \frac{\dist\big(w,z\mathbb{V}_{z,r}\big)}{4\oldC{C:U}r}\leq2^{1+1/(h+1)}\vartheta^{1/(h+1)}(2^{-1}(4^{-1}\vartheta^{-1}\oldC{C:U}^{-1}\mu)^{h+3})^{\frac{1}{h+1}}\leq 2\oldC{C:U}^{-1}\mu.
    \label{eq:::num1}
\end{equation}
The above bound shows that the set $E$ inside the ball $B(z,\oldC{C:U}r)$ is very squeezed around the plane $\mathbb{V}_{z,r}$. From now on we should denote $\mathbb{W}:=\mathbb{V}_{z,r}$, $P:=P_{\mathbb{V}_{z,r}}$, $\mathbb L:=\mathbb L_{z,r}$, and $T(\cdot,r):=T_\mathbb{W}(\cdot,r)$. In order to simplify the notation, since all the statements are invariant up to substituting $\phi$ with $T_{z,1}\phi$, we can assume that $z=0$. Let us recall once more that $\mathfrak e(\mathbb V_{z,r})\geq C$ from item (ii) of \cref{prop::5.2satisfied}.

Since it will turn out to be useful later on, we estimate the distance of the points $w$ of $E\cap T(0,r_1)$ from $0$. Thanks to \cref{prop:structurecyl} and the fact that $w\in T(0,r_1)$, we have $\lVert P_\mathbb{W}(w)\rVert\leq \oldC{C:proj}r_1$. On the other hand, \eqref{eq:::num2} and \eqref{eq:::num1} imply that
$$
\lVert P_\mathbb{L}(w)\rVert\leq \oldC{C:proj}\dist(w,\mathbb{W})\leq 8\oldC{C:proj} \mu r_1.
$$
This in particular implies that
$$
d(0,w)\leq \lVert P_\mathbb{W}(w)\rVert+\lVert P_\mathbb{L}(w)\rVert\leq \oldC{C:proj}r_1+ 8\oldC{C:proj} \mu r_1\leq 2\oldC{C:proj}r_1,
$$
showing that
\begin{equation}
E\cap T(0,r_1)\subseteq B(0,2\oldC{C:proj}r_1).
\label{eq:::num3}
\end{equation}

In the following $A$, $A_1$ and $A_2$ are the sets inside $P(B(0,r_1))$ constructed in the statement of \cref{prop::5.2} with respect to the $0$ and the plane $\mathbb{W}$. Now, let $\widetilde{A}$ be the set of those $u\in A$ for which there exists $\rho(u)>0$ such that
\begin{equation}
    \phi(B(0,\oldC{C:U}r)\cap T(u,s))\leq 2(1-\varepsilon)^4(s/\oldC{C:U}r)^h\mathcal{C}^h(P(B(0,1)))\phi(B(0,\oldC{C:U}r)),
    \label{eq:num1uno}
\end{equation}
for all $0<s<\rho(u)$. We claim that $\widetilde{A}$ is a Borel set. To prove this, we note that
$$
\widetilde{A}=\bigcup_{k\in\N}\{u\in A: \eqref{eq:num1uno}\text{ holds for any }0<s<1/k\}=:\bigcup_{k\in\N}\widetilde{A}_k.
$$
Let us show that $\widetilde A_k$ is a compact set for any $k\in\N$, and in order to do this, let us assume $\{u_i\}_{i\in\N}$ is a sequence of points of $\widetilde A_k$. Since $\widetilde A_k\subseteq A$, and $A$ is compact, we can suppose that, up to a non re-labelled subsequence, $u_i$  converges to some $u\in A$. Thus, we have that for every $0<s<1/k$ the following chain of inequality holds
\begin{equation}
    \begin{split}
        \phi(B(0,\oldC{C:U} r)\cap T(u, s))&\leq  \limsup_{i\to\infty}\phi(B(0,\oldC{C:U} r)\cap T(u_i,s+d(u,u_i)))\\
        &\leq 2(1-\varepsilon)^4\mathcal{C}^h(P(B(0,1)))(s/\oldC{C:U} r)^h\phi(B(0,\oldC{C:U}r)).    \nonumber
    \end{split}
\end{equation}
This concludes the proof of the fact that $\widetilde{A}_k$ is compact and thus $\widetilde{A}$ is an $F_\sigma$ set, and thus Borel. 

Let us notice that, since $r_1<r$, by a compactness argument one finds that there exists a $\widetilde s:=\widetilde s(r_1,r)$ such that whenever $u\in P(B(0,r_1))$, then $P(B(u,\widetilde s))\subseteq P(B(0,r))$. The family
$$\mathcal{B}:=\{P(B(u,s)):u\in A\setminus \widetilde A,\,\,\text{and }s\leq \widetilde s\text{ does not satisfy \eqref{eq:num1uno}}\}$$ is a fine cover of $A\setminus \widetilde A$ by the very definition of $\widetilde A$. Thus \cite[2.8.17]{Federer1996GeometricTheory} with a routine argument
implies that $\mathcal{B}$ is a $\mathcal{S}^h\llcorner(A\setminus\widetilde A)$-Vitali relation (\cite[2.8.16]{Federer1996GeometricTheory}). Therefore, the set $A\setminus \widetilde{A}$ can be covered $\mathcal{S}^h$-almost all by a sequence of disjointed projected balls $\{P(B(u_k,s_k))\}_{k\in\N}$ such that $u_k\in A\setminus \widetilde{A}$ and
$$
\phi(B(0,\oldC{C:U}r)\cap T(u_k, s_k))> 2(1-\varepsilon)^4\mathcal{C}^h(P(B(0,1)))(s_k/\oldC{C:U}r)^h\phi(B(0,\oldC{C:U}r)),$$
for every $k\in\N$. Note that since $T(u_k, s_k)= P^{-1}(P(B(u_k,s_k)))$, see \cref{prop:structurecyl}, we get that $\{T(u_k, s_k)\}_{k\in\mathbb N}$ is a disjointed family of cylinders. Moreover, from the very definition of $\widetilde s$, since $u_k\in P(B(0,r_1))$ and $s_k\leq \widetilde s$, we have that $P(B(u_k,s_k))\subseteq P(B(0,r))$. This implies that
\begin{equation}
\begin{split}
    \phi(T(0,r)\cap B(0,\oldC{C:U}r))&\geq\sum_{k\in\mathbb N} \phi(B(0,\oldC{C:U}r)\cap T(u_k, s_k))\\
    &> 2(1-\varepsilon)^4\mathcal{C}^h(P(B(0,1))) \oldC{C:U}^{-h}r^{-h}\phi(B(0,\oldC{C:U}r))\sum_{k\in\N} s^h_k.
\end{split}
\end{equation}
Therefore, we have
\begin{equation}
    \begin{split}
\mathcal{C}^h(A\setminus \widetilde{A})&=\sum_{k\in\N}\mathcal{C}^h(P(B(u_k,s_k)))\leq \mathcal{C}^h(P(B(0,1)))\sum_{k\in\N} s^h_k\\
&< 2^{-1}(1-\varepsilon)^{-4}\frac{\phi(T(0,r)\cap B(0,\oldC{C:U}r))\oldC{C:U}^{h}r^h}{\phi(B(0,\oldC{C:U}r))} \leq 2^{-1}(1-\varepsilon)^{-5}\mathcal{C}^h(P(B(0,1)))r^h\\
&\leq \frac{27}{50}\mathcal{C}^h(P(B(0,1)))r^h,
\nonumber
    \end{split}
\end{equation}
where the second inequality on the second line above follows from \eqref{reg:0:cil}. Furthermore, from the previous inequality and from item (iv) of \cref{prop::5.2} we deduce that
\begin{equation}
\begin{split}
    \mathcal{C}^h(\widetilde{A})&=\mathcal{C}^h(P( B(0,r)))-\mathcal{C}^h(P(B(0,r))\setminus A)-\mathcal{C}^h(A\setminus\widetilde{A})\\
    &>\mathcal{C}^h(P(B(0,1)))r^h -5^{h+3}\oldC{C:U}^{3h}\varepsilon \mathcal{C}^h(P(B(0,1)))r^h-\mathcal{C}^h(P(B(0,1)))\frac{27}{50}r^h\\
    &\geq(1-1/25-27/50)\mathcal{C}^h(P(B(0,1)))r^h>\frac{2}{5}\mathcal{C}^h(P(B(0,1)))r^h.
\end{split}
    \nonumber
\end{equation}
Since $\widetilde{A}$ is measurable, we can find a compact set $\hat{A}\subseteq \widetilde A$ and a $\delta\in (0,\varepsilon r/h)$ such that $\mathcal{C}^h(\hat{A})>0$ and \eqref{eq:num1uno} holds for any $u\in \hat A$ and $s\in (0,\delta)$.
This can be done by taking an interior approximation with compact sets of $\widetilde A$.

Thanks to item (v) of \cref{prop::5.2} we know that
\begin{equation}
    \hat{A}\subseteq A= P(E\cap P^{-1}(A)),
    \label{eq:dundee}
\end{equation}
and thus for any $u\in \hat{A}$ we can find a $x\in E$ such that $P(x)=u$. We claim that for any $x\in E$ for which $P(x)\in \hat{A}$, any $s<\min\{r/4,\delta/(1+\oldC{C:U})\}$ and any $w\in \mathbb{V}_{x,s}$ we have
\begin{equation}
\lVert P(w)\rVert>\lVert w\rVert/2\oldC{C:U}.
    \label{eq:::num7}
\end{equation}
Suppose by contradiction that there are an $s<\min\{r/4,\delta/(1+\oldC{C:U})\}$ and a $w\in \mathbb{V}_{x,s}$ with $\lVert w\rVert=1$ such that $\lVert P(w)\rVert\leq {1/2\oldC{C:U}}$. This would imply that for any  $k=0,\ldots,\lfloor\oldC{C:U}/4\rfloor-1$ and any $p\in B(0,s/2)$ we have, by exploiting $P(x)=u$ and that $P$ is a homogeneous homomorphism, that
\begin{equation}
    \begin{split}
        d(P(x\delta_{2ks}(w)p),u)&=d(\delta_{2ks}(P(w))P(p),0)\leq \rVert\delta_{2ks}(P(w))\rVert+\lVert P(p)\rVert\\
        &\leq 2ks\lVert P(w)\rVert+\lVert P(p)\rVert\leq ks/\oldC{C:U}+\oldC{C:proj}s\leq (1+\oldC{C:proj})s.
        \label{godot:1}
    \end{split}
\end{equation}
Since $u\in \hat A\subseteq A\subseteq P(B(0,r_1))$, and since $P(x)=u$, we conclude that $x\in T(0,r_1)$. Hence, taking into account that $r_1<r$, thanks to the inclusion \eqref{eq:::num3}, we have
\begin{equation}
    \begin{split}
        d(x\delta_{2ks}(w)p,0)\leq     \lVert x\rVert+2ks+s\leq 2\oldC{C:proj}r+(2k+1)s<2\oldC{C:proj}r+3\oldC{C:U}r/4<\oldC{C:U}r.
        \label{godot:2}
    \end{split}
\end{equation}
Putting together \eqref{godot:1} and \eqref{godot:2}, we infer that for any $k=0,\ldots, \lfloor\oldC{C:U}/4\rfloor-1$ we have
$$
B(x\delta_{2ks}(w),s/2)\subseteq T(u,(1+\oldC{C:proj})s)\cap B(0,\oldC{C:U}r).
$$
Furthermore, since $x\in E$, $B(x\delta_{ks}(w),s/2)$ are disjoint and contained in $B(x,\oldC{C:U}s)$, we have by items (ii.2) and (iii) of \cref{prop::5.2satisfied} that
\begin{equation}
    \begin{split}
        \phi\Big(B(0,\oldC{C:U}r)\cap T(u,(1+\oldC{C:proj})s)\Big)&\geq \sum_{k=1}^{\lfloor\oldC{C:U}/4\rfloor-1}\phi(B(x\delta_{ks}(w),s/2))\\
        &\geq \frac{(1-\varepsilon)\oldC{C:U}}{8}\Big(\frac{s/2}{\oldC{C:U}s}\Big)^h\phi(B(x,\oldC{C:U}s))
        \geq \frac{(1-\varepsilon)^2\oldC{C:U}}{8}\Big(\frac{s/2}{\oldC{C:U}r}\Big)^h\phi(B(0,\oldC{C:U}r))\\
        &=(1-\varepsilon)^2\frac{\oldC{C:U}}{2^{h+3}}\Big(\frac{s}{\oldC{C:U}r}\Big)^h\phi(B(0,\oldC{C:U}r)).
        \label{eq:::num6}
    \end{split}
\end{equation}
Since by assumption $u\in \hat{A}\subseteq \widetilde A$ and $(1+\oldC{C:proj})s<\delta$, we infer thanks to \eqref{eq:::num6} and the definition of $\hat{A}$ that
\begin{equation}
    \begin{split}
        (1-\varepsilon)^2\frac{\oldC{C:U}}{2^{h+3}}\Big(\frac{s}{\oldC{C:U}r}\Big)^h\phi(B(0,\oldC{C:U}r))&\leq \phi\Big(B(0,\oldC{C:U}r)\cap T(u,(1+\oldC{C:proj})s)\Big)\\
        &\leq 2(1-\varepsilon)^4(1+\oldC{C:proj})^h\Big(\frac{s}{\oldC{C:U}r}\Big)^h\mathcal{C}^h(P(B(0,1)))\phi(B(0,\oldC{C:U}r)).
        \label{eq:::num7prim}
    \end{split}
\end{equation}
The chain of inequalities \eqref{eq:::num7prim} is however in contradiction with the choice of $\oldC{C:U}$ thanks to \cref{rk:contradiction5.2}, and thus the claim \eqref{eq:::num7} is proved.

Since $P$ restricted to $E\cap P^{-1}(A)$ is surjective on $\hat{A}$ as remarked in \eqref{eq:dundee}, thanks to the axiom of choice there exists a function $f:\hat{A}\to E\cap P^{-1}(A)$ such that $P(f(u))=u$. We claim that for $\phi$-almost every $x\in f(\hat{A})$ there exists a $\mathfrak{r}(x)>0$ such that for any $y\in f(\hat{A})\cap B(x,\mathfrak{r}(x))$ we have
\begin{equation}
\lVert P(x)^{-1}P(y)\rVert=\lVert P(x^{-1}y)\rVert>\oldC{C:U}^{-2}\lVert x^{-1}y\rVert=\oldC{C:U}^{-2}\lVert f(P(x))^{-1}f(P(y))\rVert,
\label{eq:::num16}
\end{equation}
where the last identity comes from the fact that $f$ is bijective on its image and thus the left and right inverse must coincide. In order to prove the latter claim, assume by contradiction that there exists an $x\in f(\hat{A})$ such that $\mathrm{Tan}_h(\phi,x)\subseteq \mathfrak{M}(h,\G^{\mathfrak s}_\unlhd(h))$ and a sequence $\{y_i\}_{i\in\N}\subseteq f(\hat{A})$, with $y_i\to x$, such that
\begin{equation}
    \lVert P(x^{-1}y_i)\rVert\leq \oldC{C:U}^{-2}\lVert x^{-1}y_i\rVert,\qquad \text{for any }i\in\N.
    \label{eq:numbero120}
\end{equation}
Defined $\rho_i:=\lVert x^{-1}y_i\rVert$, thanks to the hypothesis on $x$ and the definitions of $y_i$ and $\rho_i$ we can assume without loss of generality that
\begin{enumerate}
\item for any $i\in\N$ we have $\rho_i\leq \min\{r/4,\delta/(1+\oldC{C:U})\}$,
    \item the points $g_i:=\delta_{1/\rho_i}(x^{-1}y_i)$ converge to some $y\in \partial B(0,1)$ such that $\lVert P(y)\rVert\leq \oldC{C:U}^{-2}$,
    \item $\rho_i^{-h}T_{x,\rho_i}\phi\rightharpoonup \lambda \mathcal{C}^h\llcorner \mathbb{V}$ for some $\lambda>0$ and $\mathbb{V}\in\G^{\mathfrak s}_\unlhd(h)$.
\end{enumerate}
Since $\mathcal{C}^h\llcorner \mathbb{V}(\partial B(p,s))=0$, see e.g., \cite[Lemma 3.5]{JNGV20}, for any $p\in \mathbb{G}$ and any $s\geq 0$, thanks to \cite[Proposition 2.7]{DeLellis2008RectifiableMeasures} we infer that
\begin{equation}
    \begin{split}
        \lambda\mathcal{C}^h\llcorner \mathbb{V}(B(y,\rho))&=\lim_{i\to\infty}T_{x,\rho_i}\phi(B(y,\rho))/\rho_i^h\geq \lim_{i\to\infty} T_{x,\rho_i}\phi(B(g_i,\rho-d(g_i,y)))/\rho_i^h\\
        &\geq \lim_{i\to\infty}\phi(B(y_i,\rho_i\rho/2))/\rho_i^h\geq \vartheta^{-1}(\rho/2)^h>0,
        \nonumber
    \end{split}
\end{equation}
where we stress that in the second inequality in the second line we are using that there exists $\vartheta,\gamma\in\mathbb N$ such that $E\subseteq E(\vartheta,\gamma)$, since $E$ is provided by \cref{prop::5.2satisfied}.
The above computation shows that the (contradiction) assumption \eqref{eq:numbero120} implies that at $x$ there is a flat tangent measure whose support $\mathbb{V}$ contains an element $y\in\partial B(0,1)$ such that $\lVert P(y)\rVert\leq \oldC{C:U}^{-2}$. Let us prove that if
\begin{itemize}
    \item[(\hypertarget{HC}{HC})]there exists a suitably big $i_0\in\N$ such that we can find a $q_{i_0}\in \mathbb{V}_{x,\rho_{i_0}}$ such that $d(y,q_{i_0})\leq \mu$,
\end{itemize}
then we achieve a contradiction with \eqref{eq:::num7}, and thus we prove the claim \eqref{eq:::num16}. Indeed, the claim (\hyperlink{HC}{HC}) above would imply thanks to the definition of $\mu$, \eqref{eq:::num7}, \cref{lip:const:proj:conorm}, and \cref{prop:projhom}, that 
\begin{equation}
\begin{split}
(4\oldC{C:U})^{-1}&<(1-\mu)/2\oldC{C:U}\leq(\lVert y\rVert-\lVert y^{-1}q_{i_0}\rVert)/2\oldC{C:U}\leq\lVert q_{i_0}\rVert/2\oldC{C:U}\\
&< \lVert P(q_{i_0})\Vert\leq \lVert P(y)\rVert+\lVert P(y^{-1}q_{i_0})\rVert\leq \oldC{C:U}^{-2}+\oldC{C:proj}\mu  <2\oldC{C:U}^{-2},     
\end{split}
 \nonumber
\end{equation}
which is a contradiction since $\oldC{C:U}>10^Q$.

In this paragraph we prove the claim (\hyperlink{HC}{HC}), which is sufficient to conclude the proof of the claim \eqref{eq:::num16}. Let $\Theta_i$ be the positive numbers yielded by item (ii.1) of \cref{prop::5.2satisfied} with the choices $\rho:=\rho_i$ around the point $x$, and notice that
\begin{equation}
    \begin{split}
        \limsup_{i\to\infty}F_{0,4\oldC{C:U}}(\lambda\mathcal{C}^h\llcorner \mathbb{V},\Theta_i\mathcal{C}^h\llcorner \mathbb{V}_{x,\rho_i})&\leq \limsup_{i\to\infty}F_{0,4\oldC{C:U}}\Big(\frac{T_{x,\rho_i}\phi}{\rho_i^h},\lambda\mathcal{C}^h\llcorner \mathbb{V}\Big)+
        \limsup_{i\to\infty}F_{0,4\oldC{C:U}}\Big(\frac{T_{x,\rho_i}\phi}{\rho_i^h},\Theta_i\mathcal{C}^h\llcorner \mathbb{V}_{x,\rho_i}\Big)\\
        &=\limsup_{i\to\infty}F_{0,4\oldC{C:U}}\Big(\frac{T_{x,\rho_i}\phi}{\rho_i^h},\Theta_i\mathcal{C}^h\llcorner \mathbb{V}_{x,\rho_i}\Big)
        =\limsup_{i\to\infty}\frac{F_{x,4\oldC{C:U}\rho_i}\Big(\phi,\Theta_i\mathcal{C}^h\llcorner x\mathbb{V}_{x,\rho_i}\Big)}{\rho_i^{h+1}} \\ &\leq (\vartheta^{-1}\mu)^{(h+3)},
        \label{eq:::num8}
    \end{split}
\end{equation}
where the first identity in the second line above comes from \cref{prop:WeakConvergenceAndFk}, the second identity from the scaling property in \cref{rem:ScalinfFxr} and the last inequality from item (ii.1) of \cref{prop::5.2satisfied} and some algebraic computations that we omit. Defined $g(w):=(\min\{1,2-\|w\|\})_+$ by \cref{unif} for any $\mathbb{V}'\in\G(h)$ we have

\begin{equation}
    \begin{split}
        \int gd\mathcal{C}^h\llcorner \mathbb{V}'&=h\int s^{h-1} (\min\{1,2-\lvert s\rvert\})_+ds\\ &=h\left(\int_0^1 s^{h-1} + \int_1^2 s^{h-1}(2-s) ds \right)= \frac{2^{h+1}-1}{h+1}.
    \end{split}
\end{equation}
Therefore, since $\supp(g)\subseteq B(0,4\oldC{C:U})$, thanks to \eqref{eq:::num8} we infer that
\begin{equation}
    \begin{split}
        \limsup_{i\to\infty}\lvert \lambda-\Theta_i\rvert&= \limsup_{i\to\infty} \frac{\lvert \lambda\int gd\mathcal{C}^h\llcorner \mathbb{V}-\Theta_i\int gd\mathcal{C}^h\llcorner \mathbb{V}_{x,\rho_i}\rvert}{\int g d\mathcal{C}^h\llcorner \mathbb{V}} \\
        &\leq \limsup_{i\to\infty} (h+1)\frac{F_{0,4\oldC{C:U}}(\lambda\mathcal{C}^h\llcorner \mathbb{V},\Theta_i\mathcal{C}^h\llcorner \mathbb{V}_{x,\rho_i})}{2^{h+1}-1}
        \leq \frac{(h+1)(\vartheta^{-1}\mu)^{(h+3)}}{2^{h+1}-1}\\
        &\leq2(\vartheta^{-1}\mu)^{(h+3)}.
        \label{eq:::num11}
    \end{split}
\end{equation}
Let  $p\in\mathbb{V}\cap B(0,1)$ and  $\ell(w):=(\mu\|p\|-\|p^{-1}w\|)_+$. The function $\ell$ is a positive $1$-Lipschitz function whose support is contained in $B(0,(1+\mu)\lVert p\rVert)$ and therefore, thanks to \cref{rem:ScalinfFxr}, we deduce that
\begin{equation}
   \begin{split}
           \liminf_{i\to\infty}\,&\lambda\int \ell(w) d\mathcal{C}^h\llcorner\mathbb{V}_{x,\rho_i} \geq \lambda \int \ell(w)d\mathcal{C}^h\llcorner \mathbb{V}
           -\limsup_{i\to\infty}\lambda\bigg\lvert \int \ell(w)d\mathcal{C}^h\llcorner \mathbb{V}- \int \ell(w)d\mathcal{C}^h\llcorner \mathbb{V}_{x,\rho_i}\bigg\rvert\\
           &\geq \lambda \int \ell(w)d\mathcal{C}^h\llcorner \mathbb{V}-\limsup_{i\to\infty}\lvert \lambda-\Theta_i\rvert\int \ell(w)d\mathcal{C}^h\llcorner \mathbb V_{x,\rho_i} -\limsup_{i\to\infty}\bigg\lvert \int \ell(w)d\lambda\mathcal{C}^h\llcorner \mathbb{V}- \int \ell(w)d\Theta_i\mathcal{C}^h\llcorner \mathbb{V}_{x,\rho_i}\bigg\rvert\\
           &\geq\lambda \int \ell(w)d\mathcal{C}^h\llcorner \mathbb{V}-\limsup_{i\to\infty}\lvert \lambda-\Theta_i\rvert\int \ell(w)d\mathcal{C}^h\llcorner \mathbb V_{x,\rho_i} 
          -\limsup_{i\to\infty}F_{0,(1+\mu)\lVert p\rVert}(\lambda\mathcal{C}^h\llcorner \mathbb{V},\Theta_i\mathcal{C}^h\llcorner \mathbb{V}_{x,\rho_i}).
           \label{eq:::num12}
   \end{split}
\end{equation}
Let us bound separately the two last terms in the last line above. Thanks to the triangle inequality the points $q_i\in\mathbb{V}_{x,\rho_i}$ of minimal distance of $p$ from $\mathbb{V}_{x,\rho_i}$ are contained in $B(0,2\lVert p\rVert)$. This, together \cref{rem:Ch1}, implies that
\begin{equation}
    \int \ell(w)d\mathcal{C}^h\llcorner \mathbb{V}_{x,\rho_i}\leq \mu\lVert p\rVert\mathcal{C}^h\llcorner \mathbb{V}_{x,\rho_i}(B(q_i,(3+2\mu)\lVert p\rVert))\leq(3+2\mu)^{h+1}\lVert p\rVert^{h+1}.
    \label{eq:::num9}
\end{equation}
On the other hand, thanks to \cref{rem:ScalinfFxr} and the fact that $\mathcal{C}^h\llcorner \mathbb{V}$ and $\mathcal{C}^h\llcorner \mathbb{V}_{x,\rho_i}$ are invariant under rescaling, we infer that
\begin{equation}
F_{0,(1+\mu)\lVert p\rVert}(\lambda\mathcal{C}^h\llcorner \mathbb{V},\Theta_i\mathcal{C}^h\llcorner \mathbb{V}_{x,\rho_i})=\bigg(\frac{(1+\mu)\lVert p\rVert}{4\oldC{C:U}}\bigg)^{h+1}F_{0,4\oldC{C:U}}(\lambda\mathcal{C}^h\llcorner \mathbb{V},\Theta_i\mathcal{C}^h\llcorner \mathbb{V}_{x,\rho_i}).
\label{eq:::num10}
\end{equation}
Putting together \eqref{eq:::num8}, \eqref{eq:::num11}, \eqref{eq:::num12}, \eqref{eq:::num9} and \eqref{eq:::num10}  we finally infer that
\begin{equation}
\begin{split}
     \liminf_{i\to\infty}\,\lambda\int \ell(w) d\mathcal{C}^h\llcorner\mathbb{V}_{x,\rho_i}
     &\geq \lambda \int \ell(w)d\mathcal{C}^h\llcorner \mathbb{V}-2(\vartheta^{-1}\mu)^{(h+3)}(3+2\mu)^{h+1}\lVert p\rVert^{h+1}
     -\bigg(\frac{(1+\mu)\lVert p\rVert}{4\oldC{C:U}}\bigg)^{h+1}(\vartheta^{-1}\mu)^{(h+3)}\\
    &\geq\lambda \int \ell(w)d\mathcal{C}^h\llcorner \mathbb{V}-(\vartheta^{-1}\mu)^{(h+3)}\lVert p\rVert^{h+1}\big(2(3+2\mu)^{h+1}+1\big).
        \label{eq:::num13}
\end{split}
\end{equation}
Finally, \cref{prop:density} and the fact that $x\in E(\vartheta,\gamma)$ imply that $\lambda\geq \vartheta^{-1}$. This together with a simple computation that we omit, based on \cref{unif}, shows that
\begin{equation}
\lambda\int \ell(w) d\mathcal{C}^h\llcorner \mathbb{V}= \vartheta^{-1}(\mu\lVert p\rVert)^{h+1}/(h+1).
\label{eq:::num14}    
\end{equation}
Putting together \eqref{eq:::num13} and \eqref{eq:::num14} we eventually infer that
$$
\liminf_{i\to\infty}\lambda\int \ell(w) d\mathcal{C}^h\llcorner\mathbb{V}_{x,\rho_i}\geq \vartheta^{-1}(\mu^{h+1}/(h+1)-2^{2(h+2)}\mu^{(h+3)})\lVert p\rVert^{h+1}>0,
$$
proving that for any $p\in B(0,1)\cap \mathbb V$ we have $B(p,\mu\lVert p\rVert)\cap \mathbb{V}_{x,\rho_i}\neq \emptyset$ provided $i$ is chosen suitably big. Thus the claim (\hyperlink{HC}{HC}) is proved taking $p=y$.

Let us conclude the proof of the proposition exploiting the claim \eqref{eq:::num16} that we have proved. Defined $\mathcal{B}$ to be the set of full measure in $f(\hat A)$ on which \eqref{eq:::num16} holds, we note that since $B(P(x),r)\subseteq P(B(x,r))$, the \eqref{eq:::num16} implies the following one: for any $u\in P(\mathcal{B})$ there exists a $\mathfrak{r}(u)>0$ such that
\begin{equation}
\lVert f(u)^{-1}f(w)\rVert\leq \oldC{C:U}^{2}\lVert u^{-1}w\rVert, \qquad\text{whenever }w\in\hat A\cap B(u,\mathfrak{r}(u)).
\label{eq:::num16'}
\end{equation}
Furthermore, note that thanks to the proof of item (v) of \cref{prop::5.2} and recalling that $f(\hat A)\subseteq E\cap P^{-1}(A)$, we deduce that $\mathcal{S}^h\llcorner f(\hat{A})$ is mutually absolutely continuous with respect to $\phi$ and by \cref{cor:2.2.19} we finally infer that
$$
\mathcal{S}^h(\hat{A}\setminus P(\mathcal{B}))=\mathcal{S}^h(P(f(\hat{A})\setminus \mathcal{B}))=0,
$$
where the first equality above comes from the fact that $f:\hat{A}\to f(\hat{A})$ is bijective. 
 
We now prove that if $\mathfrak{r}(u)$ is chosen to be the biggest radius for which \eqref{eq:::num16'} holds, then the map $u\mapsto \mathfrak{r}(u)$ is upper semicontinuous on $\hat A$. Indeed, assume $\{u_i\}_{i\in\N}$ is a sequence in $\hat{A}$ such that $u_i\to u\in \hat A$ and $\limsup_{i\to\infty}\mathfrak{r}(u_i)=r_0\geq 0$. If $r_0=0$, then the inequality $\limsup_{i\to\infty}\mathfrak{r}(u_i)\leq \mathfrak{r}(u)$ is trivially satisfied. Thus, we can assume that $r_0>0$, and, without loss of generality, also that the $\limsup$ is actually a $\lim$. For any fixed $0<s<r_0$ there exists an $i_0\in\N$ such that
$$
s+d(u,u_i)<\mathfrak{r}(u_i)\qquad\text{for any }i\geq i_0.
$$
As a consequence $B(u,s)\subseteq B(u_i,\mathfrak{r}(u_i))$ and thus for any $y\in \hat A\cap B(u,s)$ and $i\geq i_0$ we have
\begin{equation}
\begin{split}
      \lVert f(u)^{-1}f(y)\rVert&\leq \lVert f(u)^{-1}f(u_i)\rVert+\lVert f(u_i)^{-1}f(y)\rVert\leq \oldC{C:U}^2\lVert u_i^{-1}u\rVert+\oldC{C:U}^2\lVert u_i^{-1}y\rVert.
     \label{eq:::num17}
\end{split}
\end{equation}
Sending $i$ to $+\infty$, thanks to \eqref{eq:::num17} we conclude that for any $y\in B(u,s)\cap \hat{A}$ we have $\lVert f(u)^{-1}f(y)\rVert\leq\oldC{C:U}^2\lVert u^{-1}y\rVert$ and thus $s\leq \mathfrak{r}(u)$. The arbitrariness of $s$ concludes that $\mathfrak{r}$ is upper semicontinuous and thus for any $j\in\N$ the sets
$$L_j:=\{w\in \hat{A}:\mathfrak{r}(w)\geq 1/j\},$$
are Borel. Furthermore, since $\mathfrak{r}(u)>0$ everywhere on $P(\mathcal{B})$, we infer that $P(\mathcal{B})\subseteq \cup_{j\in\mathbb N}L_j$. This, jointly with the fact that $\mathcal{S}^h(\hat A)>0$, and that $\mathcal{S}^h(\hat A\setminus P(\mathcal{B}))=0$ tells us that we can find a $j\in\N$ and compact subset $\mathcal{A}$ of $L_j$ such that $\mathcal{S}^h(\mathcal{A})>0$ and $\diam(\mathcal{A})<1/2j$. 

Let us conclude the proof by showing that $f$ is Lipschitz on $\mathcal{A}$ and that $\phi(f(\mathcal{A}))>0$. 
The fact that $f(\mathcal{A})$ is $\phi$-positive follows from \cref{lip:const:proj:conorm}, item (v) of \cref{prop::5.2} and the following computation
$$
0<\mathcal{S}^h(\mathcal{A})=\mathcal{S}^h(P(f(\mathcal{A})))\leq \oldC{C:proj}^h\mathcal{S}^h(f(\mathcal{A})).
$$
On the other hand, for any $u,v\in \mathcal{A}$ we have $d(u,v)\leq 1/2j$ and since $u,v\in L_j$ then
$$\lVert f(u)^{-1}f(v)\rVert\leq \oldC{C:U}^2\lVert u^{-1}v\rVert.$$
This eventually concludes the proof of the proposition.
\end{proof}

\begin{proof}[Proof of \cref{thm:MMconormale}]
If we prove the result for $\phi\llcorner B(0,k)$ for any $k\in\N$, the general case follows taking into account the locality of tangents \cref{prop:LocalityOfTangent} and Lebesgue differentiation theorem \cref{prop:Lebesuge}. Therefore, we can assume without loss of generality that $\phi$ is supported on a compact set.
Let us set 
\begin{equation}
    \begin{split}
        \mathscr F:=\{\cup_{i\in\mathbb N}f_i(K_i):K_i\,\,\text{is a compact subset of}\,\,\mathbb W_i
        \,\,\text{with}\,\,\mathbb W_i\in\G_\unlhd(h)\,\,\text{and}\,\,f_i:K_i\to\mathbb G\,\,\text{is Lipschitz}\}.
    \end{split}
\end{equation}
Let $m:=\inf_{F\in\mathscr F}\{\phi(\mathbb G\setminus F)\}$. We claim that if $m=0$ the proof of the proposition is concluded. Indeed, if $m=0$ we can take $F_n\in\mathscr F$ such that $\phi(\mathbb G\setminus F_n)<1/n$ and then $\phi(\mathbb G\setminus\cup_{n\in\mathbb N}F_n)=0$. Let us prove that $m=0$. Indeeed, if by contradiction $m>0$, we can take, as before, $F_n'\in\mathscr F$ such that $0<\phi(\mathbb G\setminus \cup_{n\in\mathbb N}F_n')\leq m$. Since $F':=\cup_{n\in\mathbb N}F'_n$ is Borel, we have, thanks to the locality of tangents \cref{prop:LocalityOfTangent} and Lebesgue differentiation theorem \cref{prop:Lebesuge}, that $\phi\llcorner F'$ is a $\mathscr{P}^{*,\unlhd}_h$-rectifiable measure with compact support. Thus we can apply \cref{prop:PHIK>0} to conclude that there exists $\mathbb W\in\G_\unlhd(h)$, $K$ a compact subset of $\mathbb W$ and a Lipschitz function $f:K\to \mathbb G$ such that $\phi\llcorner F'(f(K))>0$. Thus we get that $\phi(\mathbb G\setminus (f(K)\cup F'))<m$, that is a contradiction with the definition of $m$.

In order to prove the last part of the theorem, let us notice that, thanks to the locality of tangents in \cref{prop:LocalityOfTangent} and Lebesgue differentiation theorem in \cref{prop:Lebesuge}, we can reduce on $\phi\llcorner E(\vartheta,\gamma)$, thanks also to \cref{prop::E}. Moreover, taking into account that $\mathcal{S}^h\llcorner E(\vartheta,\gamma)$ is mutually absolutely continuous with respect to $\phi\llcorner E(\vartheta,\gamma)$, see \cref{prop:MutuallyEthetaGamma}, we can finally reduce to prove that $\mathcal{S}^h\llcorner f(K)$ is a $\mathscr{P}^c_h$-rectifiable measure whenever $K$ is a compact subset of $\mathbb W\in \G_\unlhd(h)$ and $f:K\to\mathbb G$ is a Lipschitz function. The fact that $\mathcal{S}^h\llcorner f(K)$ is a $\mathscr{P}^c_h$-rectifiable measure follows from the following claim: if $K$ is a compact subset of $\mathbb W\in \G_\unlhd(h)$ and $f:K\to\mathbb G$ is a Lipschitz function, then for $\mathcal{S}^h\llcorner f(K)$-almost every $x\in\mathbb G$ we have that there exists $\mathbb W(x)\in\G(h)$ such that the following convergence of measures holds
\begin{equation}\label{eqn:CONVECONVE}
r^{-h}T_{x,r}\mathcal{S}^h\llcorner f(K)\rightharpoonup \mathcal{S}^h\llcorner\mathbb W(x), \qquad \text{as r goes to $0$.}
\end{equation}
Let us finally sketch the proof of \eqref{eqn:CONVECONVE}. Since $\mathbb W\in \G_\unlhd(h)$, i.e., it admits a normal complementary subgroup, we get that $\mathbb W$ is a Carnot subgroup of $\mathbb G$, see \cite[Remark 2.1]{AM20}. Thus we can apply Pansu-Rademacher theorem to $f:K\subseteq\mathbb W\to\mathbb G$, see \cite[Theorem 3.4.11]{MagnaniPhD}, to obtain that $f$ is Pansu-differentiable $\mathcal{S}^h$-almost everywhere, with Pansu differential $df$, and the area formula holds, see \cite[Corollary 4.3.6]{MagnaniPhD}. The proof of \eqref{eqn:CONVECONVE} with $\mathbb W(x):=Df(x)(\mathbb W)$ for $\mathcal{S}^h\llcorner f(K)$-almost every $x$ is now just a routine task, building on \cite[Proposition 4.3.1 and Proposition 4.3.3]{MagnaniPhD}, and by using the area formula in \cite[Corollary 4.3.6]{MagnaniPhD}. We do not give all the details as the proof follows verbatim as in the argument contained in \cite[pages 716-717]{MatSerSC}, with the obvious substitutions taking into account that the authors in \cite{MatSerSC} only deal with Heisenberg groups $\mathbb H^n$ in the case $\mathbb W$ is horizontal.
\end{proof}

\printbibliography

\end{document}